\documentclass[EJP]{ejpecp} 

\SHORTTITLE{Spatial evolutionary games} 

\TITLE{Spatial evolutionary games \\ with small selection coefficients \thanks{Partially supported
by NSF grant DMS 1305997 and NIH grant R01-GM096190}} 

\AUTHORS{%
  Rick Durrett\footnote{Duke University
    \EMAIL{rtd@math.duke.edu}} }

\KEYWORDS{replicator equation; PDE; Lyapunov function ; cancer models } 

\AMSSUBJ{60K35 ; 92D15} 

\SUBMITTED{June 24, 2014} 
\ACCEPTED{December 27, 2014} 

\ARXIVID{1406.5876} 

\VOLUME{0}
\YEAR{2012}
\PAPERNUM{0}
\DOI{vVOL-PID}

\ABSTRACT{Here we will use results of Cox, Durrett, and Perkins \cite{CDP} for voter model perturbations to study spatial evolutionary games on $\ZZ^d$, $d\ge 3$ when the interaction kernel is finite range, symmetric, and has covariance matrix $\sigma^2I$. The games we consider have payoff matrices of the form ${\bf 1} + wG$ where ${\bf 1}$ is matrix of all 1's and $w$ is small and positive. 
Since our population size $N=\infty$, we call our selection small rather than weak which usually means $w =O(1/N)$. 

The key to studying these games is the fact that when the dynamics are suitably rescaled in space and time they convergence to solutions of a reaction diffusion equation (RDE). Inspired by work of Ohtsuki and Nowak \cite{ON06} and Tarnita et al \cite{TOAFN, TWN} we show that the reaction term is 
the replicator equation for a modified game matrix and the modifications of the game matrix depend on the interaction kernel only through the values of two or three simple probabilities for an associated coalescing random walk. 

Two strategy games lead to an RDE with a cubic nonlinearity, so we can describe the phase diagram completely. Three strategy games lead to a pair of coupled RDE, but using an idea from our earlier work \cite{memoir}, we are able to show that if there is a repelling function for the replicator equation for the modified game, then there is coexistence in the spatial game when selection is small. This enables us to prove coexistence in the spatial model in a wide variety of examples including the behavior of four evolutionary games that have recently been used in cancer modeling.}

\newcommand\beq{\begin{equation}}
\newcommand\eeq{\end{equation}}
\newcommand\ep{\varepsilon}
\newcommand\hbr{\hfill\break}
\newcommand\mn{\medskip\noindent}
\newcommand\ms{\medskip}
\newcommand\nass{\noalign{\smallskip}}
\newcommand\RR{\mathbb{R}}
\newcommand\ZZ{\mathbb{Z}}
\newcommand\sqz{\kern -0.2em}

\begin{document}



\tableofcontents

\clearpage

\section{Introduction}

Game theory was invented by John von Neumann and Oscar Morgenstern \cite{vNM} to study strategic and economic decisions of humans. Maynard Smith and Price \cite{MSP73}, introduced the concept into ecology in order to explain why conflicts over territory between male animals of the same species are usually of the ``limited war'' type and do not cause serious damage. As formulated in John Maynard Smith's classic book, \cite{MS82}, this is the

\begin{example} \label{HD}
{\bf Hawks-Doves game.}
When a confrontation occurs, Hawks escalate
and continue until injured or until an opponent retreats, while Doves make a display of force but retreat at once
if the opponent escalates. The payoff matrix is
\begin{center}
\begin{tabular}{ccc}
& Hawk & Dove \\
Hawk & $(V-C)/2$ & $V$ \\
Dove & $0$ & $V/2$
\end{tabular}
\end{center}
Here $V$ is the value of the resource, which two doves split, and $C$ is the cost of competition. 
There are two conventions in the literature for the interpretation of the payoff matrix. Here we are using the one in which the first row gives the payoffs for the Hawk strategy based on the opponent's strategy choice. 

If we suppose that $C>V$ then a frequency $p=V/C$ of Hawks in the population represents an equilibrium, since in this case each strategy has the same payoff.
The equilibrium is stable. If the frequency of the Hawk strategy rises to $p>V/C$ then the Hawk strategy has a worse payoff than the Dove strategy so its frequency will decrease. 
\end{example}

Axelrod and Hamilton \cite{AH81}, see also \cite{Ax84}, studied the evolution of cooperation by investigating 

\begin{example} \label{PD}
{\bf Prisoner's dilemma.} Rather than formulate this in terms of prisoners deciding whether to confess or not, we consider the following game between cooperators $C$ and defectors $D$. Here $c$ is the cost
that cooperators pay to provide a benefit $b$ to the other player.  
\begin{center}
\begin{tabular}{ccc}
& C & D \\
C & $b-c$ & $-c$ \\
D & $b$ & $0$
\end{tabular}
\end{center}
If $b>c>0$ then the defection dominates cooperation, and, as we will see, altruistic cooperators are destined to die out in a homogeneously mixing population. This is unfortunate since the $(D,D)$ payoff is worse than $(C,C)$ payoff. 
Nowak and May \cite{NoMa92,NoMa93} found an escape from this dilemma by showing that spatial structure allowed persistence of cooperators. 
Huberman and Glance \cite{HG93} suggested that this was an artifact of the synchronous deterministic update rule, but later work \cite{NBM94,NBM94b} showed that the same phenomena occurred for asynchronous stochastic updates. 
\end{example}

There are many examples of $2\times 2$ games with names like Harmony, the Snowdrift Game, and the Battle of the Sexes, see e.g., \cite{Ha01}, but as we will see in Section \ref{sec:2sg} there are really only 3 (or 4) examples. To complete the set of possibilities, we now consider

\begin{example} \label{StagHunt}
{\bf Stag Hunt.} The story of this game was briefly described by Rousseau in his 1755 book {\it A Discourse on Inequality}.  If two hunters cooperate to hunt stag (an adult male deer) then they will bring home a lot of food, but there is practically no chance of bagging a stag by oneself. If both hunters go after rabbit they split what they kill. An example of a game of this type is:
\begin{center}
\begin{tabular}{ccc}
& Stag & Hare \\
Stag & 3 & 0 \\
Hare & 2 & 1
\end{tabular}
\end{center}
In this case if the two strategies have frequency $p=1/2$ in the population, then the two strategies have equal payoffs. If the frequency of the stag strategy rises to $p>1/2$ then it has the better payoff and will continue to increase so this is an unstable equilibrium.
\end{example}

In the forty years since the pioneering work of Maynard Smith and Price
 \cite{MSP73} evolutionary game theory has been used to study many biological problems. For surveys see \cite{HarSel}--\cite{Nowak06}. This is natural because evolutionary game theory provides
a general framework for the study of frequency dependent selection. In the references we have listed representative samples of work of this type, \cite{DL94}--\cite{NL13}. There are literally hundreds of references, so we have restricted our attention to those that are the most relevant to our investigations.

In recent years, evolutionary game theory has been applied to study cancer. This provides an important motivation for our work, so we will now consider four examples beginning with one first studied by Tomlinson \cite{Tom}.

\begin{example} \label{Tom}
{\bf Three species chemical competition.}
In this system, there are cells of three types.

\medskip
1. Ones that Produce a toxic substance, and are sensitive to toxins produced by other cells. 

\medskip
2. Others that are Resistant to the toxin, but do not produce it.

\medskip
3. Wild type cells that are Sensitive to the toxin but do not produce it.

\mn
Based on the verbal description, Tomlinson wrote down the following game matrix.
$$
\begin{matrix} 
 & P & R & S \\
P & z-e-f+g & z-e & z-e+g \\
R & z-h & z-h & z-h \\
S & z-f & z  & z 
\end{matrix}
$$
Here, and in what follows it is sometimes convenient to number the strategies $1=P$, $2=R$ and $3=S$. For example, this makes it easier to say that $G_{ij}$ is the payoff to a player who plays strategy $i$ against an opponent playing strategy $j$. Taking the rows in reverse order, $z$ is the baseline fitness while $f$ is the cost
to a sensitive cell due to the presence of the toxin. The cost of resistance to the toxin 
is $h$. In top row $e$ is the cost of producing the toxin, and $g$ is advantage to a producer when it interacts with a sensitive cell. 
\end{example}

It is interesting to note that in the same year \cite{Tom} was published, Durrett and Levin \cite{DL97} used a spatial model to model the competition two strains of {\it E. coli}, whose behaviors correspond to strategies $P$, $R$, and $S$ above. In their model, there are also empty cells (denoted by 0). Thinking of a petri dish, their system takes place on the two dimensional lattice with the following dynamics. 

\begin{itemize}
  \item 
Individuals of type $i>0$ give birth at rate $\beta_i$ with their offspring sent to a site chosen at random from the four nearest neighbors of $x$. If the site is occupied nothing happens. 
  \item
Each species dies at rate $\delta_i$ due to natural causes, while type 3's die at an additional rate $\gamma$ times the fraction of neighbors of type 1 due to the effect of colicin.
\end{itemize}

\begin{example} \label{Glyco}
{\bf Glycolytic phenotype.}
It has been known for some time that cancer cells commonly switch to glycolysis for energy production. 
Glycolysis is less efficient than the citrate cycle in terms of energy, but allows cell survival 
in hypoxic environments. In addition, cells with glycolytic metabolism can change the pH of the
local microenvironment to their advantage. 

The prevalence of glycolytic cells in invasive tumor suggests that their presence could benefit the emergence of
invasive phenotypes. To investigate this using evolutionary game theory, Basanta et al \cite{BSHD} considered a three strategy
game in which cells are initially characterized as having autonomous growth ($AG$), but could switch to
glycolysis for energy production ($GLY$), or become increasing mobile and invasive ($INV$). 
The payoff matrix for this game, which is the transpose of the one in their Table 1, is:

\begin{center}
\begin{tabular}{lccc}
& $AG$ & $INV$ & $GLY$ \\
$AG$  & $\frac{1}{2}$ & 1 & $\frac{1}{2} - n$\\
\nass
$INV$ & $1-c$ & $1-\frac{c}{2}$ & $1-c$ \\
\nass
$GLY$ & $\frac{1}{2} + n - k$ & $1-k$ & $\frac{1}{2} - k$
\end{tabular}
\end{center}

\noindent
Here $c$ is the cost of mobility, $k$ is the cost to switch to glycolysis, and
$n$ is the detriment for nonglycolytic cell in glycolytic environment, which is equal to the bonus for a glycolytic cell.
Recently this system has been extended to a four player game by adding an invasive-glycolytic strategy, see \cite{BasIDH1}.
\end{example}

\begin{example} \label{Stroma}
{\bf Tumor-Stroma Interactions.} 
Tumors are made up of a mixed population of different types of cells that include normal structures as well as ones associated with malignancy, and there are multiple interactions between the malignant cells and the local microenvironment. These intercellular interactions effect tumor progression. In prostate cancer it has been observed that one can have three different outcomes: 

\mn
(i) The tumor remains well differentiated and relatively benign. In this case the local stromal cells 
(also called connective tissue) may serve to restrain the growth of the cancer.

\mn
(ii) Early in its genesis the tumor acquires a highly malignant phenotype, growing 
rapidly and displacing the stromal population (often called small cell prostate cancer).

\mn
(iii) The tumor co-opts the local stroma to aid in its growth.

\mn
To understand the origin of these behaviors, Basanta et al \cite{BSFAHA} formulated a game with
three types of players $S =$ stromal cells, $I =$ cells that have become independent of the micro-environment, and $D = $ cells that remain dependent on the microenvironmnet.
The payoff matrix is:

\begin{center}
\begin{tabular}{lccc}
& $S$ & $D$ & $I$ \\
$S$ & 0 & $\alpha$ & 0 \\
$D$ & $1+\alpha-\beta$ & $1-2\beta$ & $1-\beta + \rho$ \\
\nass
$I$ & $1 - \gamma$  & $1 - \gamma$ & $1 - \gamma$
\end{tabular}
\end{center}

\mn
Again this is the transpose of their matrix.
Here $\gamma$ is the cost of being environmentally independent,
$\beta$ cost of extracting resources from the micro-environment,
$\alpha$ is the benefit derived from cooperation between $S$ and $D$,
and $\rho$ benefit to $D$ from paracrine growth factors produced by $I$. 
\end{example}

\begin{example} \label{bone}
{\bf Multiple Myeloma.}
Normal bone remodeling is a consequence of a dynamic balance between osteoclast ($OC$) mediated bone resorption
and bone formation due to osteoblast ($OB$) activity.
Multiple myeloma ($MM$) cells disrupt this balance in two ways. 

\mn
(i) $MM$ cells produce a variety of cytokines that stimulate the growth of the $OC$ population.

\mn
(ii) Secretion of $DKK1$ by $MM$ cells inhibits  $OB$ activity.

\mn
$OC$ cells produce osteoclast activating factors that stimulate the growth of $MM$ cells
where as $MM$ cells are not effected by the presence of $OB$ cells.
These considerations lead to the following game matrix. Here, $a,b,c,d,e >0$.

\begin{center}
\begin{tabular}{lccc}
& $OC$ & $OB$ & $MM$ \\
$OC$  & 0 & $a$ & $b$\\
$OB$ & $e$ & $0$ & $-d$ \\
$MM$ & $c$ & $0$ & $0$
\end{tabular}
\end{center}

\mn
For more on the biology, see Dingli et al.~\cite{DCSSP}.

\end{example}

A number of other examples have been studied. Tomlinson and Bodmer  \cite{TomBod} studied games motivated by angiogenesis
and apototsis. See Basanta and Deutsch \cite{BasDeu} for a survey of this and other early work.
Swierniak and Krzeslak's survey  \cite{SwKrz} contains the four examples we have covered here, as well as
a number of others.

\section{Overview}

In a homogeneously mixing population, $x_i =$ the frequency of players using strategy $i$ follows the replicator equation
$$
\frac{dx_i}{dt} = x_i(F_i - \bar F) 
$$
where $F_i = \sum_j G_{i,j} x_j$ is the fitness of strategy $i$ and $\bar F = \sum_i x_i F_i$ is the average fitness. 

Twenty years ago, Durrett and Levin \cite{DL94} studied evolutionary games and formulated rules for predicting the behavior
of spatial models from the properties of the mean-field differential equations obtained by supposing that adjacent sites 
are independent. See \cite{RDWald} for an overview of this approach. Our main motivation for revisiting this question
is that the recent work of Cox, Durrett, and Perkins \cite{CDP} allows us to turn the heuristic principles
of \cite{DL94} into rigorous results for evolutionary games with  matrices of the form $\bar G = {\bf 1} + w G$, where $w>0$ is small,
and ${\bf 1}$ is a matrix that consists of all 1's. 

We will study these games on $\ZZ^d$ where $d \ge 3$ and the interactions between an individual and its neighbors are given by an irreducible probability
kernel $p(x)$ on $\ZZ^d$ with $p(0)=0$ that is finite range, symmetric $p(x)=p(-x)$, and has covariance matrix $\sigma^2 I$. To describe the dynamics we let
$\xi_t(x)$ be the strategy used by the individual at $x$ at time $t$ and  
$$
\psi_t(x) = \sum_y \bar G(\xi_t(x),\xi_t(y)) p(y-x)
$$
be the fitness of $x$ at time $t$. In Birth-Death dynamics, site $x$ gives birth at rate $\psi_t(x)$ and sends its offspring to replace the individual at $y$
with probability $p(y-x)$. In Death-Birth dynamics, the individual at $x$ dies at rate 1, and replaces itself by a copy of the one at $y$
with probability proportional to $p(y-x)\psi_t(y)$.

When $w=0$ either version of the dynamics reduces to the voter model, a system in which each site at rate 1 changes its state to that of a randomly chosen neighbor. As we will explain in Section \ref{sec:vmp}, when $w$ is small, our spatial evolutionary game is a voter model perturbation in the sense of Cox, Durrett, and Perkins \cite{CDP}. Section \ref{sec:vmsi} describes the duality between the voter model and coalescing random walk, which is the key to study of the voter model and its perturbations. More details are given in Section \ref{sec:vmd}. Section \ref{sec:vmsi} also introduces the coalescence probabilities that are the key to our analysis and states some identities between them that are proved in Section \ref{sec:coalid}.

The next step is to show that when $w \to 0$ and space and time are rescaled appropriately then frequencies of strategies in the evolutionary game converge to a limiting PDE. Section \ref{sec:pdelim} states the results, while Section \ref{sec:derPDE} shows how these conclusions follow from results in \cite{CDP}.
While this is a ``known result,'' there are two important new features here. The reaction term is identified as the replicator equation for a modified game, and it is shown that the modifications of the game matrix depend only on the effective number of neighbors $\kappa$ (defined in \eqref{enn} below) and two probabilities for the coalescing random walk with jump kernel $p$. The first fact allows us to make use of the theory that has been developed for replicator equations, see \cite{HS98}. 
 
Two strategy games are studied in Section \ref{sec:2sg}. A complete description of the phase diagram for Death-Birth and Birth-Death dynamics is possible because the limiting RDE's have cubic reaction terms $f(u)$ with $f(0)=f(1)=0$, so we can make use of work of \cite{AW75,AW78,FM77,FM81}. As a corollary of this analysis, we are able to prove that Tarnita's formula for two strategy games with weak selection holds in our setting. They say that a strategy in a $k$ strategy game is ``favored by selection'' if its frequency in equilibrium is $> 1/k$ when $w$ is small. Tarnita et al \cite{TOAFN} showed that this holds for strategy 1 in a 2 by 2 game if and only if
$$
\sigma G_{1,1} + G_{1,2} > G_{2,1} + \sigma G_{2,2}
$$
where $\sigma$ is a constant that depends only on the dynamics. In our setting $\sigma=1$ for Birth-Death dynamics while $\sigma=(\kappa+1)/(\kappa-1)$ for Death-Birth dynamics where
\beq
\kappa = 1\left/ \sum_{x} p(x)p(-x) \right.
\label{enn}
\eeq
is the ``effective number of neighbors.'' To explain the name note that if $p$ is uniform on a symmetric set $S$ of size $m$, (i.e.,  $x\in S$ implies $-x\in S$), and $0 \neq S$ then $\kappa=m$.

In Section \ref{sec:3sgode} we begin our study of three strategy games by analyzing the behavior of their replicator equations.
This is done using the invadability analysis developed in \cite{memoir}. The first step is to analyze the three $2 \times 2$ subgames. In $1,2$ subgame there are
four possibilities: 

\begin{itemize}
  \item  strategy 1 dominates 2, $1 \gg 2$; 
  \item strategy 2 dominates 1, $2 \gg 1$; 
  \item there is a stable mixed strategy equilibrium which is the limit starting from any point on the interior of the edge; 
  \item there is an unstable mixed strategy which separates the sets of points that converge to the two endpoints. 
\end{itemize} 
In the case that there is a mixed strategy equilibrium, we also have to see if it can be invaded by the strategy not in the subgame, i.e., its frequency will increase when rare.

Our method for analyzing the spatial game is to prove the existence of a repelling function (i.e., a convex Lyapunov function that blows up near the boundary and satisfies some mild technical conditions, see  Section \ref{ssec:repel} for a precise definition). Unfortunately, a repelling function cannot exist when there is an unstable fixed point on some edge, but this leaves a large number of possibilities. In Sections \ref{ssec:blemmas} and  \ref{ssec:Lexist}, we prove that they exist for three of the four classes of examples identified in Section \ref{sec:3sgode} that have attracting interior fixed points. 

To build repelling functions, we construct for $i=1,2,3$ a function $h_i$ that blows up on $\{ u_i=0 \}$, and is decreasing along trajectories near the side $u_i=0$. We then pick $M_i$ large enough so that $\bar h_i = \max\{h_i,M_i\}$ is always decreasing along trajectories and define 
$$
\psi = \bar h_1 + \bar h_2 + \bar h_3
$$
This repelling function is often easy to construct and allows us to prove coexistence of the three types, but it does not allow us to obtain very precise information about the frequencies of the
three types in the equilibrium. An exception is the set of games considered in Section \ref{ssec:acs} that are almost constant sum. In that
case we have a repelling function that is decreasing everywhere except at the interior fixed point $\rho$, so Theorem 1.4 of \cite{CDP} allows us to 
conclude that when $w$ is small any nontrivial stationary distribution for the spatial model has frequencies near $\rho$.

In Section \ref{sec:cancer}, we turn our attention to the three strategy examples from the introduction, Examples \ref{Tom}--\ref{bone}. Using the results from the Sections \ref{sec:3sgode} and \ref{sec:3sgsp}, we are able to analyze these games in considerable detail. While we are able to treat a number of examples, there are also some large gaps in our knowledge. 

\begin{itemize}
\item Consider a three strategy game with payoff matrix
$$
\begin{matrix}
 0 & \alpha_3 & \beta_2 \\
\beta_3 & 0 & \alpha_1 \\
\alpha_2 & \beta_1 & 0
\end{matrix}
$$
where $\alpha_i < 0 < \beta_i$ so that the strategies have a rock-paper-scissors cyclic dominance relationship. Find conditions for coexistence
of the three strategies in the spatial game. For the replicator equation the condition is $\beta_1\beta_2\beta_3 + \alpha_1\alpha_2\alpha_3>0$.

\item
Perhaps the most important open problem is that we cannot 
handle the case in which the replicator equation is bistable, i.e., there are two locally attracting fixed points $u^*$ and $v^*$. See Examples 7.2B and 7.3B.
By analogy with the work of Durrett and Levin \cite{DL94}, we expect that the winner in the spatial game is dictated by the direction of
movement of the traveling wave solution that tends to $u^*$ at $-\infty$ and $v^*$ at $\infty$. However, we are not aware of results
that prove the existence of traveling waves solutions, or methods to prove convergence of solutions to them.

\item
A less exciting open problem is prove ``dying out results'' which show that one strategy takes over the system (see Examples 7.1A and 7.3A) or that one strategy dies out leaving an equilibrium that consists of a mixture of the other two strategies (see Examples 7.2A, 7.3C and 7.3D).  Given results in Section \ref{sec:3sgsp}, it is natural to start by finding suitable convex Lyapunov functions. However, (i) in contrast to the repelling functions used to prove coexistence, these functions must be decreasing everywhere except at the point which is the limiting value of the replicator equation, and (ii) auxillary arguments such as the ones used in Chapter 7 of \cite{CDP} are needed to prove that the limiting frequencies are 0 rather than just small.

\end{itemize}

\section{Voter model perturbations} \label{sec:vmp}

The state of the system at time $t$ is $\xi_t : \ZZ^d \to S$ where $S$ is a finite set of sates that
are called strategies or opinions. $\xi_t(x)$ gives the state of the individual at $x$ at time $t$.
To formulate this class of models, 
let $f_i(x,\xi) = \sum_y p(y-x) 1(\xi(y)=i)$ be the fraction of neighbors of $x$ in state $i$.
In the voter model, the rate at which the voter at $x$ changes its opinion from $i$ to $j$ is
$$
c^v_{i,j}(x,\xi) = 1_{(\xi(x)=i)} f_j(x,\xi) 
$$

The voter model perturbations that we consider have flip rates 
\beq
c^v_{i,j}(x,\xi) + \ep^2 h^\ep_{i,j}(x,\xi) 
\label{frates}
\eeq
The perturbation functions $h^\ep_{ij}$ may be negative (and will be for games with negative entries) but 
in order for the analysis in \cite{CDP} to work, there must be a 
law $q$ of $(Y^1, \ldots Y^{m}) \in (\ZZ^d)^{m}$ and functions $g_{i,j}^\ep \ge 0$, which 
converge to limits $g_{i,j}$ as $\ep\to 0$, so that for some $\gamma < \infty$, we have for $\ep\le \ep_0$
\beq
h^\ep_{i,j}(x,\xi) = - \gamma f_i(x,\xi) + E_{Y}[g_{i,j}^\ep(\xi(x+Y^1), \ldots \xi(x+Y^{m}))]
\label{vptech}
\eeq
In words, we can make the perturbation positive by adding a positive multiple of the voter flip rates.
This is needed so that \cite{CDP} can use $g^\ep_{i,j}$ to define jump rates of a Markov process.

We have assumed that $p$ is finite range. As \eqref{hBD} and \eqref{hDB} will show, $q$ is a pointmass
on the vector of sites that can be reached from 0 by two steps taken according to $p$. Applying
Proposition 1.1 of \cite{CDP} now implies the existence of a suitable $g^\ep_{i,j}$
and that all our calculations can be done using the original perturbation. However, to use Theorems 1.4 and 1.5 in \cite{CDP}
we need to suppose that 
$$
h_{i,j} = \lim_{\ep\to 0} h^\ep_{i,j}.
$$
has $|h(i,j) - h^\ep(i,j)| \le C \ep^r$ for some $r>0$, see (1.41) in \cite{CDP}. We will study two update rules. In the first the perturbation is
independent of $\ep$. In the second the last condition holds with $r=2$.

Let $\xi^\ep_t$ be the process with flip rates given in \eqref{frates}.
The next result is the key to the analysis of voter model perturbation.  
Intuitively it says that if we rescale space to $\ep\ZZ^d$ and speed up time by $\ep^{-2}$ 
the process converges to the solution of a partial differential equation. The first thing we have to do is to define
the mode of convergence. Given $r\in(0,1)$, let $a_\ep = \lceil \ep^{r-1} \rceil \ep$, $Q_\ep = [0, a_\ep)^d \cap \ep \ZZ^d$, and
$|Q_\ep|$ the number of points in $Q_\ep$. For $x \in a_\ep \ZZ^d$ and $\xi\in \Omega_\ep$ the space of all functions
from $\ep\ZZ^d$ to $S$ let
$$
D_i(x,\xi) = |\{ y \in Q_\ep : \xi(x+y) = i \}|/|Q_\ep|
$$ 

We endow $\Omega_\ep$ with the $\sigma$-field ${\cal F}_\ep$ generated by the finite-dimensional distributions. Given a sequence of measures
$\lambda_\ep$ on $(\Omega_\ep,{\cal F}_\ep)$ and continuous functions $v_i$, we say that $\lambda_\ep$ has asymptotic densities
$v_i$ if for all $0 < \delta, R < \infty$ and all $i\in S$
$$
\lim_{\ep\to 0} \sup_{x\in a_\ep\ZZ^d, |x| \le R} \lambda_\ep ( |D_i(x,\xi) - v_i| > \delta ) \to 0
$$

\begin{theorem} \label{hydro}
Suppose $d \ge 3$.
Let $v_i: \RR^d \to [0,1]$ be continuous with $\sum_{i\in S} v_i = 1$. Suppose the initial conditions $\xi^\ep_0$
have laws $\lambda_\ep$ with local density $v_i$ and let  
$$
u^\ep_i(t,x) = P( \xi^\ep_{t\ep^{-2}}(x) = i)
$$ 
If $x_\ep \to x$ then $u^\ep_i(t,x_\ep) \to u(t,x)$ the solution of the system of partial differential equations:
$$
\frac{\partial}{\partial t} u_i(t,x) = \frac{\sigma^2}{2} \Delta u_i(t,x) +  \phi_i(u(t,x))
$$
with initial condition $u_i(0,x) = v_i(x)$. The reaction term 
$$
\phi_i(u) = \sum_{j \neq i} \langle 1_{(\xi(0)=j)} h_{j,i}(0,\xi) -  1_{(\xi(0)=i)} h_{i,j}(0,\xi) \rangle_u
$$
where the brackets are expected value with respect to the voter model stationary distribution $\nu_u$
in which the densities are given by the vector $u$.
\end{theorem}

\noindent
Intuitively, since on the fast time scale the voter model runs at rate $\ep^{-2}$ versus the
perturbation at rate 1, the process is always close to the voter equilibrium for 
the current density vector $u$. Thus, we can compute the rate of change of $u_i$ by assuming the nearby
sites are in that voter model equilibrium. The restriction to dimensions $d \ge 3$ is due to the
fact that the voter model does not have nontrivial stationary distributions in $d\le 2$.
For readers not familiar with the voter model, we recall the relevant facts in Section \ref{sec:vmd}.

Theorem 1.2 in \cite{CDP} gives more information. It shows that the joint distribution 
of the values at $x_\ep + \ep y_i$ converges to those of the voter model with densities $u_i(t,x)$.
In addition, it shows that if we consider $k$ points $x^k_\ep$ with $\ep^{-1}|x^j_\ep - x^k_\ep| \to \infty$
then the finite dimensional distributions near these points are asymptotically independent. Once this is done, see Theorem 1.3 in \cite{CDP}, it is easy to
conclude that the local densities at time $t$, $D_i(x,\xi_t)$ converge to $u_i(t,x)$ in $L^2$ uniformly on compact sets.
This result makes precise the sense in which the rescaled particle system converges to the solution of a PDE.

The results we have just stated are proved in \cite{CDP} only for the case $S = \{0,1\}$. The proof is almost 30 pages
but it is easy to see that it extends easily to a general finite set $S$. The proof begins by constructing the process with flip rates \eqref{frates} on $\ZZ^d$
using Poisson processes to generate the changes at each site. This enables us to construct for each $x$ and $t$ a dual process $\zeta^{\ep,x,t}_s$
which is what Durrett and Neuhauser call the ``influence set.'' If we know the values of $\xi^\ep_{t-s}$ on $\zeta^{\ep,x,t}_s$
then we can compute $\xi^\ep_t(x)$. In the dual particles undergo random walks at rate 1, coalescing when they hit, and at rate $\ep^2$
we have to add points $z+Y^1, \ldots z+Y^m$ when an event in the perturbing process occurs at a point $z$ currently in the dual. 
Only the rates at which various sets are added is relevant, not how they are used to compute the change of state in the process,
so the convergence of the dual to branching Brownian motion is the same as before, and that result leads easily to the
conclusions stated above.

\mn
{\bf Update rules.} We will consider two versions of spatial evolutionary game dynamics.

\mn
{\bf Birth-Death dynamics.}
In this version of the model, a site $x$ gives birth at a rate equal to its fitness
$$
\sum_{y} p(y-x) \bar G(\xi(x),\xi(y))
$$
and the offspring replaces a ``randomly chosen neighbor of $x$.'' Here and in what follows, the phrase in quotes
means $z$ is chosen with probability $p(z-x)$. If we let 
$r_{i,j}(0,\xi)$ be the rate at which the state of $0$ flips from $i$ to $j$, then setting $w=\ep^2$
and using symmetry $p(x)=p(-x)$, we get
\begin{align}
r_{i,j}(0,\xi) & = \sum_{x} p(x) 1(\xi(x)=j) \sum_{y} p(y-x) \bar G(j,\xi(y))
\nonumber \\
 & = \sum_{x} p(x) 1(\xi(x)=j) \left( 1 + \ep^2 \sum_{k} f_k(x,\xi) G_{j,k} \right) 
\nonumber \\
& = f_j(0,\xi) + \ep^2 \sum_k f^{(2)}_{j,k}(0,\xi) G_{j,k},
\label{rijBD}
\end{align}
where $f^{(2)}_{j,k}(0,\xi) = \sum_x \sum_y p(x) p(y-x) 1(\xi(x)=j,\xi(y)=k)$, so the perturbation,
which does not depend on $\ep$ is
\beq
h_{i,j}(0,\xi) = \sum_{k} f^{(2)}_{j,k}(0,\xi) G_{j,k}
\label{hBD}
\eeq

\mn
{\bf Death-Birth Dynamics.}
In this case, each site dies at rate one and is replaced by a neighbor chosen with probability proportional to its fitness. 
Using the notation in \eqref{rijBD} the rate at which $\xi(0)=i$ jumps to state $j$ is
\begin{align}
\bar r_{i,j}(0,\xi) & = \frac{r_{i,j}(0,\xi)}{\sum_k r_{i,k}(0,\xi)}  
= \frac{f_j(0,\xi) + \ep^2 h_{i,j}(0,\xi)}{ 1 + \ep^2 \sum_k h_{i,k}(0,\xi) } 
\nonumber\\
& = f_j + \ep^2 h_{i,j}(0,\xi) - \ep^2 f_j \sum_k h_{i,k}(0,\xi) + O(\ep^4)
\label{rijDB}
\end{align}
The new perturbation, which depends on $\ep$, is
\beq
\bar h_{i,j}^\ep(0,\xi) = h_{i,j}(0,\xi) -f_j \sum_k h_{i,k}(0,\xi) +O(\ep^2)
\label{hDB}
\eeq  
It is not hard to see that it also satisfies the technical condition \eqref{vptech}.

\medskip
There are a number of other update rules. In {\bf Fermi updating}, a site $x$ and a neighbor $y$
are chosen at random. Then $x$ adopts $y$'s strategy with probability
$$
\left[ 1 + \exp( \beta(F_x - F_y ))\right]^{-1}
$$
where $F_z = \sum_w p(w-z) G(\xi(z),\xi(w))$. The main reason for interest in this rule is the Ising model like phase transition
that occurs, for example in Prisoner's Dilemma games
as $\beta$ is increased, see \cite{SzTo98, HaSz05, SzVuSz05}.

In {\bf Imitate the best} one adopts the strategy of the neighbor with the largest fitness.
All of the matrices $\bar G = 1 + w G$ have the same dynamics, so this is
not a voter model perturbation. In discrete time (i.e., with synchronous updates) the
process is deterministic. See Evilsizor and Lanchier \cite{EvLa} for recent work on this version.

\section{Voter model duality} \label{sec:vmsi}

In this section we set $w=0$, so $\bar G = {\bf 1}$ and the system becomes the voter model. 
Let $\xi_t(x)$ be the state of the voter at $x$ at time $t$.
The key to the study of the voter model is that
we can define for each $x$ and $t$, random walks $\zeta^{x,t}_s$, $0 \le s \le t$ that move independently 
until they hit and then coalesce to one walk, so that 
\beq
\xi_t(x) = \xi_{t-s}(\zeta^{x,t}_s)
\label{dualeq}
\eeq
Intuitively, the $\zeta^{x,t}_s$ are genealogies that trace the origin of the opinion at $x$ at time $t$. See Section \ref{sec:vmd}
for more details about this and other facts about the voter model we cite in this section.

Consider now the case of two opinions.
A consequence of this duality relation is that if we let $p(0|x)$ be the probability that
two continuous time random walks with jump distribution $p$, one starting at the origin 0, and one starting at $x$ never hit then
$$
\langle \xi(0) = 1, \xi(x) = 0 \rangle_u = p(0|x) u(1-u)
$$
To prove this, we recall that the stationary distribution $\nu_u$ is the limit in distribution as $t\to\infty$ of $\xi^u_t$, the
voter model starting with sites that are independent and $=1$ with probability $u$, and then observe that \eqref{dualeq} implies
$$
P( \xi^u_t(0) = 1, \xi^u_t(x) = 0 ) = P( \xi^u_0(\zeta^{0,t}_t) = 1, \xi^u_0(\zeta^{x,t}_t) =  0 ) = 
u(1-u) P( \zeta^{0,t}_t  \neq \zeta^{x,t}_t )
$$
Letting $t\to\infty$ gives the desired identity.

To extend this reasoning to three sites, let $p(0|x|y)$ be the probability that the three random walks never
hit and let $p(0|x,y)$ be the probability that the walks starting from $x$ and $y$ coalesce, but
they do not hit the one starting at 0. Considering the possibilities that the walks starting from $x$ and  $y$
may or may not coalesce:  
$$
\langle \xi(0) = 1, \xi(x) = 0 , \xi(y) = 0 \rangle_u  = p(0|x|y) u(1-u)^2  + p(0|x,y) u(1-u)
$$  

Let $v_1, v_2, v_3$ be independent and chosen according to the distribution $p$ and let  $\kappa = 1/P( v_1+v_2 = 0 )$ be the ``effective number of neighbors''
defined in \eqref{enn}.  The coalescence probabilities satisfy some remarkable identities
that will be useful for simplifying formulas later on.  
Since the $v_i$ have the same distribution as steps in the random walk, simple arguments given in Section \ref{sec:coalid} show that
\begin{align}
&p(0|v_1) = p(0|v_1+v_2) = p(v_1|v_2) 
\label{Ted2}\\
&p(v_1|v_2+v_3) = ( 1 + 1/\kappa) p(0|v_1) 
\label{Ted3}
\end{align}
Here $p(0|v_1) = \sum_x p(0|x)P( v_1=x)$, $p(0|v_1+v_2) = \sum_{x,y} p(0|x+y) P(v_1=x)P(v_2=y)$, etc. 

It is easy to see that for any $x,y,z$ coalescence probabilities must satisfy
\beq
p(x|z) = p(x,y|z) + p(x|y,z) + p(x|y|z)
\label{3to2}
\eeq
Combining this with the identities for \eqref{Ted2}, \eqref{Ted3} leads to
\begin{align}
&p(0,v_1|v_1+v_2) = p(0,v_1+v_2|v_1) =  p(v_1,v_1+v_2|0) 
\label{BDswitch}\\
& p(v_1,v_2|v_2+v_3) = p(v_2,v_2+v_3|v_1) = p(v_1,v_2+v_3|v_2) + (1/\kappa)p(0|v_1)
\label{DBswitch}
\end{align}
All of the identities stated here are proved in Section \ref{sec:coalid}.
From \eqref{3to2} and \eqref{BDswitch} it follows that
\beq
p(0|v_1) = 2p(x,y|z) + p(0|v_1|v_1+v_2)
\label{BD1}
\eeq
where $x,y,z$ is any ordering of $0, v_1, v_1+v_2$. 
Later, we will be interested in $p_1=p(0|v_1|v_1+v_2)$ and $p_2 =p(0|v_1,v_1+v_2)$. In this case,
\eqref{BD1} implies
\beq
2p_2 + p_1 = p(0|v_1)
\label{BDonly2c}
\eeq

Similar reasoning to that used for \eqref{BD1} gives
\begin{align}
p(v_1|v_2) (1+1/\kappa) & = 2 p(v_2,v_2+v_3|v_1) + p(v_1|v_2|v_2+v_3) \label{DB1} \\
& = 2 p(v_1,v_2|v_2+v_3) + p(v_1|v_2|v_2+v_3) \label{DB2} \\
p(v_1|v_2) (1-1/\kappa) & = 2 p(v_1,v_2+v_3|v_2) + p(v_1|v_2|v_2+v_3) \label{DB3}
\end{align}
Later, we will be interested in $\bar p_1=p(v_1|v_2|v_2+v_3)$ and $\bar p_2 =p(v_1|v_2,v_2+v_3)$.
In this case, \eqref{DB1} implies that
\beq
2\bar p_2 + \bar p_1 = p(v_1|v_2)(1+1/\kappa)
\label{DBonly2c}
\eeq
We will also need the following consequence of \eqref{DB1} and \eqref{3to2}
\beq
\bar p_2 - p(v_1|v_2)/\kappa = p(v_1|v_2) - \bar p_1 - \bar p_2 = p(v_1,v_2+v_3|v_2) > 0
\label{cpos}
\eeq

Work of Tarnita et al.~\cite{TOAFN,TWN} has shown that when selection is weak (i.e., $w \ll 1/N$
where $N$ is the population size) one can determine whether a strategy in an $k$-strategy game
is favored by selection (i.e., has frequency $> 1/k$) by using an inequality that is 
linear in the entries of the game matrix that involves one ($k=2$) or two ($k\ge 3$) constants that
depend on the spatial structure. Our analysis will show that on $\ZZ^d$, $d \ge 3$,
the only aspects of the spatial structure relevant for a complete analysis of the game with small selection
are $p(0|v_1)$ and $p(0|v_1|v_1+v_2)$ for Birth-Death updating and $\kappa$,
$p(v_1|v_2)$ and $p(v_1|v_2|v_2+v_3)$ for Death-Birth updating. 

The coalescence probabilities $p(0|v_1)=p(v_1|v_2)$ are easily calculated since the difference
between the positions of the two walkers is a random walk. Let $S_n$ be the discrete time
random walk that has jumps according to $p$ and let 
$$
\phi(t) = \sum_x e^{itx} p(x)
$$
be its characteristic function (a.k.a~Fourier transform). The inversion formula implies
$$
P(S_n=0) = (2\pi)^{-d} \int_{(-\pi,\pi)^d} \phi^n(t) \, dt
$$
so summing we have
$$
\chi \equiv \sum_{n=0}^\infty P(S_n=0) =   (2\pi)^{-d} \int_{(-\pi,\pi)^d} \frac{1}{1-\phi(t)} \, dt
$$
For more details see pages 200--201 in \cite{PTE}. Since the number of visits to 0 has a geometric 
distribution with success probability $p(0|v_1)$ it follows that
$$
p(0|v_1) = \frac{1}{\chi}
$$
In the three dimensional nearest neighbor case it is know that $\chi = 1.561386\ldots$ so we have
$$
p(0|v_1) = p(v_1|v_2) = 0.6404566
$$

To evaluate $p(0|v_1|v_1+v_2)$ and $p(v_1|v_2|v_2+v_3)$ we have to turn to simulation. Simulations of
Yuan Zhang suggest that $p(0|v_1|v_1+v_2) \in [0.32,0.33]$ and $p(v_1|v_2|v_2+v_3) \in [0.34,0.35]$.

\section{PDE limit} \label{sec:pdelim}

In a homogeneously mixing population the frequencies of the strategies in an evolutionary game follow the 
replicator equation, see e.g., Hofbauer and Sigmund's book \cite{HS98}: 
\beq
\frac{du_i}{dt}  = \phi^i_R(u) \equiv u_i \left( \sum_k G_{i,k} u_k - \sum_{j,k} u_j G_{j,k} u_k \right). 
\label{rep}
\eeq

\mn
{\bf Birth-Death dynamics.} Let
$$
 p_1 =  p(0|v_1|v_1+v_2) \quad\hbox{and}\quad p_2 =  p(0|v_1,v_1+v_2).
$$
In this case the limiting PDE in Theorem \ref{hydro} is 
$\partial u_i/\partial_t = (1/2d) \Delta u + \phi^i_B(u)$ where
\beq
\phi^i_B(u) =   p_1\phi^i_R(u) + p_2 \sum_{j\neq i}  u_i u_j (G_{i,i}-G_{j,i}+ G_{i,j}-G_{j,j}). 
\label{phiBi}
\eeq
See Section \ref{sec:derPDE} for a proof.
Formula \eqref{BD1} implies that 
$$
2p(0|v_1,v_1+v_2) = p(0|v_1) - p(0|v_1|v_1+v_2),
$$
so it is enough to know the two probabilities on the right-hand side.

If coalescence is impossible then $p_1=1$ and $p_2=0$ and $\phi^i_B = \phi^i_R$.
There is a second more useful connection to the replicator equation. Let
$$
A_{i,j} = \frac{ p_2 }{ p_1   } (G_{i,i} + G_{i,j} - G_{j,i} - G_{j,j} ).
$$
This matrix is skew symmetric $A_{i,j} = -A_{j,i}$ so $\sum_{i,j} u_i A_{i,j} u_j = 0$ and it follows that
$\phi^i_B(u)$ is $p_1$ times the RHS of the replicator equation for the game matrix $A+G$. 
This observation is due to Ohtsuki and Nowak \cite{ON06} who studied the limiting ODE
that arises from the nonrigorous pair approximation. In their case, the perturbing matrix, see their (14), is
$$
\frac{1}{\kappa-2} (G_{i,i} + G_{i,j} - G_{j,i} - G_{j,j} ).
$$
To connect the two formulas note if space is a tree in which each site has $\kappa$
neighbors then $p(0,v_1) = 1/(\kappa-1)$. Under the pair approximation, the coalescence of 0 and $v_1$
is assumed independent of the coalescence of $v_1$ and $v_1+v_2$, so
$$
\frac{p_2}{p_1} = \frac{p(0|v_1,v_1+v_2)}{p(0|v_1|v_1+v-2)}  = \frac{p(v_1,v_1+v_2)}{p(v_1|v_1+v-2)} = \frac{1}{\kappa-2}.
$$  

\mn
{\bf Death-Birth Dynamics.} Let
$$
\bar p_1 =  p(v_1|v_2|v_2+v_3) \quad\hbox{and}\quad \bar p_2 =  p(v_1|v_2,v_2+v_3).
$$
Note that in comparison with $p_1$ and $p_2$, 0 has been replaced by $v_1$ and then the other two $v_i$ have been renumbered.
In this case the limiting PDE is $\partial u_i/\partial t = (1/2d) \Delta u + \phi^i_D(u)$ where
\begin{align}
\phi^i_D(u) &= \bar p_1 \phi^i_R(u)   + \bar p_2 \sum_{j\neq i}  u_i u_j (G_{i,i} -G_{j,i}+ G_{i,j} -G_{j,j}) 
\nonumber\\
& -(1/\kappa) p(v_1|v_2)  \sum_{j \neq i}  u_i u_j (  G_{i,j}-G_{j,i}).
\label{phiDi}
\end{align} 
Again see Section \ref{sec:derPDE} for a proof.
The first two terms are the ones in \eqref{phiBi}.
The similarity is not surprising since the numerators of the flip rates in \eqref{rijDB} are the flip rates in \eqref{rijBD}.
The third term comes from the denominator in \eqref{rijDB}. 
Formula \eqref{DB1} implies that 
$$
2p(v_1|v_2,v_2+v_3) = (1+1/\kappa) p(v_1|v_2) - p(v_1|v_2|v_2+v_3),
$$
so it is enough to know the two probabilities on the right-hand side.

As in the Birth-Death case, if we let 
$$
\bar A_{i,j}  =  \frac{\bar p_2 }{\bar p_1  } (G_{i,i} + G_{i,j} - G_{j,i} - G_{j,j} ) 
- \frac{p(v_1|v_2)} {  \kappa \bar p_1 } (G_{i,j} - G_{j,i}),
$$
then $\phi^D_i(u)$ is $\bar p_1$ times the RHS of the replicator equation for $\bar A+G$.
Again, Ohtsuki and Nowak \cite{ON06} have a similar result for the ODE resulting from the pair approximation.
In their case the perturbing matrix, see their (23), is
\begin{align*}
\frac{1}{\kappa-2} (G_{i,i} + G_{i,j} - G_{j,i} - G_{j,j} ) - \frac{\kappa}{(\kappa+1)(\kappa-2)} (G_{i,j} - G_{j,i}).
\end{align*}
This time the connection is not exact, since under the pair approximation 
$$
\frac{p(v_1|v_2)}{\kappa \bar p_1} = \frac{p(v_1|v_2)}{\kappa p(v_1|v_2|v_2+v_3)} = 
\frac{1}{\kappa p(v_2|v_2+v_3)} = \frac{\kappa-1}{\kappa(\kappa-2)}.
$$

\section{Two strategy games} \label{sec:2sg}

We now consider the special case of a $2 \times 2$ games.
\beq
\begin{matrix}
& 1 & 2 \\
1 & \alpha & \beta \\
2 & \gamma & \delta
\end{matrix}
\label{CDPgame}
\eeq
Let $u$ denote the frequency of strategy 1. In a homogeneously mixing population, 
$u$ evolves according to the replicator equation \eqref{rep}:
\begin{align}
\frac{du}{dt} & = u \{ \alpha u + \beta (1-u) 
- u [ \alpha u + \beta (1-u) ] - (1-u) [ \gamma u + \delta (1-u) ] \} 
\nonumber\\
& = u(1-u)[\beta - \delta + \Gamma u ] \equiv \phi_R(u)
\label{rep2}
\end{align}
where we have introduced $\Gamma = \alpha - \beta - \gamma + \delta$. Note that $\phi_R(u)$ is a cubic
with roots at 0 and at 1.if there is a fixed point in $(0,1)$
it occurs at
\beq
\bar u =  \frac{\beta-\delta }{\beta- \delta  + \gamma - \alpha }
\label{fixp}
\eeq

Using results from the previous section gives the following.

\mn
{\bf Birth-Death dynamics.} 
The limiting PDE is $\partial u/\partial t = (1/2d) \Delta u + \phi_B(u)$ where 
$\phi_B(u)$ is the RHS of the replicator equation for the  game
\beq
\begin{pmatrix} \alpha & \beta + \theta \\ \gamma - \theta & \delta \end{pmatrix}
\label{2spert}
\eeq
and $\theta = (p_2/p_1)(\alpha+\beta-\gamma-\delta)$.

\mn
{\bf Death-Birth dynamics.} 
The limiting PDE is $\partial u/\partial t = (1/2d) \Delta u + \phi_D(u)$ where
$\phi_D(u)$ is the RHS of the replicator equation for the game in \eqref{2spert} but now
$$
\theta = (\bar p_2/\bar p_1)(\alpha+\beta-\gamma-\delta) - (p(v_1|v_2)/\kappa \bar p_1)(\beta-\gamma).
$$

\subsection{Analysis of $2 \times 2$ games}

Suppose that the limiting PDE is $\partial u/\partial t = (1/2d) \Delta u + \phi(u)$ 
where $\phi$ is a cubic with roots at 0 and 1. There are four possibilities

\begin{center}
\begin{tabular}{ccc}
$S_1$ & $\bar u$ attracting & $\phi'(0)>0$, $\phi'(1)>0$ \\
$S_2$ &$\bar u$ repelling  & $\phi'(0)<0$, $\phi'(1)<0$ \\
$S_3$ & $\phi<0$ on $(0,1)$   & $\phi'(0)<0$, $\phi'(1)>0$ \\
$S_4$ & $\phi>0$ on $(0,1)$   & $\phi'(0)>0$, $\phi'(1)<0$
\end{tabular}
\end{center}

\noindent
To see this, we draw a picture. For convenience, we have drawn the cubic as a piecewise linear function.

\begin{center}
\begin{picture}(220,105)
\put(15,75){$S_1$}
\put(30,75){\line(1,0){60}}
\put(30,75){\line(1,1){15}}
\put(45,90){\line(1,-1){30}}
\put(75,60){\line(1,1){15}}
\put(82,80){\vector(-1,0){14}}
\put(38,70){\vector(1,0){14}}
\put(105,75){$S_2$}
\put(120,75){\line(1,0){60}}
\put(120,75){\line(1,-1){15}}
\put(135,60){\line(1,1){30}}
\put(165,90){\line(1,-1){15}}
\put(142,80){\vector(-1,0){14}}
\put(158,70){\vector(1,0){14}}
\put(15,30){$S_3$}
\put(30,30){\line(1,0){60}}
\put(30,30){\line(2,-1){30}}
\put(60,15){\line(2,1){30}}
\put(70,35){\vector(-1,0){20}}
\put(105,30){$S_4$}
\put(120,30){\line(1,0){60}}
\put(120,30){\line(2,1){30}}
\put(150,45){\line(2,-1){30}}
\put(140,25){\vector(1,0){20}}
\end{picture}
\end{center} 

We say that {\it $i$'s take over} if for all $L$
$$
P( \xi_s(x) = i \hbox{ for all $x\in[-L,L]^d$ and all $s \ge t$}) \to 1 \quad\hbox{as $t\to\infty$.}
$$
Let $\Omega_0 = \{ \xi : \sum_x \xi(x) = \infty, \sum_x (1-\xi(x)) = \infty \}$
be the configurations with infinitely many 1's and infinitely many 0's. We say that
{\it coexistence occurs} if there is a stationary distribution $\nu$ for the spatial model with $\nu(\Omega_0)=1$. 
The next result follows from Theorems 1.4 and 1.5 in \cite{CDP}. The PDE assumptions and the other conditions
can be checked as in the arguments in Section 1.4 of  \cite{CDP} for the Lotka-Volterra system.    

\begin{theorem} \label{CDPG}
If $\ep < \ep_0(G)$, then:\hbr 
In case $S_3$, 2's take over. In case $S_4$, 1's take over. \hbr
In case $S_2$, 1's take over if $\bar u < 1/2$, and 2's take over if $\bar u > 1/2$.\hbr
In case $S_1$, coexistence occurs. Furthermore, if $\delta>0$ and $\ep < \ep_0(G,\delta)$ then 
any stationary distribution with $\nu(\Omega_0)=1$ has
$$
\sup_x |\nu(\xi(x) =1 )- \bar u| < \delta.
$$
\end{theorem}

\noindent
We write $i \gg j$ if strategy $i$ dominates strategy $j$, i.e., it gives a strictly larger payoff against every reply.  
To begin to apply Theorem \ref{CDPG}, we note that if $\phi$ is the RHS of the replicator equation for the game matrix in \eqref{CDPgame}
then the cases are: 
\beq
\begin{matrix}
  & \beta > \delta &  \beta < \delta \\
 \alpha < \gamma & S_1. \hbox{ Coexistence} & S_3.\ 2 \gg 1 \\
 \alpha > \gamma & S_4.\ 1 \gg 2 & S_2. \hbox{ Bistable}
\end{matrix}
\label{4cases}
\eeq
To check $S_1$ we draw a picture.

\begin{center}
\begin{picture}(180,90)
\put(30,30){\line(1,0){120}}
\put(30,40){\line(6,1){120}}
\put(30,70){\line(4,-1){120}}
\put(30,25){\line(0,1){10}}
\put(150,25){\line(0,1){10}}
\put(28,15){0}
\put(149,15){1}
\put(90,15){$u$}
\put(20,67){$\beta$}
\put(20,37){$\delta$}
\put(155,57){$\gamma$}
\put(155,37){$\alpha$}
\end{picture}
\end{center}
 
\mn
When the frequency of strategy 1 is $u \approx 0$ then strategy 1 has fitness $\approx \beta$ and
strategy 2 has fitness $\approx \delta$, so $u$ will increase. The condition $\alpha<\gamma$ implies that when $u \approx 1$
it will decrease and the fixed point is attracting. When both inequalities are reversed in $S_2$, the fixed point exists
but is unstable. Finally the second strategy dominates the first in $S_3$, and the first strategy dominates the second in $S_4$.

\begin{figure}[tbp] 
  \centering
  \includegraphics[bb=0 0 300 279,width=3.83in,height=3.56in,keepaspectratio]{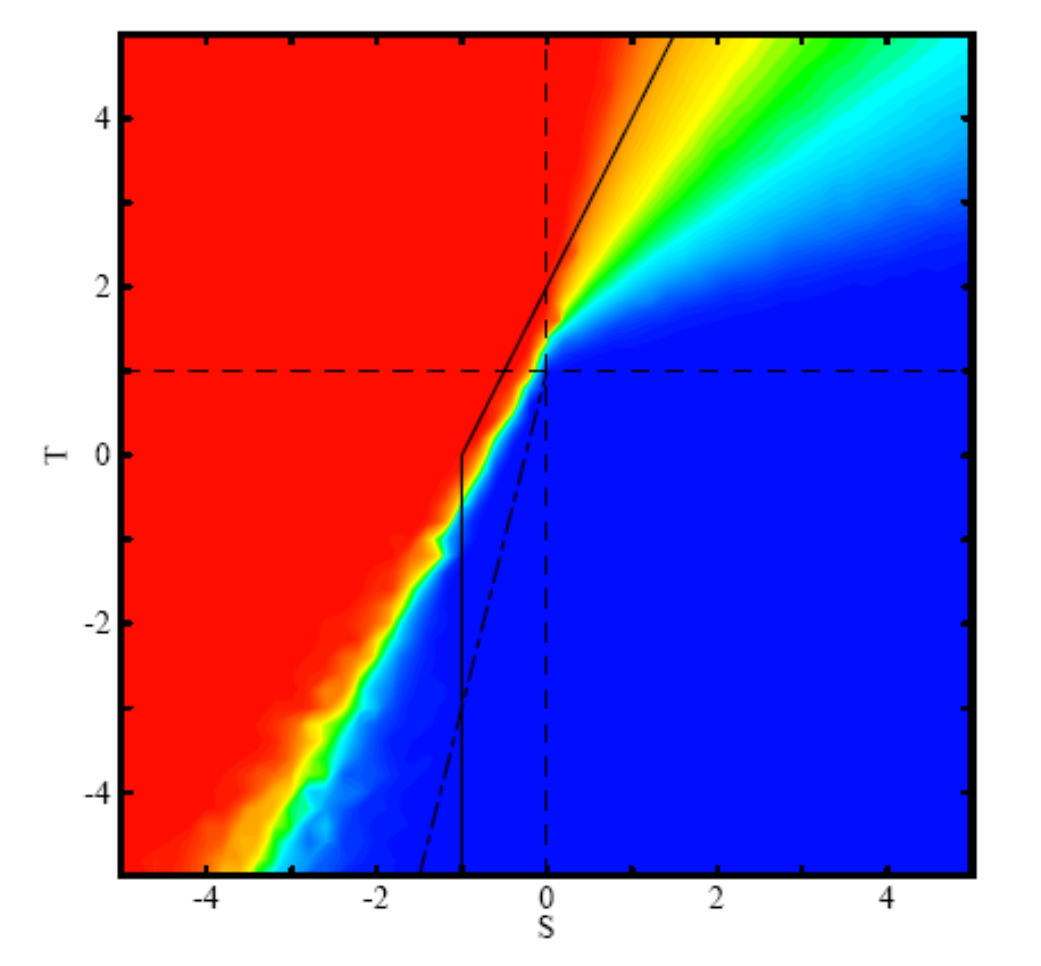}
  \caption{Phase diagram from Hauert's simulations}
  \label{fig:1dpd}
\end{figure}

\subsection{Phase diagram}

At this point, we can analyze the spatial version of any two strategy game. In the literature on $2 \times 2$ games it is common to use the following notation for payoffs, which was introduced in the classic paper by Axelrod and Hamilton \cite{AH81}.
\begin{center}
\begin{tabular}{ccc}
& C & D \\
C & $R$ & $S$ \\
D & $T$ & $P$
\end{tabular}
\end{center}
Here $T=$ temptation, $S=$ sucker payoff, $R=$ reward for cooperation, $P=$ punishment for defection.
If we assume, without loss of generality, that $R>P$ then there are 12 possible orderings for the payoffs.
However, from the viewpoint of Theorem \ref{CDPG}, there are only four cases. 

Hauert \cite{Ha02} simulates spatial games with $R=1$ and $P=0$ for a large number of values of $S$ and $T$. 
He considers three update rules: (a) switch to the strategy of the best neighbor, (b) pick a neighbor's strategy with
probability proportional to the difference in scores, and (c) pick a neighbor's strategy with probability
proportional to its fitness, or in our terms Death-Birth updating. 
He considers discrete and continuous time updates using the von Neumann neighborhood
(four nearest neighbors) and the Moore neighborhood (which also includes the diagonally adjacent neighbors). 
The picture most relevant to our investigation here is the graph in the lower left corner of his Figure 5, reproduced here as Figure 1, which shows equilibrium frequencies of the two strategies in continuous time for update rule (c)
on the von Neumann neighborhood. Similar pictures can be found in the work of Roca, Cuesta, and Sanchez \cite{RCS09a,RCS09b}.

Our situation is different from his, since the games we consider 
are small perturbations of the voter model game ${\bf 1}$, but as we will see, the qualitative features of the phase diagrams are
the same. Under either update the game matrix changes to 
\begin{center}
\begin{tabular}{cll}
& \quad C & \quad D \\
C & $\alpha = R$ & $\beta = S+\theta$ \\
D & $\gamma = T - \theta$ & $\delta = P$
\end{tabular}
\end{center}

\mn
{\bf Birth-Death updating.} In this case 
\beq
\theta = \frac{p_2}{p_1}(R+S-T-P)
\label{thBD}
\eeq
We will now find the boundaries between the four cases using \eqref{4cases}. Letting $\lambda = p_2/p_1 \in (0,1)$,
we have $\alpha=\gamma$ when
$$
R-T = - \theta = - \lambda (R+S-T-P)
$$
Rearranging gives $\lambda(S-T) = (1+\lambda)(T-R)$, and we have
\beq
T-R = \frac{\lambda}{\lambda+1}(S-P)
\label{L1BD}
\eeq
Repeating this calculation shows that $\beta=\delta$ when 
\beq
T-R = \frac{\lambda+1}{\lambda}(S-P)
\label{L2BD}
\eeq
This leads to the four regions drawn in Figure 2. Note that the coexistence region is smaller than in
the homogeneously mixing case. 

In the coexistence region, the equilibrium is
\beq
\bar u = \frac{S-P+\theta}{S-P+T-R}
\label{neweqBD}
\eeq
Plugging in the value of $\theta$ from \eqref{thBD} this becomes
\beq
\bar u = \frac{(1+\lambda)(S-P) + \lambda(R-T)}{S-P+T-R}
\label{eqBD}
\eeq
Note that in the coexistence region, $\bar u$ is constant on lines through $(S,T)=(P,R)$.

In the lower left region where there is bistability, 1's win if $\bar u < 1/2$
or what is the same if strategy 1 is better than strategy 2 when $u=1/2$, that is,
$$
R + S + \theta > T-\theta + P
$$
Plugging in the value of $\theta$ this becomes $(1+2\lambda)(R+S-T-P) > 0$ or 
\beq
R-T > P-S.
\label{1swinBD}
\eeq  
Writing this as $R+S > T+P$, we see that the population converges to strategy 1
when it is ``risk dominant'', a term introduced by Harsanyi and Selten \cite{HarSel}.
Note that bistablity in the replicator equation disappears in the
spatial model, an observation that goes back to Durrett and Levin \cite{DL94}.

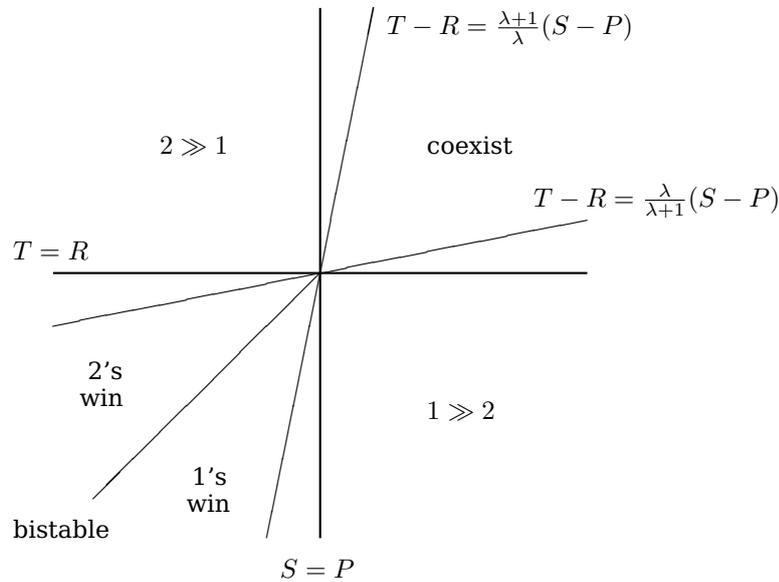
\begin{figure}[h]
\begin{center}
\begin{picture}(320,240)
\put(20,100){\line(5,1){200}}
\put(200,145){$T-R = \frac{\lambda}{\lambda+1}(S-P)$}
\put(100,20){\line(1,5){40}}
\put(145,210){$T-R = \frac{\lambda+1}{\lambda}(S-P)$}
\put(160,165){coexist}
\put(60,165){$2 \gg 1$}
\put(160,65){$1 \gg 2$}
\put(5,20){bistable}
\put(120,120){\line(-1,-1){85}}
\put(33,80){2's}
\put(30,70){win}
\put(70,30){win}
\put(72,40){1's}
\put(5,125){$T=R$}
\put(105,5){$S=P$}
\thicklines
\put(120,20){\line(0,1){200}}
\put(20,120){\line(1,0){200}}
\end{picture}
\caption{Phase diagram for Birth-Death Updating.}
\end{center}
\label{fig:BD}
\end{figure}

{\bf Death-Birth updating.} The phase diagram is similar to that for Birth-Death updating but
the regions are shifted over in space. Since the algebra in the derivation is messier,
we state the result here and hide the details away in Section \ref{sec:2sDB}. 
If we let $\mu= \bar p_2/\bar p_1$, $\nu = p(v_1|v_2)/\kappa \bar p_1$,
$$
P^* =  P - \frac{\nu (R-P)}{1+2(\mu-\nu)}, 
\qquad  R^* = R + \frac{\nu(R-P)}{1+2(\mu-\nu)}, 
$$
and let $\lambda = \mu-\nu$, then the two lines $\alpha=\gamma$ and $\beta=\delta$ can be written as 
$$
T-R^* = \frac{\lambda}{1+\lambda}(S-P^*) 
\quad\hbox{and}\quad T-R^* = \frac{1+\lambda}{\lambda}(S-P^*).
$$
This leads to the four regions drawn in Figure 3.
In the coexistence region, the equilibrium $\bar u$ is constant on lines through $(S,T)=(R^*,P^*)$.
In the lower left region where there is bistability, 1's win if 
$$ 
R^*-T > P^*-S.
$$ 
Even though Hauert's games do not have weak selection, there are many similarities with Figure 1.
The equilibrium frequencies are linear in the coexistence region, and in the lower left,
the equilibrium state goes from all 1's to all 2's over a very small distance.

\begin{figure}[ht]
\begin{center}
\begin{picture}(240,240)
\put(20,100){\line(5,1){200}}
\put(100,20){\line(1,5){40}}
\put(117,117){$\bullet$}
\put(75,125){$(P^*,R^*)$}
\put(200,145){$T-R^* = \frac{\lambda}{\lambda+1}(S-P^*)$}
\put(145,210){$T-R^* = \frac{\lambda+1}{\lambda}(S-P^*)$}
\put(160,165){coexist}
\put(60,165){$2 \gg 1$}
\put(160,65){$1 \gg 2$}
\put(5,20){bistable}
\put(120,120){\line(-1,-1){85}}
\put(33,80){2's}
\put(30,70){win}
\put(70,30){win}
\put(72,40){1's}
\put(5,115){$T=R$}
\put(115,5){$S=P$}
\thicklines
\put(130,20){\line(0,1){200}}
\put(20,110){\line(1,0){200}}
\end{picture}
\caption{Phase diagram for Death-Birth Updating.}
\end{center}
\label{fig:DB}
\end{figure}

\subsection{Tarnita's formula} 

Tarnita et al.~\cite{TOAFN} say that strategy $C$ is favored over $D$ in a structured population, and write $C > D$, if the frequency of $C$ in equilibrium is $> 1/2$ in the game $\bar G = {\bf 1} + wG$ when $w$ is small. Assuming that

\ms
(i) the transition probabilities are differentiable at $w=0$,

\ms
(ii) the update rule is symmetric for the two strategies, and

\ms
(iii) strategy $C$ is not disfavored in the game given by the matrix
$$
\begin{matrix} 
& C & D \\
C & 0 & 1 \\
D & 0 & 0 
\end{matrix}
$$

\mn
they argued that 

\mn
I. {\it $C>D$ is equivalent to $\sigma R + S > T + \sigma P$
where $\sigma$ is a constant that only depends on the spatial structure and update rule.}

\medskip
By using results for the phase diagram given above, we can show that 

\begin{theorem} \label{TF2}
I holds for the Birth-Death updating with $\sigma=1$ and 
for the Death-Birth updating with $\sigma = (\kappa+1)/(\kappa-1)$.
\end{theorem}

\begin{proof}
For Birth-Death updating it follows from \eqref{1swinBD} that this is the correct condition
in the bistable quadrant. By \eqref{eqBD}, in the coexistence quadrant,
$$
\bar u = \frac{(1+\lambda)(S-P) + \lambda(R-T)}{S-P+T-R} 
$$
Cross-multiplying we see that $\bar u > 1/2$ when we have
$$
0 < (1/2+ \lambda)(S-P) +(\lambda +1/2)(R-T) = (1/2 + \lambda)(R+S-T-P)
$$
Thus in both quadrants the condition is $R+S>T+P$.
The proof of the formula for Death-Birth updating is similar but requires more algebra, so again we hide the details
away in Section \ref{sec:2sDB}. Since the derivation of the formula from the phase diagram in the Death-Birth case is messy, 
we also give a simple self-contained proof of Theorem \ref{TF2} in this case. 
\end{proof}

\noindent

\subsection{Concrete examples}

In this, we present calculations for concrete examples to complement the general conclusions from the phase diagram.
Before we begin, recall that the original and modified games are
$$
\begin{matrix}
& C & D \\
C & R & S \\
D & T & P
\end{matrix}
\hphantom{xxxxxx}
\begin{matrix}
& C & D \\
C & \alpha=R & \beta=S+\theta \\
D & \gamma=T-\theta & \delta=P
\end{matrix}
$$
where $\theta = (p_2/p_1)(R+S-T-P)$ for Birth-Death updating and  
$$
\theta = \frac{\bar p_2}{\bar p_1}(R+S-T-P) - \frac{p(v_1|v_2)}{\kappa \bar p_1} (S-P) 
$$
for Death-Birth updating.

\begin{example}
{\bf Prisoner's Dilemma.} As formulated in Example \ref{PD}  
\begin{center}
\begin{tabular}{ccc}
& C & D \\
C & $R=b-c$ & $S=-c$ \\
D & $T=b$ & $P=0$
\end{tabular}
\end{center}
Under either updating the matrix changes to 
\begin{center}
\begin{tabular}{ccc}
& C & D \\
C & $\alpha=b-c$ & $\beta=-c+\theta$ \\
D & $\gamma=b- \theta$ & $\delta=0$
\end{tabular}
\end{center}
In the Birth-Death case, $\theta = (p_2/p_1)(b-c-c-b) = -2cp_2/p_1$. In this modified game 
$\Gamma=\alpha-\beta-\gamma+\delta=0$ and $\beta-\delta<0$ so 
recalling $\phi_R(u) = u(1-u)[\beta-\delta +\Gamma u]$, the cooperators always die out.
Under Death-Birth updating 
$$
\theta = -2c\frac{\bar p_2}{\bar p_1} - (-c-b) \frac{p(v_1|v_2)}{\kappa \bar p_1}.
$$  
Again $\Gamma=0$ so the victor is determined by the sign of 
$$
\bar p_1(\beta-\delta) = - c \bar p_1 - 2c \bar p_2 +(c+b) \frac{p(v_1|v_2)}{\kappa}
$$
Identity \eqref{DB1} implies that $2\bar p_2 + \bar p_1 = p(v_1|v_2) (1 + 1/\kappa)$ so cooperators will persist if
$$
(-c + b/\kappa) p(v_1|v_2) > 0.
$$ 
Since $p(v_1|v_2)>0$ the condition is just $b/c > \kappa$ giving a proof of the result of Ohtsuki et al.~\cite{OHLN},
which has already appeared as Corollary 1.14 in Cox, Durrett, and Perkins \cite{CDP}.
\end{example}

\begin{example}
{\bf Nowak and May \cite{NoMa92}} considered the ``weak'' Prisoner's Dilemma game with payoff matrix:
\begin{center}
\begin{tabular}{ccc}
& C & D \\
C & $R=1$ & $S=0$ \\
D & $T=b$ & $P=0$
\end{tabular}
\end{center}
As you can see from Figure 3, if Death-Birth updating is used, these games will show a dramatic departure from
the homogeneously mixing case. Nowak and May used ``imitate the best dynamics'' so the process was deterministic
and there are only finitely many different evolutions. See Section 2.1 of \cite{NoMa93} for locations of the transitions 
and pictures of the various cases. When $1.8 < b < 2$, if the process starts with a single $D$ in a sea 
of C's, and color the sites based on the values of $(\xi_{n-1}(x), \xi_n(x))$ a kaleidoscope of Persian carpet style patterns results.
As Huberman and Glance \cite{HG93} have pointed out, these patterns disappear if asynchronous updating is used. However, work of
Nowak, Bonhoeffer and May \cite{NBM94,NBM94b} showed that their conclusion that spatial structure enhanced cooperation
remained true with stochastic updating or when games were played on random lattices.   
\end{example}

\begin{example}
{\bf The Harmony game} has $P< S$ and $T < R$. In this game strategy 1 dominates strategy 2,
but in contrast to Prisoner's dilemma the payoff for the outcome $(C,C)$ is the largest in the matrix.  
Licht \cite{Licht} used this game to explain the proliferation of MOUs (memoranda of understanding) between securities agencies
involved in international antifraud regulation. 
From Figures 2 and 3, we see that in the spatial model cooperators take over the system.  
\end{example}

\begin{example}
{\bf Snowdrift game.} In this game, two motorists are stuck in their cars on opposite sides of a snowdrift.
They can get out of their car and start shoveling ($C$) or do nothing ($D$). The payoff matrix is
\begin{center}
\begin{tabular}{ccc}
& C & D \\
C & $R=b-c/2$ & $S=b-c$ \\
D & $T=b$ & $P=0$
\end{tabular}
\end{center}
That is, if both players shovel then the work is cut in half, but if one player cooperates and the other defects then
the $C$ player gains the benefit of sleeping in his bed rather than in his car.

The story behind the game makes it sound frivolous, however, in a paper published in Nature \cite{GYvO},
the snowdrift game has been used to study ``facultative cheating in yeast.'' For yeast to grow on sucrose, a disaccharide,
the sugar has to be hydrolyzed, but when a yeast cell does this, most of the resulting
monosaccharide diffuses away. None the less, due to the fact that the hydrolyzing cell reaps some benefit,
cooperators can invade a population of cheaters.

If $b>c$ then the game has a mixed strategy equilibrium, which by 
\eqref{fixp} is
$$
\frac{S-P}{S-P+T-R} = \frac{b-c}{b-(c/2)}
$$
Under either update rule the modified payoff becomes
\begin{center}
\begin{tabular}{ccc}
& C & D \\
C & $b-c/2$ & $b-c+\theta$ \\
D & $b-\theta$ & $0$
\end{tabular}
\end{center}
and using \eqref{fixp} again the equilibrium changes to
$$
\bar u = \frac{S-P+\theta}{S-P+T-R} = \frac{b-c+\theta}{b - (c/2)}
$$
assuming that this stays in $(0,1)$. If this $\bar u>1$, then 1 becomes an attracting fixed point; 
if $\bar u < 0$, then 0 is attracting. 

If $\theta>0$ then spatial structure enhances cooperation.
If we use Birth-Death updating:  
$$
\theta = \frac{p_2}{p_1}(R+S-T-P) 
= \frac{p_2}{p_1}(b - 3c/2)
$$
If we use Death-Birth updating:  
\begin{align*}
\theta & = \frac{\bar p_2}{\bar p_1}(R+S-T-P) - \frac{p(v_1|v_2)}{\kappa \bar p_1} (S-P) \\
& = \frac{\bar p_2}{\bar p_1}(b-3c/2) - \frac{p(v_1|v_2)}{\kappa \bar p_1} (b-c) 
\end{align*}
Hauert and Doebeli \cite{HaDo} have used simulation to show that spatial structure can inhibit the evolution of cooperation in the snowdrfit game. One of their examples has $R=1$, $S=0.38$, $T=1.62$, and $P=0$ in which case $\theta <0$ for both update rules. At the end of their article they conclude that ``space should benefit cooperation for low cost to benefit ratios,'' which is consistent with our calculation. For more discussion of the contrasting effects on
cooperation in Prisoner's Dilemma and Snowdrift games, see the review by Doebeli and Hauert \cite{DoHa}.
\end{example}

\begin{example}
{\bf Hawk-dove game.} As formulated in Example \ref{HD} the payoff matrix is
\begin{center}
\begin{tabular}{ccc}
& Hawk & Dove \\
Hawk & $(V-C)/2$ & $V$ \\
Dove & $0$ & $V/2$
\end{tabular}
\end{center}
Killingback and Doebeli \cite{KiDo96} studied the spatial version of the game and set $V=2$, $\beta = C/2$ to arrive at the payoff matrix
\begin{center}
\begin{tabular}{ccc}
& Hawk & Dove \\
Hawk & $1-\beta$ & $2$ \\
Dove & $0$ & $1$
\end{tabular}
\end{center}
We will assume $\beta>1$. In this case, \eqref{fixp} implies that the equilibrium mixed strategy plays Hawk with probability $\bar u = 1/\beta$. To put this game into our phase diagram we need to label the Dove strategy as Cooperate and the Hawk strategy as Defect:
\begin{center}
\begin{tabular}{ccc}
& C & D \\
C & $R=1$ & $S=0$ \\
D & $T=2$ & $P=1-\beta$
\end{tabular}
\end{center}

Under either of our update rules the payoff matrix changes to
\begin{center}
\begin{tabular}{ccc}
& H & D \\
H & $1-\beta$ & $2+\theta$ \\
D & $-\theta$ & $1$
\end{tabular}
\end{center} 
where $\theta = (p_2/p_1)(2-\beta)$ in the Birth-Death case and
$$
\theta = \frac{\bar p_2}{\bar p_1}(2-\beta) - \frac{p(v_1|v_2)}{\kappa \bar p_1} \cdot 1
$$
for Death-Birth updating. In both cases the frequency of Hawks in equilibrium is
$$
u_H = \frac{S-P+\theta}{S-P+T-R}
$$
see \eqref{neweqBD} and \eqref{neweqDB} below. In one of Killingback and Doebeli's favorite cases, $\beta = 2.2$, both of these terms are negative
in the death-birth case, so the frequency of Hawks in equilibrium is reduced, in agreement with their simulations. 

While the conclusions may be similar, the updates used in the models are very different. 
In \cite{KiDo96}, a discrete time dynamic (``synchronous updating'') was used in which the state of a cell at time $t+1$ is that of the eight Moore neighbors with the best payoff. As in the pioneering work of 
Nowak and May \cite{NoMa92} this makes the system deterministic and there are only finitely many different
behaviors as $\beta$ is varied with changes at $\beta$ passes through 9/7, 5/3, 2, and 7/3. Figure 1 in \cite{KiDo96}
shows spatial chaos, i.e., the dynamics show a sensitive dependence on initial conditions.  For more on this see \cite{KiDo98}.
In continuous time with small selection our results predict that as long as the mixed strategy equilibrium is
preserved in the perturbed game we will get coexistence of Hawks and Doves in an equilibrium with density of Hawks
and Doves close to that predicted by the perturbed game matrix.
\end{example}

\begin{example}
{\bf The Battle of the Sexes} is another game that leads to an attracting fixed point in the replicator equation. The story is that the man wants to go to a sporting event while the woman wants to go to the opera. In an age before cell phones they make their choices without communicating with each other. If $C$ is the choice to go to the other person's favorite and $D$ is go to your own then the game matrix might be
\begin{center}
\begin{tabular}{ccc}
& C & D \\
C & $R=0$ & $S=1$ \\
D & $T=2$ & $P=-1$
\end{tabular}
\end{center} 
In Hauert's scheme this case is defined by $T>S>R>P$ in contrast to the inequalities $T>R>S>P$ for
the snowdrift game, and Hawks-Doves.

Despite the difference in the inequalities the results are very similar. Again in either case the modified
payoffs in this particular example are 
\begin{center}
\begin{tabular}{ccc}
& C & D \\
C & $\alpha=0$ & $\beta=1+\theta$ \\
D & $\gamma=2-\theta$ & $\delta=-1$
\end{tabular}
\end{center} 
Under Birth-Death updating $\theta = 0$ since $R+S-T-P=0$, while for Death-Birth updating 
$$
\theta = - \frac{2p(v_1|v_2)}{\kappa \bar p_1} < 0.
$$
Since the equilibrium changes to
$$
\bar u = \frac{S-P+\theta}{S-P+T-R} 
$$
spatial structure inhibits cooperation.
\end{example}

\begin{example}
{\bf Stag Hunt.} 
As formulated in Example \ref{StagHunt}, 
\begin{center}
\begin{tabular}{ccc}
& Stag & Hare \\
Stag & $R=3$ & $S=0$ \\
Hare & $T=2$ & $P=1$
\end{tabular}
\end{center}

In Hauert's scheme this case is defined by the inequalities $R>T>P>S$. Since $R>T$ and $P>S$, we are in
the bistable situation. Returning to the general situation in either case the modified payoffs in this particular example are 
\begin{center}
\begin{tabular}{ccc}
& C & D \\
C & $\alpha=R$ & $\beta=S+\theta$ \\
D & $\gamma=T-\theta$ & $\delta=P$
\end{tabular}
\end{center} 
If $R+S > T+P$  then $\theta_{BD} >0$ while for Death-Birth updating 
$$
\theta_{DB} =  \theta_{BD} + (P-S) \frac{p(v_1|v_2)}{\kappa \bar p_1} > 0.
$$
So the 1's win out. If $R+S<T+P$ then $\theta_{BD} < 0$ so the 2's win,
but $\theta_{DB}$ may be positive or negative.

Under Birth-Death updating the winner is always the risk dominant strategy,
and under Death-Birth updating it often is. This is consistent with results
in the economics literature. See Kandori, Mailath, and Rob \cite{KMR}, Ellison \cite{E93}
and Blume \cite{Blume}. Blume uses a spatially explicit model with a
log-linear strategy revision, which turns the system into an Ising model.
\end{example}

\section{ODEs for the three strategy games} \label{sec:3sgode}

In this section we will prove results for the replicator equation in order to prepare for
analyzing examples in Section \ref{sec:cancer}. 
For simplicity, we will assume the game is written with zeros on the diagonal.
For the replicator equation and for the spatial model with Birth-Death updating, this entails no loss of generality.
\beq
G =
\begin{matrix} 
0 & \alpha_3 & \beta_2 \\ 
\beta_3 & 0 & \alpha_1 \\ 
\alpha_2 & \beta_1 & 0 
\end{matrix}
\label{gen33}
\eeq
Here we have numbered the entries by the strategy that is left out in the corresponding $2\times 2$ game.
It would be simpler to put the $\alpha$'s above the diagonal and the $\beta$'s below but (i) this scheme
simplifies the statement of Theorem \ref{HSrps} and (ii) this pattern of $\alpha$'s and $\beta$'s is unchanged by a cyclic permutation of the 
strategies
$$
\begin{matrix}
& 2 & 3 & 1\\ 
2 & 0 & \alpha_1 & \beta_3 \\ 
3 &\beta_1 & 0 & \alpha_2 \\ 
1 & \alpha_3 & \beta_2 & 0 
\end{matrix}
$$

It is not hard to check that in general
if the game $G$ in \eqref{gen33} has an interior fixed point it must be:
\begin{align}
\rho_1 & = (\beta_1\beta_2 + \alpha_1\alpha_3 - \alpha_1\beta_1)/D 
\nonumber\\
\rho_2 & = (\beta_2\beta_3 + \alpha_2\alpha_1 - \alpha_2\beta_2)/D 
\label{Geq}\\
\rho_3 & = (\beta_3\beta_1 + \alpha_3\alpha_2 - \alpha_3\beta_3)/D 
\nonumber
\end{align}
where $D$ is the sum of the three numerators. Conversely if the $\rho_i >0$ 
then this is an interior fixed point. See Section \ref{sec:alg3} for details.

\subsection{Special properties of replicator equations}

To study replicator equations it is useful to know some of the existing theory. To keep our
treatment self-contained we will prove many of the results we need. Our first two result are
for $n$ strategy games.

\subsubsection{Projective transformation}

\begin{theorem} \label{projt}
Trajectories $u_i(t)$ of the replicator equation for $G$ 
are mapped onto trajectories $v_i(t)$ for the replicator equation for $\hat G_{ij} = G_{ij}/m_j$
by $v_i = u_im_i/\sum_k u_km_k$. 
\end{theorem}

\noindent
This comes from Exercise 7.1.3 in \cite{HS98}. We will use this in the proof of Theorem \ref{HSrps} to
transform the game so that $\alpha_i + \beta_i$ constant. Another common application is to choose $m_i = \rho_i^{-1}$
where the $\rho_i$ are the coordinates of the interior equilibrium in order to make the equilibrium uniform.

\begin{proof}
To prove this note that 
\begin{align*}
\frac{dv_i}{dt} & = \frac{u_im_i}{\sum_k u_k m_k} \left( \sum_j G_{ij} u_j - \sum_{j,k} u_j G_{j,k} u_k \right) \\
& - \frac{u_im_i}{\left(\sum_k u_k m_k\right)^2} 
\sum_\ell u_\ell m_\ell \left( \sum_j G_{\ell,j} u_j - \sum_{j,k} u_j G_{j,k} u_k \right)
\end{align*}
The second terms on the two lines cancel leaving us with
\begin{align*}
& = \frac{u_im_i}{\sum_k u_k m_k} \left( \sum_j G_{ij} u_j 
-\sum_\ell \frac{u_\ell m_\ell}{\sum_k u_k m_k} G_{\ell,j} u_j \right) \\
& = \left(\sum_k u_k m_k \right) v_i \left( \sum_j \frac{G_{ij}}{m_j} v_j -\sum_\ell v_\ell \frac{G_{\ell,j}}{m_j}  v_j \right)
\end{align*}
The factor $\sum_k u_k m_k$ is a time change, so we have proved the desired result. 
\end{proof}

\subsubsection{Reduction to Lotka-Volterra systems}

We begin with the ``quotient rule''
\beq
\frac{d}{dt} \left( \frac{u_i}{u_n} \right) = \left( \frac{u_i}{u_n} \right) [ (Gu)_i - (Gu)_n ]
\label{quotr}
\eeq

\begin{proof} Using the quotient rule of calculus,
\begin{align*}
\frac{d}{dt}\left( \frac{u_i}{u_n} \right) & = \frac{1}{u_n} u_i [ (Gu)_i - u^TGu ]
- \frac{u_i}{u_n^2} u_n [ (Gu)_n - u^TGu ] \\
& = \left( \frac{u_i}{u_n} \right) [ (Gu)_i - (Gu)_n ] 
\end{align*}
which proves the desired result.
\end{proof}

\begin{theorem} \label{reptoLV}
The mapping $v_i = u_i/u_n$ sends trajectories $u_i(t)$, $1\le i \le n$, of the replicator equation
$$
\frac{du_i}{dt} = u_i [(Gu)_i - u^T Gu]
$$
onto the trajectories $v_i(t)$, $1\le i \le n-1$, of the Lotka-Volterra equation 
$$
\frac{dv_i}{dt} = v_i \left( r_i + \sum_{j=1}^{n-1} B_{ij} v_j \right)
$$
where $r_i = G_{i,n} - G_{n,n}$ and $B_{ij} = G_{i,j} -G_{n,j}$.
\end{theorem}

\begin{proof}
By subtracting $G_{n,j}$ from the $j$th column, we can suppose 
without loss of generality that the last row is 0.
By the quotient rule \eqref{quotr} and the fact that $v_i = u_i/u_n$
\begin{align*}
v_i' & =  v_i [ (Gu)_i - (Gu)_n ] = v_i \sum_{j=1}^n G_{i,j} v_j u_n \\
& = u_n v_i \left( G_{i,n} + \sum_{j=1}^{n-1} G_{i,j} v_j \right)
\end{align*}
The factor $u_n$ corresponds to a time change so the desired result follows.
\end{proof}

Theorem \ref{reptoLV} allows us to reduce the study of the replicator equation for three strategy
games to the study of two species Lotka-Volterra equation:
\begin{align}
dx/dt & = x(a+bx+cy) 
\nonumber\\
dy/dt & = y(d+ex+fy)
\label{LVeq}
\end{align}
If we suppose that $bf-ce \neq 0$ then the right-hand side is 0 when
\beq
x^* = \frac{dc-fa}{bf-ce} \qquad y^* = \frac{ea-bd}{bf-ce}
\label{LVfp}
\eeq

Suppose for the moment that $x^*, y^* > 0$. If we have an ODE 
$$
dx/dt = F(x,y) \qquad dy/dt = G(x,y)
$$
with a fixed point at $(x^*,y^*)$, then linearizing around the fixed point
and setting $X = x - x^*$ and $Y = y - y^*$ gives
\begin{align*}
\frac{dX}{dt} & = \frac{\partial F}{\partial x} X + \frac{\partial F}{\partial y} Y \\
\frac{dY}{dt} & = \frac{\partial G}{\partial x} X + \frac{\partial G}{\partial y} Y 
\end{align*}
In the case of Lotka-Volterra systems this is:
\begin{align*}
dX/dt & = x^*(bX+cY) \\
dY/dt & = y^*(eX+fY)
\end{align*}
In ecological competitions it is natural to assume
that $b<0$ and $f<0$ so that the population does not explode when the other species is absent.
In this case the trace of the linearization, which is the sum of the eigenvalues is  $x^*b + y^* f < 0$. 
The determinant, which is the product of the eigenvalues is $x^*y^*(bf-ce)$,
 so the fixed point will be locally attracting if $bf-ec>0$. The next result
was proved by Goh \cite{Goh} in 1976, but as Harrison \cite{HLV} explains, it has been proved many times, and was
known to Volterra in 1931.

\begin{theorem} \label{LVattr}
Suppose $b,f<0$, $bf-ec>0$, and $x^*,y^*>0$, which holds if $dc-fa>0$ and $ea-bd>0$. 
If $A,B >0$ are chosen appropriately then 
$$
V(x,y) = A(x - x^* \log x) + B( y - y^* \log y)
$$
is a Lyapunov function for the Lotka-Volterra equation, i.e., it is decreasing along solutions of \eqref{LVeq}, and hence
$x^*,y^*$ is an attracting fixed point.
\end{theorem}

\begin{proof} A little calculus gives
\begin{align}
\frac{dV}{dt} & = A(x-x^*) (a+bx+cy) + B(y-y^*) (d+ex+fy) 
\nonumber\\
& =  Ab(x-x^*)^2 + (Ac+ Be)(x-x^*)(y-y^*) +  Bf(y-y^*)^2
\label{VRHS}
\end{align}
since $a=-bx^*-cy^*$ and $d = -ex^* - f y^*$. 

If $c$ and $e$ have different signs then we can choose $A/B = -e/c$, so that $Ac+Be=0$ and
$$
\frac{dV}{dt} =  Ab(x-x^*)^2 +  Bf(y-y^*)^2
$$
$Ab$ and $Bf$ are negative so $V$ is a Lyapunov function.

To deal with the case $ce>0$, we write \eqref{VRHS} as
$$
\frac{1}{2}(X,Y)Q \begin{pmatrix} X \\ Y \end{pmatrix}  
\quad\hbox{where}\quad Q = \begin{pmatrix} 2Ab & Ac + Be \\ Ac + Be & 2Bf \end{pmatrix}
$$
We would like to arrange things so that the symmetric matrix $Q$ is negative definite, i.e., both
eigenvalues are negative. The trace of $Q$, which is the sum of the eigenvalues is $Ab+Bf < 0$. The determinant,
which is the product of the eigenvalues,  is
$$
4ABbf - (A^2c^2 + 2ABce + B^2e^2) = 4AB(bf-ce) - (Ac - Be)^2
$$
To conclude both eigenvalues are negative, we want to show that the determinant is positive.
We have assumed $bf-ce>0$. To deal with the second term we choose $A,B$ so that $Ac-Be=0$.

Finally, we have to consider the situation $ce=0$. We can suppose without loss of generality that $e=0$. In this case the determinant is
$A[4Bbf -Ac^2]$ with $bf>0$ so if $A/B$ is small this is positive.
\end{proof} 

It follows from Theorem \ref{reptoLV} that 
$$
\bar V(u_1,u_2,u_3) = A\left( \frac{u_1}{u_3} - x^* \log \frac{u_1}{u_3}\right) 
+ B\left( \frac{u_2}{u_3} - y^* \log \frac{u_2}{u_3} \right)
$$
is a Lyapunov function for the replicator equation. Unfortunately for our purposes, it is not convex near the boundary.

\subsection{Classifying three strategy games}

To investigate the asymptotic behavior of the replicator equation
in three strategy games, we will use a technique we learned from mathematical ecologists \cite{memoir}. We 
begin by investigating the three two-strategy sub-games.
When the game is written with 0's on the diagonal it is easy to determine the behavior
of the 1 vs.~2 subgame. The mixed strategy equilibrium, if it exists, is given by
\beq
\left( \frac{\beta_3}{\alpha_3+\beta_3}, \frac{\alpha_3}{\alpha_3+\beta_3} \right) 
\label{Gmse}
\eeq
In equilibrium both types have fitness $\alpha_3\beta_3/(\alpha_3+\beta_3)$. There are four cases:

\begin{itemize}
  \item 
$\alpha_3, \beta_3 > 0$, attracting (stable) mixed strategy equilibrium.
  \item
$\alpha_3 > 0$, $\beta_3 < 0$, strategy 1 dominates strategy 2, or $1 \gg 2$.
  \item
$\alpha_3 < 0$, $\beta_3 > 0$, strategy 2 dominates strategy 1, or $2 \gg 1$.
  \item
$\alpha_3, \beta_3 < 0$, repelling (unstable) mixed strategy equilibrium unstable.
\end{itemize}

\mn
Note that here the word stable refers only to the behavior of the replicator equation on the edge.

Our main technique for proving coexistence in the spatial game is to show the existence of a repelling  
function for the replicator equation associated with the (modified) game. 
(See Section \ref{ssec:repel} for the definition.) 
A repelling function will not exist if on one of the edges there is an unstable mixed strategy equilibrium,
so we will ignore games that have them. 
As the reader will soon see, if an edge, say $u_3=0$, has a stable mixed strategy equilibrium, 
we will also need to determine if the third strategy can invade: i.e., 
if $u_1,u_2$ is close to the boundary equilibrium then $u_3$ will increase.

Bomze \cite{Bomze83} drew 46 phase portraits to illustrate the possible behaviors of the replicator equation.
A follow up paper twelve years later,  \cite{Bomze95}, corrected the examples associated with five of the cases
and add two more pictures. Bomze considered situations in which an entire edge consists of fixed points
or there was a line of fixed points connecting two edge equilibria. Here, we will restrict our attention
to ``generic'' cases that do not have these properties. We approach the enumeration of possibilities by considering the number of 
stable edge fixed points, which can be 3, 2, 1, or 0, and then the number of these fixed points that can be invaded.

\subsubsection{Three edge fixed points}

Here, and in the next three subsections, we begin with the example with an attracting
interior equilibrium. In this and the next two subsections this is the case in which
all edge fixed points can be invaded.

\mn
{\bf Example 7.1. Three stable invadable edge fixed points.} (Bomze \#7)
If we assume that all $\alpha_i, \beta_i > 0$ (for notation see \eqref{gen33}), then there are three stable boundary equilibria. 
We assume that in each case, the third strategy can invade the equilibrium.

\begin{center}
\begin{picture}(180,130)
\put(87,125){1}\put(20,25){3}\put(155,25){2}
\put(59,76){$\bullet$}\put(87,27){$\bullet$}\put(116,76){$\bullet$}\put(87,55){$\bullet$} 
\put(65,76){\vector(3,-2){20}} \put(90,32){\vector(0,1){20}} \put(115,76){\vector(-3,-2){20}} 
\put(30,30){\line(2,3){60}}\put(35,48){\vector(2,3){15}}\put(80,115){\vector(-2,-3){15}} 
\put(150,30){\line(-2,3){60}} \put(143,50){\vector(-2,3){15}}\put(103,110){\vector(2,-3){15}} 
\put(30,30){\line(1,0){120}}\put(45,25){\vector(1,0){30}}\put(135,25){\vector(-1,0){30}} 
\end{picture}
\end{center}

The 2 vs.~3 subgame has a mixed strategy equilibrium at 
$$
(p_{23},q_{23}) = \left(\frac{\alpha_1}{\alpha_1+\beta_1},\frac{\beta_1}{\alpha_1+\beta_1}\right).
$$
Each strategy has fitness $\alpha_1\beta_1/(\alpha_1+\beta_1)$ at this point, so 1's can
invade this equilibrium if
\beq
\alpha_3 \cdot \frac{\alpha_1}{\alpha_1+\beta_1} + \beta_2 \cdot \frac{\beta_1}{\alpha_1+\beta_1}
> \frac{\alpha_1\beta_1}{\alpha_1+\beta_1}
\label{inv1}
\eeq
which we can rewrite as 
\beq
\alpha_3\alpha_1 + \beta_2\beta_1 - \alpha_1\beta_1 > 0
\label{num1}
\eeq
The invadability condition (\ref{num1}) implies that the numerator of $\rho_1$ in \eqref{Geq} is positive.

The 1 vs.~3 subgame has a mixed strategy equilibrium at 
$$
(p_{13},q_{13}) = \left( \frac{\beta_2}{\alpha_2+\beta_2},\frac{\alpha_2}{\alpha_2+\beta_2}\right).
$$
Each strategy has fitness $\alpha_2\beta_2/(\alpha_2+\beta_2)$ at this point, so 2's can
invade this equilibrium if
\beq
\beta_3 \cdot \frac{\beta_2}{\alpha_2+\beta_2} + \alpha_1 \cdot \frac{\alpha_2}{\alpha_2+\beta_2}
> \frac{\alpha_2\beta_2}{\alpha_2+\beta_2}
\label{inv2}
\eeq
which we can rewrite as 
\beq
\beta_2\beta_3 + \alpha_1\alpha_2  - \alpha_2\beta_2 > 0
\label{num2}
\eeq
The invadability condition (\ref{num2}) implies that the numerator of $\rho_2$ given in \eqref{Geq} is positive.

The 1 vs.~2 subgame has a mixed strategy equilibrium at 
$$
(p_{12},q_{12}) = \left(\frac{\alpha_3}{\alpha_3+\beta_3},\frac{\beta_3}{\alpha_3+\beta_3}\right).
$$
Each strategy has fitness $\alpha_3\beta_3/(\alpha_3+\beta_3)$ at this point, so 3's can
invade this equilibrium if
\beq
\alpha_2 \cdot \frac{\alpha_3}{\alpha_3+\beta_3} + \beta_1 \cdot \frac{\beta_3}{\alpha_3+\beta_3}
> \frac{\alpha_3\beta_3}{\alpha_3+\beta_3}
\label{inv3}
\eeq
which we can rewrite as 
\beq
\alpha_2\alpha_3 + \beta_1\beta_3 - \alpha_3\beta_3 > 0
\label{num3}
\eeq
The invadability condition (\ref{num3}) implies that the numerator of $\rho_3$ in \eqref{Geq} is positive.

Combining the last three results we see that there is an interior equilibrium.
To apply Theorem \ref{LVattr}, note that this game transforms into the Lotka-Volterra equation:
\begin{align}
dx/dt & = x[ \beta_2 -\alpha_2 x + (\alpha_3-\beta_1) y] 
\nonumber\\
dy/dt & = y[\alpha_1 + (\beta_3-\alpha_2)x  -\beta_1 y]
\label{GtoLV}
\end{align}
with an interior equilibrium $x^*,y^*>0$.
Clearly, $b, f < 0$. To check $bf-ec>0$ we note that by \eqref{num3}
\begin{align*}
bf-ce & = \alpha_2\beta_1 - (\alpha_3-\beta_1)(\beta_3-\alpha_2) \\
& = \alpha_2\alpha_3 + \beta_1\beta_3 - \alpha_3\beta_3 > 0
\end{align*}
which is the condition for $\rho_3>0$. 

For future reference note that if $bf-ce>0$ and there is interior fixed point then
\eqref{LVfp} implies
$$
dc-fa > 0 \quad\hbox{and}\quad ea-bd >0
$$
Plugging in the coefficients, these conditions become
\begin{align*}
0 < \alpha_1(\alpha_3-\beta_1) + \beta_1\beta_2 
& = \beta_1\beta_2 + \alpha_1\alpha_3 - \alpha_1\beta_1 \\
0 < \beta_2(\beta_3-\alpha_2) + \alpha_1 \alpha_2 
& = \beta_2\beta_3 + \alpha_2\alpha_1 - \alpha_2\beta_3
\end{align*}
i.e., the numerators of $\rho_1$ and $\rho_2$ are positive.

\mn
{\bf Example 7.1A.} If only TWO of the edge fixed points are invadable, then the numerator of one of the $\rho_i$ will be negative while the other two are positive, so there is no interior equilibrium. Bomze \#35 shows that the system will converge to the noninvadable fixed point. We can prove this
by transforming to a Lotka-Volterra equation. Since that argument will also cover Example 7.2A, we will wait until then to give the details.

\begin{center}
\begin{picture}(180,130)
\put(87,125){1}\put(20,25){3}\put(155,25){2}
\put(59,76){$\bullet$}\put(87,27){$\bullet$}\put(116,76){$\bullet$}
\put(78,68){\vector(-3,2){10}} \put(90,35){\vector(0,1){10}} \put(115,76){\vector(-3,-2){10}} 
\put(30,30){\line(2,3){60}}\put(35,48){\vector(2,3){15}}\put(80,115){\vector(-2,-3){15}} 
\put(150,30){\line(-2,3){60}} \put(143,50){\vector(-2,3){15}}\put(103,110){\vector(2,-3){15}} 
\put(30,30){\line(1,0){120}}\put(45,25){\vector(1,0){30}}\put(135,25){\vector(-1,0){30}} 
\end{picture}
\end{center}

\begin{lemma}
It is impossible to have a game with three stable edge fixed points and have ONE or ZERO of them are invadable.
\end{lemma}

\begin{proof}
Suppose that 1 cannot invade the $2,3$ equilibrium and that 2 cannot invade the $1,3$ equilibrium. 
By making a projective transformation, see Theorem \ref{projt}, we can suppose the $\alpha_i=1$. The failure of the two invadabilities implies that the numerators of $\rho_1$ and $\rho_2$ are negative: 
\begin{align*}
\beta_1\beta_2 + 1 &< \beta_1 \\
\beta_2\beta_3 + 1 &< \beta_2
\end{align*}
Since the $\beta_i>0$ the second equation implies $\beta_2>1$ and hence the first inequality is impossible.
\end{proof}

\subsubsection{Two edge fixed points}

\mn
{\bf Example 7.2. Two stable invadable edge fixed points.} (Bomze \#9)
If we suppose $\alpha_2, \beta_2 > 0$, $\alpha_1, \beta_1 > 0$ and $1 \gg 2$ ($\alpha_3>0>\beta_3$)
then there are two stable edge equilibria. 
In words all entries are positive except for $\beta_3$. Again, we assume that the two edge equilibria can be invaded.

\begin{center}
\begin{picture}(180,130)
\put(87,125){1}\put(20,25){3}\put(155,25){2}
\put(59,76){$\bullet$}\put(87,27){$\bullet$}
\put(87,55){$\bullet$} \put(65,76){\vector(3,-2){20}} \put(90,32){\vector(0,1){20}} 
\put(30,30){\line(2,3){60}}\put(35,48){\vector(2,3){15}}\put(80,115){\vector(-2,-3){15}} 
\put(150,30){\line(-2,3){60}}\put(135,60){\vector(-2,3){20}} 
\put(30,30){\line(1,0){120}}\put(45,25){\vector(1,0){30}}\put(135,25){\vector(-1,0){30}} 
\end{picture}
\end{center}

By the reasoning in Example 7.1, if the 2's can invade the $1,3$ equilibrium then
the numerator of $\rho_2$ given in \eqref{Geq} is positive., and if 1's can invade the $2,3$ equilibrium
then the numerator of $\rho_1$ in \eqref{Geq} is positive.
In the numerator of $\rho_3$, $\alpha_3\alpha_2>0$ and $-\alpha_3\beta_3>0$, but $\beta_3\beta_1<0$.
It does not seem to be possible to prove algebraically that 
$$
\beta_3\beta_1 + \alpha_3\alpha_2 - \alpha_3\beta_3 > 0
$$
In Section \ref{ssec:Lexist}, we will show that in this class of examples there is a repelling function. This will imply that trajectories cannot reach the boundary, so the existence of a fixed point follows from Theorem \ref{reptoLV} and the next result which is Theorem 5.2.1 in \cite{HS98}.
Here, an $\omega$-limit is a limit of $u(t_n)$ with $t_n\to\infty$,

\begin{theorem} \label{restpt}
If a Lotka-Volterra equation has an $\omega$-limit in $\Gamma = \{ (u_1, \ldots u_n) : u_i > 0, u_1 + \cdots + u_n = 1 \}$ 
then it has a fixed point in that set.
\end{theorem}

\noindent
Once we know that the fixed point exists it is easy to conclude that it is attracting. As in Example 7.1
in the associated Lotka-Volterra equation $b, f < 0$ and the condition $bf-ec>0$ follows from the fact that $\rho_3>0$.

\mn
{\bf Example 7.2A.} 
If only ONE edge fixed point is invadable, then one numerator is positive and two are negative, so there is no interior equilibrium. Bomze \#37 and \#38 (which differ only in the dominance relation between 1 and 2) show that in this case the replicator equation approaches the noninvadable edge fixed point.

\begin{center}
\begin{picture}(180,130)
\put(87,125){1}\put(20,25){3}\put(155,25){2}
\put(59,76){$\bullet$}\put(87,27){$\bullet$} 
\put(78,68){\vector(-3,2){10}} \put(90,35){\vector(0,1){10}} 
\put(30,30){\line(2,3){60}}\put(35,48){\vector(2,3){15}}\put(80,115){\vector(-2,-3){15}} 
\put(150,30){\line(-2,3){60}}\put(135,60){\vector(-2,3){20}} 
\put(30,30){\line(1,0){120}}\put(45,25){\vector(1,0){30}}\put(135,25){\vector(-1,0){30}} 
\end{picture}
\end{center}

\begin{proof}
To prove the claim about the limit behavior of this system and of Example 7.1A, we transform to a
Lotka-Volterra system.

\begin{center}
\begin{picture}(180,160)
\put(30,30){\line(1,0){120}} \put(40,25){\vector(1,0){15}} \put(100,25){\vector(-1,0){15}}
\put(30,30){\line(0,1){110}} \put(25,50){\vector(0,1){20}} \put(25,130){\vector(0,-1){15}}
\put(67,27){$\bullet$}\put(27,97){$\bullet$} 
\put(70,30){\line(-4,3){40}}
\put(30,100){\line(3,-2){105}}
\put(70,35){\vector(0,1){15}}
\put(50,100){\vector(-1,0){15}}
\end{picture}
\end{center}

\noindent
In the picture the dots are the boundary equilibria. The sloping lines are the null clines $dx/dt=0$, and $dy/dt=0$. The absence of an interior fixed point implies that they do not cross. The invadability conditions and the behavior near $(0,0)$ imply that $dx/dt=0$ lies below $dy/dt=0$. Lemma 5.2 in \cite{memoir} constructs a convex Lyapunov function which proves convergence to the fixed point on the $y$ axis. Unfortunately, when one pulls this function back to be a Lyapunov function for the replicator equation, it is no longer convex. \end{proof}

\mn
{\bf Example 7.2B.}
If NEITHER fixed point is invadable, see Bomze \#10, there is an interior saddle point separating the domains of attraction
of the two edge equilibria. As we mentioned in the overview, by analogy with results of Durrett and Levin \cite{DL94},
we expect that the winner in the spatial game is dictated by the direction of movement of the traveling wave connecting the
two boundary fixed points, but we do not even know how to prove that the traveling wave exists.

\begin{center}
\begin{picture}(180,130)
\put(87,125){1}\put(20,25){3}\put(155,25){2}
\put(59,76){$\bullet$}\put(87,27){$\bullet$}
\put(87,55){$\bullet$} \put(85,63){\vector(-3,2){20}} \put(90,52){\vector(0,-1){18}} 
\put(30,30){\line(2,3){60}}\put(35,48){\vector(2,3){15}}\put(80,115){\vector(-2,-3){15}} 
\put(150,30){\line(-2,3){60}}\put(135,60){\vector(-2,3){20}} 
\put(30,30){\line(1,0){120}}\put(45,25){\vector(1,0){30}}\put(135,25){\vector(-1,0){30}} 
\end{picture}
\end{center}

\subsubsection{One edge fixed point}

There are several possibilities based on the orientation of the edges on the other two sides.
We begin with the one that leads to coexistence.

\mn
{\bf Example 7.3. One stable invadable edge fixed point.} (Bomze \#15). 
Suppose  $3 \gg 2$ ($\beta_1 > 0 > \alpha_1$), $2 \gg 1$ ($\beta_3 > 0 >\alpha_3$), $\alpha_2, \beta_2 > 0$,
and $2$'s can invade the $1,3$ equilibrium.

\begin{center}
\begin{picture}(180,150)
\put(30,30){\line(1,0){120}}\put(30,30){\line(2,3){60}}\put(150,30){\line(-2,3){60}}
\put(90,125){$1$}\put(20,25){$3$}\put(155,25){$2$} 
\put(59,76){$\bullet$}
\put(35,48){\vector(2,3){15}}\put(80,115){\vector(-2,-3){15}} 
\put(87,55){$\bullet$} \put(65,76){\vector(3,-2){20}} 
\put(115,95){\vector(2,-3){20}} 
\put(105,25){\vector(-1,0){30}} 
\end{picture}
\end{center}

In words, all entries are positive except for $\alpha_1$ and $\alpha_3$.
As in the previous two examples, the fact that 2's can invade the 1,3 equilibrium implies that 
the numerator of $\rho_2$ in \eqref{Geq} is positive.
The numerator of $\rho_1$ is positive since
$$
\beta_1\beta_2 > 0, \quad \alpha_1\alpha_3 >0, \quad -\alpha_1\beta_1 >0
$$
In the numerator of $\rho_3$ 
$$
\beta_3\beta_1 > 0, \quad \alpha_3\alpha_2 <0, \quad -\alpha_3\beta_3 >0
$$
so again we have to resort to the existence of a repelling function proved 
in Section \ref{ssec:Lexist} and Theorem \ref{restpt} to prove that this is positive. As in the previous two cases, the game transforms into the Lotka-Volterra equation given in \eqref{GtoLV}.
It has $b,f < 0$, and the positivity of $bf-ce$ follows from that of
the numerator of $\rho_3$, so the fixed point is globally attracting.

\mn
{\bf Example 7.3A.}
Suppose now we reverse the direction of the $2,3$ edge. 

\begin{center}
\begin{picture}(180,150)
\put(30,30){\line(1,0){120}}\put(30,30){\line(2,3){60}}\put(150,30){\line(-2,3){60}}
\put(90,125){$1$}\put(20,25){$3$}\put(155,25){$2$} 
\put(57,73){$\bullet$}
\put(35,48){\vector(2,3){15}}\put(80,115){\vector(-2,-3){15}} 
\put(67,71){\vector(2,-1){50}} 
\put(115,95){\vector(2,-3){20}} 
\put(75,25){\vector(1,0){30}} 
\end{picture}
\end{center}

\noindent
If 2 can invade the $1,3$ equilibrium then 2's take over the 
system (Bomze \#40). 

\begin{proof}
We will convert the system to Lotka-Volterra, so it is convenient to relabel things so that the edge with the stable equilibrium is $1,2$.
The behavior of the edges implies that the sign pattern of the matrix
$$
\begin{matrix} 
& 1 & 2 & 3 \\
1 & 0 & + & - \\
2 & + & 0 & - \\
3 & + & + & 0
\end{matrix}
$$
When we convert to a Lotka-Volterra system, we subtract the last row of the matrix from the other two, and the last column gives us the constants.
\begin{align*}
dx/dt & = x (a + bx + cy) \\
dy/dt & = y (d + ex + fy)
\end{align*} 
The sign pattern of the game matrix tells us that $a,d,b,f<0$, while $c$ and $e$ can take either sign. 
If $c<0$ then $x$ decreases to 0, and once it is small enough then y decreases as well. A similar argument applies if $e<0$,
and in either case the limit is $(0,0)$.

\begin{center}
\begin{picture}(180,160)
\put(30,30){\line(1,0){120}} 
\put(55,25){\vector(-1,0){15}} 
\put(120,25){\vector(-1,0){20}}
\put(67,27){$\bullet$}
\put(70,30){\line(3,2){80}}
\put(155,80){$dy/dt = 0$}
\put(145,40){\vector(-3,2){15}}
\put(30,30){\line(0,1){120}} 
\put(25,55){\vector(0,-1){15}} 
\put(25,120){\vector(0,-1){20}}
\put(27,63){$\bullet$} 
\put(30,65){\line(3,4){60}}
\put(95,145){$dx/dt=0$}
\put(40,145){\vector(2,-3){10}}
\put(120,120){\vector(-1,-1){15}}
\put(80,80){\vector(-1,-1){15}}
\end{picture}
\end{center}

Now suppose $c,e>0$. Since 3 can invade the 1,2 equilibrium, the numerator of $\rho_3>0$, so referring back to the
analysis of \eqref{GtoLV}, we see that $bf-ec>0$. Using Theorem \ref{LVattr} we see that there cannot be an
interior equilibrium or it would be globally attracting, contradicting the fact that (0,0) is at least locally attracting. 
Since there is no interior equilibrium, the null clines $dx/dt=0$ and $dy/dt=0$ cannot intersect. The arrows in the diagram
show the direction of movement in the three regions. From this we see that the solution will eventually enter the central region which it cannot
leave and so must converge to $(0,0)$.

Since there are no boundary equilibria, $dx/dt=0$ cannot intersect the $x$ axis and  $dy/dt=0$ cannot intersect the $y$ axis,
so the line $dx/dt=0$ must lie above $dy/dt=0$. A little thought reveals that
\begin{center}
\begin{tabular}{lcc}
in the region & $dx/dt$ & $dy/dt$ \\
$A=$ above $dx/dt=0$ & $>0$ & $<0$ \\
$B=$ below $dy/dt=0$ & $<0$ & $>0$ \\
$C=$ in between      & $<0$ & $<0$
\end{tabular}
\end{center}
Consulting the picture we see that if the trajectory starts in $A$ or $B$ then it will enter $C$, and cannot re-enter $A$ or $B$
so it must converge to $(0,0)$.
\end{proof}

\mn
{\bf Example 7.3B.}
At first glance it may seem that $2 \gg 1$ and $2 \gg 3$ imply that $2$ can invade
the $1,3$ equilibrium. However, $2 \gg 3$, $2 \gg 1$, and the existence of a stable fixed point on the $1,3$ edge
implies only that $\alpha_1, \alpha_3 < 0$ while the other six entries are positive. The example
$$
\begin{matrix}
& 1 & 2 & 3 \\
1 & 0 & -1 & 3 \\
2 & 1 & 0  & 1 \\
3 & 3 & -1 & 0
\end{matrix}
$$
shows it is possible that $2$ is not invadable (Bomze \#14). Since 2 is always locally attracting (look at the middle column of the matrix) there is an interior fixed point, which is a saddle point. The system is bistable. See the discussion of Example 7.2B for the conjectured behavior of the spatial game.

\begin{center}
\begin{picture}(180,150)
\put(30,30){\line(1,0){120}}\put(30,30){\line(2,3){60}}\put(150,30){\line(-2,3){60}}
\put(90,125){$1$}\put(20,25){$3$}\put(155,25){$2$} 
\put(57,73){$\bullet$}
\put(35,48){\vector(2,3){10}}\put(80,115){\vector(-2,-3){15}} 
\put(81,64){\vector(-2,1){17}}\put(85,59){$\bullet$}\put(93,58){\vector(2,-1){50}} 
\put(115,95){\vector(2,-3){20}} 
\put(75,25){\vector(1,0){30}} 
\end{picture}
\end{center}

\mn
{\bf Example 7.3C.}  
If both edges point toward the $1,3$ edge, then the sign pattern in the matrix is
$$
\begin{matrix}
& 1 & 2 & 3 \\
1 & 0 & + & + \\
2 & - & 0  & - \\
3 & + & + & 0
\end{matrix}
$$
Strategy 2 is dominated by each of the other two, so the 2's die out (Bomze \#42), and the replicator equation will converge to the $1,3$ equilibrium.  Let $(p_{13},0,q_{13})$ be the boundary equilibrium. The proof of Lemma \ref{afb} will show that if $\ep$ is small enough
$$
V(u) = u_2 + \ep[ u_1 - p_{13}\log u_1 + u_3 - q_{13} \log(u_3) ]
$$
is a convex Lyapunov function, so using Theorem 1.4 in \cite{CDP} we can show that any equilibrium has frequencies close to the boundary equilibrium. It should be possible to show that the 2's die out, but as in \cite{CDP} this is much harder than proving coexistence.

\begin{center}
\begin{picture}(180,150)
\put(30,30){\line(1,0){120}}\put(30,30){\line(2,3){60}}\put(150,30){\line(-2,3){60}}
\put(90,125){$1$}\put(20,25){$3$}\put(155,25){$2$} 
\put(57,73){$\bullet$}
\put(35,48){\vector(2,3){15}}\put(80,115){\vector(-2,-3){15}} 
\put(117,46){\vector(-2,1){50}} 
\put(135,65){\vector(-2,3){20}} 
\put(105,25){\vector(-1,0){30}} 
\end{picture}
\end{center}

\mn
{\bf Example 7.3D.}
Suppose now that $1,3$ cannot be invaded by $2$'s. We have already considered this possibility in Examples 7.3B and 7.3C so we can suppose that the arrows are in opposite directions. If we transform this example or the previous one to a Lotka-Volterra system by dropping the third strategy, then we can use Lemma 5.2 in \cite{memoir} to construct a convex Lyapunov function. As in the case of Example 7.3A, it is unfortunate that (i) this does not pull back to be a convex function for the replicator equation, and (ii) we cannot generalize the Lyapunov function from the previous example to cover. Thus we leave it to some reader to show that in the spatial model $2$'s will die out leaving an equilibrium consisting of 1's and 3's.

\begin{center}
\begin{picture}(180,180)
\put(30,30){\line(1,0){120}}\put(30,30){\line(2,3){60}}\put(150,30){\line(-2,3){60}}
\put(90,125){$1$}\put(20,25){$3$}\put(155,25){$2$} 
\put(59,76){$\bullet$}
\put(35,48){\vector(2,3){15}}\put(80,115){\vector(-2,-3){15}} 
\put(85,62){\vector(-3,2){20}} 
\put(115,95){\vector(2,-3){20}} 
\put(105,25){\vector(-1,0){30}} 
\end{picture}
\end{center}

\subsubsection{No edge fixed points}

There are $2^3=8$ possible orientations for the arrows on the three edges. Two lead to an interesting situation.

\mn
{\bf Example 7.4. Rock Paper Scissors.} 
That is, $1 \ll 2 \ll 3 \ll 1$, (or $1 \gg 2 \gg 3 \gg 1$).

\begin{center}
\begin{picture}(180,150)
\put(30,30){\line(1,0){120}}\put(30,30){\line(2,3){60}}\put(150,30){\line(-2,3){60}}
\put(90,125){$3$}\put(15,25){$1$}\put(155,25){$2$} 
\put(65,95){\vector(-2,-3){20}} 
\put(135,65){\vector(-2,3){20}} 
\put(75,25){\vector(1,0){30}} 
\end{picture}
\end{center}

\noindent
In the situation drawn we have $\beta_i> 0 > \alpha_i$ so the denominator of $\rho_i$
$$
\beta_i \beta_j + \alpha_i \alpha_k - \alpha_i \beta_i > 0
$$
and there is an interior fixed point. Theorem 7.7.2 in Hofbauer and Sigmund  \cite{HS98} describes the asymptotic behavior of the game.

\begin{theorem} \label{HSrps}
Let $\Delta = \beta_1\beta_2\beta_3 + \alpha_1\alpha_2\alpha_3$. 
If $\Delta>0$ solutions converge to the fixed point. If $\Delta < 0$ 
their distance from the boundary tends to 0. If $\Delta=0$ there is a one-parameter family of periodic orbits.
\end{theorem}

\noindent 
Bomze \#17 is the case $\Delta>0$, while \#16 is the case $\Delta=0$. He does not give an example of $\Delta<0$
because in his classification he does not distinguish a flow from its reversal.  

\begin{proof}
By using the projective transformation in Lemma \ref{projt} 
we can make the game constant sum, i.e., $\beta_i+\alpha_i=\gamma$. 
In the constant sum case, if $\rho_i$ is the interior equilibrium, $V(u) = \prod_i u_i^{\rho_i}$ and $u(t)$ is a solution of the replicator equation then
\begin{align*}
\frac{d}{dt} V(u(t)) & = V \cdot (\rho^T Gu - u^T Gu) \\
& = - V \cdot (u-\rho)^T Gu  = - V \cdot (u-\rho)^T G(u - \rho)
\end{align*}
since $(G\rho)_i$ does not depend on $i$ and $\sum u_i - \rho_i = 0$. Let $\xi = u-\rho$.
Since $G_{ij}+G_{ji} = \gamma$ the above 
\begin{align*}
& = - \gamma V (\xi_1\xi_2 + \xi_2 \xi_3 + \xi_3\xi_1) \\
& = - \frac{\gamma}{2} V \left( (\xi_1+\xi_2+\xi_3)^2 - (\xi_1^2 + \xi_2^2 + \xi_3^2) \right) \\
& = - \frac{\gamma}{2} V \sum_i (u_i - \rho_i)^2
\end{align*}
In the constant sum case, the sign of $\gamma$ is the same as the sign of $\Delta$. The gradient
$$
\nabla V = V \cdot \left( \frac{\rho_1}{u_1}, \frac{\rho_2}{u_2}, \frac{\rho_3}{u_3} \right)
$$
$V$ has a local maximum on $\{ (u_1,u_2,u_3) : u_1+u_2+u_3 = 1 \}$, where $\nabla V$ is perpendicular to the constraint set, so on $\{ (u_1,u_2,u_3) : u_i \ge 0, u_1+u_2+u_3 = 1 \}$ it is maximized 
when $u_i = \rho_i$. \end{proof}

\mn
{\bf Example 7.4A.} 
The other six orientations lead to one strategy dominating the other two (Bomze \#43) and taking over the system. 
Consider now, without loss of generality the case in which $1 \gg 2$, $1 \ll 3$, and $2 \ll 3$.

\begin{center}
\begin{picture}(180,150)
\put(30,30){\line(1,0){120}}\put(30,30){\line(2,3){60}}\put(150,30){\line(-2,3){60}}
\put(90,125){$3$}\put(15,25){$1$}\put(155,25){$2$} 
\put(65,95){\vector(-2,-3){20}} 
\put(135,65){\vector(-2,3){20}} 
\put(105,25){\vector(-1,0){30}} 
\end{picture}
\end{center}

In this case
the matrix has the form
$$
\begin{matrix}
& 2 & 3 & 1 \\
2 & 0 & -e & -f \\
3 & c & 0 & -d \\
1 & a & b & 0
\end{matrix}
$$
with $a, b, c, d, e ,f >0$.
We have written the matrix this way so we can use Theorem \ref{reptoLV} to transform into a Lotka-Volterra equation with
\begin{align*}
dx/dt & = x (-f -ax -(e+b) y) \\
dy/dt & = y (-d +(c-a)x - by) \\
\end{align*}
where $x=u_2/u_1$ and $y=u_3/u_1$. From the equation it is clear that $x(t) \downarrow 0$. If $a\ge c$, $y(t) \downarrow 0$.
If $a<c$ then when $x < d/(c-a)$, $y$ is decreasing and will converge to 0.

\section{Spatial three strategy games} \label{sec:3sgsp}

Recall that we are supposing the game is written in the form. 
\beq
G =
\begin{matrix} 
0 & \alpha_3 & \beta_2 \\ 
\beta_3 & 0 & \alpha_1 \\ 
\alpha_2 & \beta_1 & 0 
\end{matrix}
\label{Gdef2}
\eeq
Since the diagonal entries $G_{i,i}=0$, the reaction terms for both updates have the form
$$
p \left( \phi^i_R(u) + \theta \sum_j u_i u_j (G_{i,j} - G_{j,i}) \right)
$$
where $p=p_1$ and $\theta = p_2/p_1$ for Birth-Death updating and for Death-Birth updating $p=\bar p_1$ and
$$
\theta = \frac{\bar p_2 - p(v_1|v_2)/\kappa}{\bar p_1}
$$
Thus the perturbation matrix $A_{ij} = \theta (G_{i,j} - G_{j,i})$, and
the limiting reaction-diffusion equation of interest has the form 
\beq
\frac{\partial u_i}{\partial t} = \frac{\sigma^2}{2} \Delta u_i + \phi_H(u)
\label{RDE}
\eeq
where $\phi_H$ is the right-hand side of the replicator equation for the matrix
\beq
H=
\begin{pmatrix} 
0 & (1+\theta)\alpha_3 - \theta\beta_3 & (1+\theta)\beta_2 -\theta\alpha_2 \\
(1+\theta)\beta_3 - \theta \alpha_3 & 0 & (1+\theta)\alpha_1 - \theta\beta_1 \\
(1+\theta)\alpha_2 - \theta\beta_2 & (1+\theta)\beta_1 - \theta\alpha_1 & 0
\end{pmatrix}
\label{pertm}
\eeq
More compactly, the diagonal entries are 0. If $i \neq j$ then $H_{i,j} = (1+\theta)G_{i,j} - \theta G_{j,i}$.

Here we are interested in how the system behaves in general, so we will forget about the
transformation that turned $G$ into $H$ and assume that $H$ has entries given in \eqref{Gdef2}.
However, we will consistently use $H$ for the game matrix to remind ourselves that we
are working with the transformed game, and the results we prove are for $1 + \ep^2 G$
where $G$ is the original game.

\subsection{Repelling functions} \label{ssec:repel}

The key to our study of spatial games is a result stated in the introduction of \cite{memoir} as Proposition 1. Given is an ODE
$$
\frac{du_i}{dt} = f_i(u)
$$
A continuous function $\phi$ from $\Gamma = \{ (u_1, \ldots u_n) : u_i \ge 0, u_1 + \cdots + u_n  = 1\}$ to $[0,\infty]$
is said to be a repelling function if there are constants $0 \le M, C < \infty$ so that

\mn
(i) $\partial \Gamma = \{ u \in \Gamma : \phi(u) = \infty \}$,

\mn
(ii) for each $\delta >0$ there is a $c_\delta>0$ so that $d\phi(u(t))/dt \le -c_\delta$ when $M+\delta < \phi < \infty$,

\mn
(iii)  $\phi$ is convex,

\mn
(iv) $\phi(u) \le C \left( 1 + \sum_{i=1}^k \log^- u_i \right)$.

\mn
The complication of the $c_\delta$ is only needed if $M=0$, i.e., we are trying to prove convergence to the fixed point of the ODE.
If $M>0$ we just check that $d\phi(u(t))/dt \le -c_0$ when $M < \phi < \infty$. This will be the case in most of our examples.

\mn
In addition, we have to assume that there are constants $\gamma_i$ for the ODE so that

\mn
(v) $f_i(u) \ge - \gamma_i u_i$.

\mn
This always holds for replicator equations.

Consider now the PDE $du_i/dt = \Delta u_i + f_i(u)$. From Propositions 1 and 2 of \cite{memoir} we get

\begin{theorem} \label{hammer}
Suppose a repelling function exists, the initial condition $u(0,x)$ is continuous and has 
$u_i(0,x) \ge \eta_i > 0$ when $|x| \le \delta$. There are constants $\kappa>0$
and $t_0 < \infty$ which only depend on $\eta_i$, $\delta$, and $\eta$ so that 
$\phi(u(t,x)) \le M + \eta$ when $|x| \le \kappa t$, $t \ge t_0$.
\end{theorem}

\begin{theorem} \label{nail}
Suppose a repelling function exists for the replicator equation for the modified game $H$. 
If $\ep < \ep_0(G)$ then there is a nontrivial stationary
distribution for the spatial game with matrix ${\bf 1} + \ep^2 G$ that concentrates on $\Omega_0$, the configurations with
infinitely many 1's, 2's, and 3's.
\end{theorem}  

These results were stated in \cite{memoir} for systems with fast stirring. 
Perhaps the easiest way to convince the reader that everything is OK is to use the methods of Durrett and Neuhauser \cite{DN94} as explained in 
Section 9 of Durrett's St.~Flour notes \cite{DStFl}. One begins with the assumption on the limiting PDE

\mn
($\star$) There are constants $A_i < a_i < b_i < B_i$, $l$ and $T$ so that if $u_i(0,x) \in (A_i,B_i)$ when
$x\in [-L,L]^d$ then $u_i(x,T) \in (a_i,b_i)$ when $x\in [-3L,3L]^d$.

\mn
Theorem \ref{hammer} implies that if we take $A_i=\eta_i$ small enough then we can find $a_i$ and $b_i <B_i$ close to 1
so that ($\star$) holds.

Combining this with the convergence of the rescaled system to the PDE, we can define a block construction that guarantees
that if $(m,n)$ is wet, the density of type $i$ is $\ge a_i$ in $2mL + [-3L,3L]^d$ at time $(n+1)T$.
The block event has a probability that goes to 1 as $\ep\to 0$, and the result follows easily. For more details see Chapter 6 of \cite{CDP}.

\subsection{Boundary lemmas} \label{ssec:blemmas}

We will construct our repelling function as a sum of functions associated with the corners and the edges. 
The results we need are developed in Sections 1 and 2 of \cite{memoir}. Since that reference is not easily
available we will give the details here. The first step is create corner functions. In the case considered
the corner $u_1=0$, $u_2=0$ is unstable: $u_1$ and $u_2$ will each increase when they are small and positive.

\begin{lemma} \label{kornf}
Suppose that $\beta_2, \alpha_1>0$. Let $u(t)$ be a positive solution of the replicator equation. 
If $\eta_3$ is small and $g_3(u_1+u_2) = \log^-((u_1+u_2)/\eta_3)$ where $x^- = \max\{-x,0\}$
then $dg_3(u(t))/dt \le -\gamma_3$ when $0 < u_1+u_2 < \eta_3$.
\end{lemma}

\begin{proof} 
To begin the calculation, we recall that the replicator equation is
\beq
\frac{du_i}{dt} = u_i \left( \sum_j H_{ij} u_j - u^T Hu \right)
\label{recrep}
\eeq
where
\beq
u^T H u = u_1u_2(\alpha_3+\beta_3) + u_1u_3(\alpha_2+\beta_2) + u_2u_3(\alpha_1+\beta_1).
\label{xGx}
\eeq
Calculus tells us that when $0 < u_1+u_2 < \eta_3$
$$
\frac{dg_3}{dt} = \frac{-1}{u_1+u_2} \cdot \left( \frac{du_1}{dt} + \frac{du_2}{dt} \right)
$$
Recalling that the game matrix is
$$
H = \begin{pmatrix}
0 & \alpha_3 & \beta_2 \\ \beta_3 & 0 & \alpha_1 \\ \alpha_2 & \beta_1 & 0 
\end{pmatrix}
$$
The last quantity is
$$
\frac{du_1}{dt} + \frac{du_2}{dt} = u_1( \alpha_3 u_2 + \beta_2 u_3 - u^T Hu ) + u_2(\beta_3 u_1 + \alpha_1 u_3 - u^T Hu )
$$
As $u_1,u_2 \to 0$, $u^T H u \to 0$, so if $\eta_3$ is small then when $u_1, u_2 < \eta_3$ we have
$$
\frac{du_1}{dt} + \frac{du_2}{dt} \ge  \frac{u_1\beta_2}{ 2} + \frac{u_2\alpha_1}{2}
$$
and the desired result follows. 
\end{proof}

We will soon be drowning in constants. To make sure that the bad constants don't get worse when we try to improve the good
ones, it is useful to define
\beq
C_G = \sum_i |\alpha_i| + |\beta_i|
\label{CGdef}
\eeq
If $u_i \ge 0$ and $u_1+u_2+u_3 = 1$ then each $u_i \le 1$ so
\beq
|u^tHu| \le C_G
\label{uHubd}
\eeq

The next case to consider is an edge with an attracting fixed point. This is the worst situation.
The function we will construct fails to be decreasing near the end points, but this can be taken care of using corner functions from Lemma \ref{kornf}
or a function from Lemma \ref{dom1}.

\begin{center}
\begin{picture}(180,65)
\put(30,20){\line(1,0){120}}
\put(30,20){\line(2,3){25}}
\put(150,20){\line(-2,3){25}}
\put(20,15){$1$}
\put(155,15){$3$} 
\put(88,18){$\bullet$}
\put(130,15){\vector(-1,0){30}} \put(50,15){\vector(1,0){30}} 
\put(45,20){\line(2,3){8}} \put(135,20){\line(-2,3){8}} \put(38,32){\line(1,0){104}}
\end{picture}
\end{center}

\begin{lemma} \label{afb}
Suppose that $(p_{13},q_{13})$ is an attracting fixed point on the side $u_2=0$
and that the $2$'s can invade this equilibrium. Let $\psi_i(u) = (\delta_2 - u_i)^{+2}$, i.e., the square 
of the positive part. Let $u(t)$ be a positive solution of the replicator equation and let 
$$
h_2(u) = u_1 - p_{13}\log(u_1 + u_2\psi_1(u)) + u_3 - q_{13}\log(u_3 + u_2\psi_3(u)) - \ep \log(u_2)
$$
If $\delta_2 < 1/4$ is chosen small enough then there are positive constants $\ep_2$, $\gamma_2$ so that then when $u_2 < \delta_2$, $\ep \le \ep_2$
$$
\frac{dh_2(u(t))}{dt} \le  \begin{cases} -\gamma_2 & u_1 , u_3 > \delta_2  \\  4C_G  & u_1 \le \delta_2 \\ 4C_G & u_3 \le \delta_2 \end{cases}
$$
\end{lemma}

\begin{itemize}
  \item 
The 1/4 here is to guarantee that the two diamond shaped bad regions do not intersect. We only use one constant to keep the geometry
of our regions simple.
 \item
To explain the definition of $h_2$, first consider
$$
h_2^0(u) = u_1 - p_{13}\log(u_1) + u_3 - q_{13}\log(u_3) - \ep \log(u_2)
$$
To explain the first four terms we note that
$$
\left( u - p_{13}\log u + 1-u - q_{13}\log(1-u) \right)' = \frac{-p_{13}}{u} + \frac{q_{13}}{1-u}
$$
has a minimum when $u=p_{13}$, so it will be decreasing along solutions of the ODE on the boundary.
The logarithmic divergence of the function near $u=0$ and $u=1$ keeps $dh_2/dt$ bounded away from 0
when $|u_1-p_{1,3}|, |u_3-q_{1,3}| \ge \eta > 0$.
The $\ep$ in front of $\log(u_2)$ implies that the impact of this term will be felt only near the boundary fixed
point, where the invadability condition will imply $h_2(u(t))$ is decreasing.
\item
Unfortunately $h^0_2$ is $\infty$ on the sides $u_1=0$ and $u_3=0$. To have $h_2$ infinite only when $u_2=0$ we
add the terms $u_2\psi_i(u)$ inside the logarithms. For $i=1,3$, since $\psi_i(u)=0$ when $u_i \ge \delta_2$ and we have squared
the positive part to make it go to 0 smoothly, the derivative are only changed when $u_i < \delta_2$.
\end{itemize}

\begin{proof}
We begin by computing $dh_2^0/dt$. From \eqref{recrep} and \eqref{xGx}, it follows that
\begin{align*}
\frac{dh_2^0}{dt} & = (u_1-p_{13})(\alpha_3 u_2 + \beta_2 u_3 - u^THu) \\
& + (u_3-q_{13})(\alpha_2 u_1 + \beta_1 u_2 - u^THu) \\
& - \ep (\beta_3 u_1 + \alpha_1 u_3 - u^THu)
\end{align*}
The terms that do not have $u_2$ are
\begin{align*}
& = (u_1-p_{13})[\beta_2 u_3 - u_1u_3(\alpha_2+\beta_2)] \\
& + (u_3 - q_{13})[\alpha_2 u_1 - u_1u_3(\alpha_2+\beta_2)] \\
& -\ep [\beta_3 u_1 + \alpha_1 u_3  - u_1u_3(\alpha_2+\beta_2)]
\end{align*}
Using the fact that $p_{13} = \beta_2/(\alpha_2+\beta_2)$ and $q_{13} = \alpha_2/(\alpha_2+\beta_2)$ we can write the above as
\begin{align}
= & -(\alpha_2+\beta_2)(u_1-p_{13})^2 u_3 -(\alpha_2+\beta_2)(u_3-q_{13})^2 u_1 
\nonumber\\
& -\ep [\beta_3 u_1 + \alpha_1 u_3  - u_1u_3(\alpha_2+\beta_2)]
\label{L2step1}
\end{align}
The sum of the first two terms is bounded away from 0, when $|u_1-p_{13}|, |u_3-q_{13}| \ge \eta_2$, and $u_2 \le 1/2$. 
When $u_1=p_{13}$ and $u_3=q_{13}$ the term in square brackets is $>0$ by the invadability condition \eqref{inv2}:
$$
\beta_3 \frac{\beta_2}{\alpha_2+\beta_2} + \alpha_1 \frac{\alpha_2}{\alpha_2+\beta_2}
 > \frac{\beta_2\alpha_2}{\alpha_2+\beta_2} = u_1u_3(\alpha_2+\beta_2)
$$
From this we see that if we choose $\eta_2$ small then the term in square brackets will be $>0$ when  $|u_1-p_{13}|, |u_3-q_{13}| \le \eta_2$.
Hence if $\ep_2$ is small and $\ep \le \ep_2$ then the expression in \eqref{L2step1} is bounded away from 0 when $u_2 \le 1/2$.

The terms that have $u_2$ are
\begin{align*}
& u_2 [ (u_1-p_{13}) \alpha_3 + (u_3-q_{13}) \beta_1 ] \\
& - \{ (u_1-p_{13}) + (u_3-q_{13})  - \ep \} (  u_1u_2(\alpha_3+\beta_3) + u_2u_3(\alpha_1+\beta_1) )
\end{align*}
The absolute value of the last expression is $\le C_2 u_2$, so we have shown that if $\delta_2$ is small then
$dh_2^0 \le -\gamma_2$ when $u_2 < \delta_2$.

This proves the first of our three conclusions. To prove the other two, it is enough to show
$$
\frac{d}{dt} \left| \log(u_1) - \log(u_1 + u_2\psi_1(u)) \right| \le 2C_G
$$
To do this we note that if $0 < u_1 < \delta_2$ the derivative is
\begin{align*}
\left[ 1 - \frac{u_1}{u_1+u_2\psi_1(u)} \cdot (1 - 2u_2(\delta_2-u_1)) \right] & (\alpha_3 u_2 + \beta_2u_3 - u^T Hu ) \\
 - \frac{u_2 \psi_1(u)}{u_1+u_2\psi_1(u)} \cdot &( \beta_3 u_1 + \alpha_1 u_3 - u^T Hu )
\end{align*}
The two fractions with denominator $u_1+u_2\psi_1(u)$ lie in $(0,1)$, so 
term in square brackets lies in $[-2u_2\delta_2,1]$ and the desired result follows from \eqref{CGdef} and \eqref{uHubd}. 
\end{proof}

The next case to consider is an edge where one strategy dominates the other (in their $2 \times 2$ subgame). 
In the first situation there is no bad region:

\begin{center}
\begin{picture}(180,65)
\put(30,20){\line(1,0){120}}
\put(30,20){\line(2,3){25}}
\put(150,20){\line(-2,3){25}}
\put(20,15){$1$}
\put(155,15){$2$} 
\put(75,15){\vector(1,0){30}} 
 \put(38,32){\line(1,0){104}}
\put(30,30){\vector(2,3){15}} 
\put(0,40){$\alpha_2>0$} 
\put(150,30){\vector(-2,3){15}}
\put(147,40){$\beta_1>0$} 
\end{picture}
\end{center}

\begin{lemma} \label{dom1}
Suppose $2 \gg 1$ $(\alpha_3 < 0 < \beta_3)$, and $\alpha_2, \beta_1>0$.
Let $u(t)$ be a positive solution of the replicator equation and 
let $h_3(u) = u_1 - \ep \log(u_3)$. If $\ep_3$, $\delta_3$, and $\gamma_3$ are chosen small enough then $dh_3(u(t))/dt \le - \gamma_3$ 
when $u_3 < \delta_3$, and $\ep \le \ep_3$. 
\end{lemma}

\begin{proof}
Using \eqref{recrep} we have
$$
\frac{dh_3}{dt}  =  u_1( \alpha_3 u_2 + \beta_2 u_3 - u^THu ) - \ep (\alpha_2 u_1 + \beta_1 u_2 - u^THu )
$$
Consulting  \eqref{xGx}, we see that the terms that do not involve $u_3$ are:
\begin{align}
& \alpha_3 u_1u_2 - u^2_1u_2(\alpha_3+\beta_3) - \ep (\alpha_2 u_1 + \beta_1 u_2) + \ep u_1u_2(\alpha_3+\beta_3) 
\nonumber\\
& = u_1u_2 [\alpha_3 (1-u_1+\ep) - \beta_3 (u_1-\ep) ] - \ep(\alpha_2 u_1 + \beta_1 u_2)
\label{L2step2}
\end{align}
If $\ep$ is chosen small enough then (here we use $\alpha_3 < 0 < \beta_3$)
$$
(1-u_1+\ep) \alpha_3  - \beta_3(u_1-\ep)  =  (1-u_1) \alpha_3 - \beta_3 u_1 + \ep(\alpha_3+\beta_3) \le -\eta_1 < 0
$$
for $u_1 \in [0,1]$. In \eqref{L2step2} this term is multiplied by $u_1u_2$ which vanishes when $u_1$ or $u_2$ is zero
Using $\alpha_2, \beta_1 >0$, the terms that do not involve $u_3$ are 
\beq
\le -\eta_1 u_1u_2 - \ep( \alpha_2 u_1 + \beta_1 u_2) \le -\eta_2 <  0
\label{plb}
\eeq
when $u_3 \le 1/2$. The terms that involve $u_3$ are
$$
u_3 \{ u_1 \beta_2 - (u_1-\ep) [ u_1(\alpha_3+\beta_3) + u_2 (\alpha_1+\beta_1) ] \}
$$
The absolute value of the last expression is $\le C_G u_3$ where $C_G$ is defined in \eqref{CGdef},
and the desired result follows.
\end{proof}

In the next result we reverse the condition $\alpha_2>0$.

\begin{center}
\begin{picture}(180,65)
\put(30,20){\line(1,0){120}}
\put(30,20){\line(2,3){25}}
\put(150,20){\line(-2,3){25}}
\put(20,15){$1$}
\put(155,15){$2$} 
\put(75,15){\vector(1,0){30}} 
\put(45,52){\vector(-2,-3){15}} 
\put(0,40){$\alpha_2\le 0$} 
\put(150,30){\vector(-2,3){15}}
\put(147,40){$\beta_1>0$} 
\put(45,20){\line(2,3){8}} \put(38,32){\line(1,0){104}}
\end{picture}
\end{center}

In this case $h_2(u(t))$ fails to be decreasing near $u_1=1$, but this can be counteracted by using a function from Lemma \ref{dom1} 
or Lemma \ref{dom2}. 

\begin{lemma} \label{dom2}
Suppose $2 \gg 1$ $(\alpha_3 > 0 > \beta_3)$, $\beta_1>0$, and $\alpha_2\le 0$.
Let $u(t)$ be a positive solution of the replicator equation and let $h_3(u) = u_1 - \ep \log(u_3)$.
Given $\delta_2 >0$,  we can pick $\delta_3, \ep_2, \gamma_3$ positive
so that if $\ep \le \ep_2$ and  $u_3 \le \delta_3$, then 
$$
\frac{dh_3(u(t))}{dt} \le 
\begin{cases} 
- \gamma_3 & \hbox{when $u_2 \ge \delta_2$}\\
2C_G \delta_3 & \hbox{when $u_2 \le \delta_2$} 
\end{cases}
$$
where $C_G$ is the constant given in \eqref{CGdef}.
\end{lemma}

\begin{proof}
As in Lemma \ref{dom1}, the terms with $u_3$ are $\le C_G\delta_3$ while from \eqref{L2step2} the terms that do not involve $u_3$ are:
$$
= u_1u_2 [\alpha_3 (1-u_1+\ep) - \beta_3 (u_1-\ep) ] - \ep(\alpha_2 u_1 + \beta_1 u_2)
$$
To bound the first term we note that
$$
\alpha_3 (1-u_1+\ep) - \beta_3 (u_1-\ep) \le (1-u_1) \alpha_3 - \beta_3 u_1 + C_G \ep \le - \eta_3 <0
$$
if $\ep$ is small. This implies, as in the previous proof, that the terms that do not involve $u_3$ are
\beq
\le -\eta_1 u_1u_2 - \ep( \alpha_2 u_1 + \beta_1 u_2)
\label{plb2}
\eeq
but this time $\alpha_2 < 0$. Dropping the negative term $-\ep\beta_1 u_2$ the above
$$
\le u_1 (-\eta_1u_2 -\ep \alpha_2).
$$
By making $\ep$ smaller and then choosing $\delta_3$ small there is a constant $\gamma_3>0$ so that 
$$
u_1 (-\eta_1u_2 -\ep \alpha_2) + C_G\delta_3 \le - \gamma_3 
$$
when $u_2 \ge \delta_2$. For $u_2\le \delta_2$ we note that the quantity in \eqref{plb2} is $\le -\ep\alpha_2$. If $\ep \le \delta_3$ then
$$
-\alpha_2 \ep + C_G \delta_3 \le 2C_G\delta_3
$$
and we have the desired result.
\end{proof}

\subsection{Results for three classes of examples} \label{ssec:Lexist}

\mn
{\bf Example 7.1.} Three stable invadable boundary fixed points.

\begin{center}
\begin{picture}(180,130)
\put(87,125){1}\put(20,25){3}\put(155,25){2}
\put(59,76){$\bullet$}\put(87,27){$\bullet$}\put(116,76){$\bullet$}\put(87,55){$\bullet$} 
\put(65,76){\vector(3,-2){20}} \put(90,32){\vector(0,1){20}} \put(115,76){\vector(-3,-2){20}} 
\put(30,30){\line(2,3){60}}\put(35,48){\vector(2,3){15}}\put(80,115){\vector(-2,-3){15}} 
\put(150,30){\line(-2,3){60}} \put(143,50){\vector(-2,3){15}}\put(103,110){\vector(2,-3){15}} 
\put(30,30){\line(1,0){120}}\put(45,25){\vector(1,0){30}}\put(135,25){\vector(-1,0){30}} 
\put(76,99){\line(1,0){28}} \put(60,30){\line(-2,3){15}} \put(120,30){\line(2,3){15}}
\put(36,30){\line(2,3){57}} \put(144,30){\line(-2,3){57}} \put(34,36){\line(1,0){112}}
\end{picture}
\end{center}

\noindent
Given $k$,  let ${i,j} = \{1,2,3\} - \{k\}$.
Using Lemma \ref{kornf} we can defined corner functions $g_k$ for each corner, which 
$dg_k/dt \le -\gamma_k$ when $u_i + u_j < \eta_k$. Once this is done 
we use Lemma \ref{afb} to construct functions $h_i$ for each edge with associated constants $\delta_i$, which we all
chose all to be equal to $\delta$, and so that $2\delta < \eta_k$ for each $k$. This guarantees that the bad regions 
for the functions $h_i$ which are the small diamonds in the corners
are inside the good regions for some $g_j$. Now pick $M_i$ large enough
so that $\{ h_i > M_i \} \subset \{ u_i < \delta \}$ and let $\bar h_i = \max\{ h_i, M_i\}$.
If $\pi_i$ are chosen small enough then
$$
\sum_i g_i + \sum_j \pi_j \bar h_j 
$$
is a repelling function.

\mn
{\bf Example 7.2.} Two stable invadable boundary fixed points

\begin{center}
\begin{picture}(180,130)
\put(87,125){1}\put(20,25){3}\put(155,25){2}
\put(59,76){$\bullet$}\put(87,27){$\bullet$}
\put(87,55){$\bullet$} \put(65,76){\vector(3,-2){20}} \put(90,32){\vector(0,1){20}} 
\put(30,30){\line(2,3){60}}\put(35,48){\vector(2,3){15}}\put(80,115){\vector(-2,-3){15}} 
\put(150,30){\line(-2,3){60}}\put(135,60){\vector(-2,3){20}} 
\put(30,30){\line(1,0){120}}\put(45,25){\vector(1,0){30}}\put(135,25){\vector(-1,0){30}} 
\put(60,30){\line(-2,3){15}} \put(135,30){\line(-2,3){52}}
\put(36,30){\line(2,3){57}} \put(34,36){\line(1,0){112}} 
\end{picture}
\end{center}

We can use Lemma \ref{kornf} on the corner $u_3=1$ to produce $g_3$ and Lemma \ref{dom1} to the side $u_3=0$ to get $h_3$.
Pick $M_3$ so that $\{ h_3 > M_3 \} \subset \{ u_3 < \delta_{3} \}$. 
We now apply Lemma \ref{afb} to the sides $u_1=0$ and $u_2=0$ with the associated constants
$\delta_{1}$ and $\delta_{2}$ chosen so that the bad regions for these functions are inside the regions
where $g_3$ and $h_3$ are strictly decreasing.  Now pick $M_i$, $i=1,2$ large enough
so that $\{ h_i > M_i \} \subset \{ u_i < \delta_{i} \}$. If $\pi_1$ and $\pi_2$ are chosen small enough then
$$
g_3 + \bar h_3 + \pi_1 \bar h_1 + \pi_2 \bar h_2
$$
is a repelling function.

\mn
{\bf Example 7.3.} One stable invadable fixed point.

\begin{center}
\begin{picture}(180,150)
\put(30,30){\line(1,0){120}}\put(30,30){\line(2,3){60}}\put(150,30){\line(-2,3){60}}
\put(90,125){$1$}\put(20,25){$3$}\put(155,25){$2$} 
\put(59,76){$\bullet$}
\put(43,58){\vector(2,3){10}}\put(75,105){\vector(-2,-3){10}} 
\put(87,55){$\bullet$} \put(65,76){\vector(3,-2){20}} 
\put(115,95){\vector(2,-3){20}} 
\put(105,25){\vector(-1,0){30}} 
\put(135,30){\line(-2,3){52}}
\put(36,39){\line(1,0){108}}
\put(36,30){\line(2,3){57}} 
\end{picture}
\end{center}

Apply Lemma \ref{dom1} to the side $u_3=0$ to get $h_3$.
Pick $M_3$ so that $\{ h_3 > M_3 \} \subset \{ u_3 < \delta_{3} \}$. 
Use Lemma \ref{dom2} on the side $u_1=0$ to get $h_1$ and choose $\delta_1$
small enough so that the bad region for $h_1$ is inside the region where $h_3$
is strictly decreasing. Pick $M_1$ so that $\{ h_1 > M_1 \} \subset \{ u_1 < \delta_1 \}$.
Now apply Lemma \ref{afb} to the side $u_2=0$ and pick $\delta_2$ so that the bad regions for
$h_2$ are inside good regions for $h_1$ and $h_3$.
Pick $M_2$ so that $\{ h_2 > M_2 \} \subset \{ u_2 < \delta_2 \}$.
If we choose $\pi_1$ small and then choose $\pi_2$ small then
$$
\bar h_3 + \pi_1 \bar h_1 + \pi_2 \bar h_2
$$
is a repelling function. 

\mn
Unfortunately we are not able to prove a result for the rock-paper-scissors case.

\subsection{Almost constant sum games} \label{ssec:acs}

The results in the previous section prove the existence of a stationary distribution but does not provide much information
about the frequencies of the three strategies. If we can find a Lyapunov function that is strictly decreasing
over the whole domain then we can conclude that when $\ep$ is small
the frequencies of strategies in the spatial game are almost the same as those of the equilibrium in the modified replicator equation.
Generalizing the calculation from the proof of Theorem \ref{HSrps} we can prove:

\begin{theorem} \label{Limpconv} 
Suppose that the three strategy game $H$ has (i) zeros on the diagonal, (ii) an interior equilibrium $\rho$, and that $H$ is almost constant sum: $H_{ij}+H_{ji} = \gamma + \eta_{ij}$  with $\gamma > 0$ and $\max_{i,j} |\eta_{i,j}| < \gamma/2$. Then $V(u) = \sum_i u_i - \rho_i \log u_i$
is a repelling function with $M=0$, i.e., it is always decreasing not just near the boundary. This implies that there is coexistence and that 
for any $\delta>0$ if $\ep < \ep_0(\delta)$ and $\mu$ is any stationary distribution concentrating on configurations with infinitely many 1's, 2's and 3's
we have
$$
\sup_x |\mu(\xi(x) = i) - \rho_i| < \delta
$$
\end{theorem} 

\noindent
Note that due to the formula for the perturbation \eqref{pertm} $G$ is almost constant sum if and only if $H$ is.

\begin{proof}
If $u(t)$ evolves according to replicator equation for $H$ then
$$
\frac{dV(u)}{dt} = \sum_i (u_i-\rho_i) \left( \sum_j H_{ij} u_j - \sum_{k,j} u_k H_{kj} u_j \right)
$$
Since $\sum_i (u_i-\rho_i)=0$ and $\sum_j H_{ij} \rho_j$ is constant in $i$,
\begin{align*}
\frac{dV(u)}{dt} & = \sum_{i,j} (u_i-\rho_i) H_{i,j} u_j \\
& =  \sum_{i,j} (u_i-\rho_i) H_{i,j} (u_j-\rho_j) =  \sum_{i<j} (u_i-\rho_i) [H_{i,j}+H_{j,i}] (u_j-\rho_j)
\end{align*}
Since $0 = (\sum_i [u_i - \rho_i] )^2$ we have
$$
\sum_{i<j} (u_i-\rho_i)\gamma (u_j-\rho_j) = -\frac{\gamma}{2} \sum_i (u_i-\rho_i)^2
$$
To bound the remaining piece we note that $|ab| \le (a^2+b^2)/2$ so
$$  
\sum_{i<j} |\eta_{ij}| \cdot |u_i-\rho_i| \cdot |u_j-\rho_j| \le \max_{ij} |\eta_{ij}| \sum_i (u_i-\rho_i)^2
$$
so if $\max_{ij} |\eta_{ij}| < \gamma/2$ then $V$ is a convex Lyapunov function. It is not hard to checking
that the other conditions to be a repelling function are satisfied.  

Coexistence now follows from Theorem \ref{nail} for the reasons indicated just after its statement, but now we must take
$a_i =\rho_i - \delta$ and $b_i = \rho_i+\delta$, and work a little harder as in Chapter 6 of \cite{CDP} to make
sure the regions controlled by the block construction cover most of the space.
\end{proof}

\subsection{Tarnita's formula} \label{ssec:TF2}

Tarnita, Wage and Nowak \cite{TWN} have generalized Theorem \ref{TF2} to games with $n>2$ strategies.
Theorem \ref{TF2} shows that for two strategy games, their formula for games in a finite population with weak selection 
agrees with our analysis of games on the infinite lattice with small selection. The aim of this section is
to show that in the special case of almost constant sum games the two conditions are different.
Their formula is linear in the coefficients of the game matrix while ours is quadratic.

In \cite{TWN} a strategy is said to be favored by selection if, in the presence of weak selection, its frequency 
is $> 1/n$. Their main result is that strategy $k$ is favored by selection if 
\beq
(\sigma_1 a_{k,k} + \bar a_{k,*} - \bar a_{*,k} - \sigma_1 \bar a_{*,*}) + \sigma_2 ( \bar a_{k,*} - \bar a) >0
\label{TFn}
\eeq
The parameters $\sigma_1$ and $\sigma_2$ depend on the population structure, the update rule and the
mutation rate, but they do not depend on the number of strategies or on the entries $a_{ij}$ of the payoff matrix.
The coefficients are
\begin{align*}
\bar a_{*,*}  = (1/n) \sum_{m=1}^n a_{m,m} & \qquad \bar a_{*,k}  = (1/n) \sum_{m=1}^n a_{m,k}\\
\bar a_{k,*}  = (1/n) \sum_{m=1}^n a_{k,m}& \qquad \bar a  = (1/n^2) \sum_{\ell=1}^n \sum_{m=1}^n a_{\ell,m}
\end{align*}
so the condition in \eqref{TFn} is linear in the coefficients.

\mn
{\bf Remark.} A paper that is in preparation will show that for Birth-Death and Death-Birth updating,
strategy $k$ is favored by selection if $\phi_k(1/n, \ldots, 1/n) > 0$.

\medskip
The almost constant sum games in the previous section are an open set of games for which we can show that
in the small selection limit, the equilibrium frequencies of the spatial game converges to those of the
equilibrium frequencies of the replicator equation for the
modified game. In this situation strategy 1 will be favored by selection
if when we use \eqref{Geq} on the modified game
\beq
2\rho_1 > \rho_2 + \rho_3
\label{PTFn}
\eeq
This quantity involves our structural coefficient $\theta$ which $p_2/p_1$ for Birth-Death updating and
$$
\bar \theta = \frac{\bar p_2 - p(v_1|v_2)/\kappa}{\bar p_1}
$$
for Death-Birth updating but due to the formulas for the $\rho_i$  \eqref{Geq} the condition \eqref{PTFn} is 
quadratic in the entries in the game matrix (once we multiply each side of the equation by the denominator $D$.

\section{Analysis of three strategy cancer games} \label{sec:cancer}

\subsection{Multiple myeloma}

We begin with this example since the analysis of the original game is fairly simple and space 
can allow coexistence, which is not possible in the replicator equation for the original game.
As explained in Example \ref{bone}, the players are osteoclasts ($OC$), osteoblasts ($OB$), 
and multiple myeloma ($MM$) cells, and the payoff matrix is

\begin{center}
\begin{tabular}{lccc}
& $OC$ & $OB$ & $MM$ \\
$OC$  & 0 & $a$ & $b$\\
$OB$ & $e$ & $0$ & $-d$ \\
$MM$ & $c$ & $0$ & $0$
\end{tabular}
\end{center}

To study the properties of the game we begin with the two strategy games it contains.

\mn
{\bf 1. OC vs. OB.} $(a/(a+e),e/(a+e))$ is a mixed strategy equilibrium. Since $a,e>0$ it is
attracting (on the $OC-OB$ edge).

\mn
{\bf 2. OC vs. MM.} $(b/(b+c),c/(b+c))$ is a mixed strategy equilibrium. Since $b,c$ are
 positive it is attracting (on the $OC-OM$ edge).

\mn
{\bf 3. OB vs. MM.} $MM$ dominates $OB$.

\mn
{\bf Invadability.} As in Section \ref{sec:3sgode}, to determine the behavior of the three strategy game, we will see when the third strategy
can invade the other two in equilibrium. Here and what follows it will be convenient to also refer to strategies by number:
$1=OC$, $2=OB$, $3=MM$.

\mn
In the $OC,OB$ equilibrium, $F_1=F_2 = ae/(a+e)$ while $F_{3}=ca/(a+e)$ so $MM$ can invade 
if $\beta = c/e > 1$. 

\mn
In the $OC,MM$ equilibrium, the fitnesses $F_1=F_3=bc/(b+c)$, while $F_2 = (eb-dc)/(b+c)$,
so $OB$ can invade if $eb-dc > bc$. Letting $\delta = dc/be$, we can write this as
$$
1 - \frac{dc}{be} > \frac{c}{e} \quad\hbox{or}\quad \beta + \delta < 1
$$

\mn
Since $\delta>0$, these two conditions $\beta > 1$ and $\beta+\delta < 1$ cannot be satisfied at the same time. Let
$$
x_1 = (a/(a+e),e/(a+e),0) \qquad x_2 = (b/(b+c),0, c/(b+c))
$$
Thus we have three cases

\mn
{\bf Case 1.} $\beta > 1$, $x_1$ can be invaded, while $x_2$ cannot. This is Example 7.2A, so the replicator equation converges to $x_1$.

\mn
{\bf Case 2.} $\beta<1$, $\beta+\delta>1$: Neither $x_1$ nor $x_2$ can be invaded. This is Example 7.2B,
so there is an interior fixed point that is a saddle point, and the replicator equation exhibits bistability. We leave it to the
reader to verify that the interior equilibrium is:
$$
\rho_1 = \frac{\delta}{D} \qquad \rho_2 = \frac{\beta(\delta+\beta-1)}{D} \qquad \rho_3 = \frac{1-\beta}{D}
$$
where $D$ is the sum of the numerators.

\mn
{\bf Case 3.} $\beta+\delta<1$, $x_1$ cannot be invaded but $x_2$ can. This is Example 7.2A, so the replicator equation converges to $x_2$.

\mn
{\bf The modified game} has entries

\begin{center}
\begin{tabular}{lccc}
& $OC$ & $OB$ & $MM$ \\
$OC$  & 0 & $(1+\theta)a - \theta e$ & $(1+\theta)b - \theta c$\\
$OB$ & $(1+\theta)e-\theta a$ & $0$ & $-(1+\theta)d$ \\
$MM$ & $(1+\theta)c - \theta b$ & $\theta d$ & $0$
\end{tabular}
\end{center}

\noindent
The modification of the game does not change the sign of $G_{2,3}$ but it puts a positive entry in $G_{3,2}$.
It may also change the signs of one or two of the other four non-zero entries. Noting that
$$
G_{12}<0 \quad \hbox{if}\quad e > (1+\theta)a/\theta \quad \hbox{while} \quad
G_{21}<0 \quad \hbox{if}\quad e < \theta a/(1+\theta) 
$$
If one of these two entries is negative the other one is positive. The same holds for $G_{13}$ and $G_{31}$
To simplify formulas, we let

\begin{center}
\begin{tabular}{lccc}
& $OC$ & $OB$ & $MM$ \\
$OC$  & 0 & $A$ & $B$\\
$OB$ & $E$ & $0$ & $-D$ \\
$MM$ & $C$ & $F$ & $0$
\end{tabular}
\end{center}

\mn
We always have $D,F>0$. Suppose to begin that $A,B,C,E>0$. The condition for $OB$ to invade the $OC,MM$ equilibrium,
when written in the new notation, is unchanged
$$
EB-DC > BC
$$
In the $OC,OB$ equilibrium, $F_1=F_2 = AE/(A+E)$ while $F_{3}=CA/(A+E) + FE/(A+E)$ so $MM$ can invade if 
$$
\frac{(C-E)A}{A+E} +\frac{FE}{A+E} >0
$$  
i.e., $(C-E)A + FE > 0$ or $C/E +FE > 1$, which is no longer inconsistent with $C/E < 1 - DC/BE$, so we have
a new possibility.

\mn
{\bf Case 4.} Both $x_i$ can be invaded if
$$
1 - \frac{DC}{BE} > \frac{C}{E} > 1 - \frac{F}{A}
$$
We are in the situation of Example 7.2 so by results for that example there is coexistence in the spatial game.

\begin{center}
\begin{picture}(180,150)
\put(30,30){\line(1,0){120}}\put(30,30){\line(2,3){60}}\put(150,30){\line(-2,3){60}}
\put(80,125){$OB$}\put(0,25){$OC$}\put(155,25){$MM$}
\put(59,76){$\bullet$} \put(46,78){$x_1$}
\put(87,27){$\bullet$} \put(83,19){$x_2$}
\put(87,55){$\bullet$} 
\put(65,75){\vector(3,-2){15}}\put(90,35){\vector(0,1){15}} 
\put(40,55){\vector(2,3){10}}\put(80,115){\vector(-2,-3){10}} 
\put(45,25){\vector(1,0){20}}\put(135,25){\vector(-1,0){20}} 
\put(115,95){\vector(2,-3){20}} 
\end{picture}
\end{center}

\mn
The next thing to consider is what happens if one of the entries $A, B,C,E < 0$. Recall that $A<0$ implies $E>0$,
$E<0$ implies $A>0$, $B<0$ implies $C>0$, and  $C<0$ implies $B>0$.

\mn
{\bf Case 5A.} If $A<0$ (so $OB \gg OC$) and $x_2$ can  be invaded, then we are in the situation of  Example 7.3
and there is coexistence in the spatial game.

\mn
{\bf Case 5B.} If $E<0$ then $OC \gg OB$. This is Example 7.3C so the $OB$'s die out and $x_2$ is the limit for the replicator equation. 

\mn
{\bf Case 6A.} If $C < 0$ (so $OC \gg MM$) and $x_1$ can be invaded, then we are in the situation of  Example 7.3
and there is coexistence in the spatial game.

\mn
{\bf Case 6B.} If $B<0$ then $MM \gg OC$, and $MM \gg OB$.  If $MM$ can invade the $OC, OB$ equilibrium
then $MM$ takes over (Example 7.3A.) If not then there is an interior saddle point and bistability. (Example 7.3B.)
However, we cannot prove the conclusion in either case. 

\medskip
One can also change the signs of two entries. To avoid case 6B, where nothing changes if $A$ or $E$ is negative, we have to suppose that $C<0$.

\mn
{\bf Case 7A.} If $C<0$ and $A<0$, we get a rock-paper scissors game, as in Example 7.4.  Let
$$
\Delta = \beta_1\beta_2\beta_3+\alpha_1\alpha_2\alpha_3 = -DCA + FEB
$$
If $\Delta>0$ then solutions of the replicator equation will spiral into a fixed point, so we expect to have coexistence in the spatial game.
However since we have no result for Example 7.4, we cannot prove this. 

\mn
{\bf Case 7B.} If $C<0$ and $E<0$ then $OC \gg MM, OB$. Again we are in the situation of Example 7.4A, so $OC$ takes over in the replicator 
equation, but we cannot prove this occurs in the spatial model.

\subsection{Three species chemical competition}

As explained in Example \ref{Tom}, the payoff matrix is
$$
\begin{matrix} 
 & P & R & S \\
P & z-e-f+g & z-e & z-e+g \\
R & z-h & z-h & z-h \\
S & z-f & z & z
\end{matrix}
$$
Here P's produce a toxin, R's are resistant to it, S's are sensitive wild-type cells, and we have supposed $g>e$
 i.e., the benefit of producing the toxin outweighs the cost.
Subtracting a constant from each column to make the diagonal entries 0. 
$$
\begin{matrix} 
 & P & R & S \\
P & 0 & h-e & g-e \\
R & e+f-g-h & 0 & -h \\
S & e-g & h & 0
\end{matrix}
$$ 
To simplify the algebra we will consider a more general game

\begin{center}
\begin{tabular}{lccc}
& $P$ & $R$ & $S$ \\
$P$  & 0 & $a$ & $c$\\
$R$ & $b$ & $0$ & $-d$ \\
$S$ & $-c$ & $d$ & $0$
\end{tabular}
\end{center}

\noindent
thinking primarily about the special case
$$
\begin{matrix}
a = (1+\theta)(h-e) - \theta (e-h+f-g) & c = (1+2\theta)(g-e) \\
b = (1+\theta)(e-h+f-g) - \theta(h-e) & d = (1+2\theta) h
\end{matrix}
$$
Here $c,d>0$ since $g>e$ and $h>0$, but $a$ and $b$ can have either sign.

To investigate the game we begin by considering the two strategy subgames it contains.

\mn
{\bf R vs. S.} $d >0$ so $S$ dominates $R$.

\mn
{\bf P vs. S.} $c>0$, so $P$ dominates $S$.

\mn
{\bf P vs. R.} All four cases can occur. We now consider them one at a time.

\mn
{\bf Case 1.} If $a>0$ and $b < 0$ then $P$ dominates $R$. Since $P$ also dominates $S$, this is Example 7.4A.
$P$'s take over in the replicator equation. We conjecture but do not know how to prove that this happens in
the spatial model.

\noindent
{\bf Case 2.} If $a<0$ and $b>0$, then $R$ dominates $P$ and we have a cyclic relationship between
the competitors. We are in the situation of Example 7.4. Let
$$
\Delta = \beta_1\beta_2\beta_3+\alpha_1\alpha_2\alpha_3 = dcb - (-d)(-c)a = dc(b-a)
$$
If $\Delta>0$ solutions of the replicator equation will spiral into a fixed point, so we expect (but cannot prove) coexistence in the spatial game.

\begin{center}
\begin{picture}(180,150)
\put(20,120){Case 1.}
\put(30,30){\line(1,0){120}}\put(30,30){\line(2,3){60}}\put(150,30){\line(-2,3){60}}
\put(90,125){$P$}\put(15,25){$R$}\put(155,25){$S$} 
\put(45,65){\vector(2,3){20}} 
\put(135,65){\vector(-2,3){20}} 
\put(75,25){\vector(1,0){30}} 
\end{picture}
\begin{picture}(180,150)
\put(20,120){Case 2.}
\put(30,30){\line(1,0){120}}\put(30,30){\line(2,3){60}}\put(150,30){\line(-2,3){60}}
\put(90,125){$P$}\put(15,25){$R$}\put(155,25){$S$} 
\put(65,95){\vector(-2,-3){20}} 
\put(135,65){\vector(-2,3){20}} 
\put(75,25){\vector(1,0){30}} 
\end{picture}
\end{center}

\mn
If neither of the two strategies $P$ and $R$ dominates the other, we will have a mixed strategy equilibrium with 
$$
p_{12} = \frac{a}{a+b} \qquad q_{12} = \frac{b}{a+b}
$$
Here and in what follows $p_{ij}$ and $q_{ij}$ denote the equilibrium frequencies of the first and second strategies
in the $i,j$ subgame. Using the facts about $2\times 2$ games recalled right after \eqref{Gmse}:

\mn
{\bf Case 3.} If $a>0$ and $b>0$ then $(p_{12},q_{12})$ is attracting on the $P,R$ edge.

\mn
{\bf Case 4.} $a<0$ and $b<0$ then $(p_{12},q_{12})$ is repelling on the $P,R$ edge.

\mn
Let $x_1 = (a/a+b, b/(a+b),0)$. Each of the two cases breaks into subcases depending on whether $x_1$ can be invaded (subcase $A$) or $B$ not (subcase $B$).  

\mn
{\bf Case 3A.} The interior equilibrium will attracting and we are in the situation of Example 7.3, and there will be coexistence in the spatial game.

\mn
{\bf Case 3B.} The fixed point on the $R,P$ edge will attracting, and we are in the situation of Example 7.3D. The $S$'s should die out in the spatial game and we should have an equilibrium consisting only of $R$'s and $P$'s but we do not know how to prove that.

\begin{center}
\begin{picture}(180,150)
\put(20,120){Case 3A.}
\put(30,30){\line(1,0){120}}\put(30,30){\line(2,3){60}}\put(150,30){\line(-2,3){60}}
\put(90,125){$P$}\put(15,25){$R$}\put(155,25){$S$} 
\put(59,76){$\bullet$}
\put(35,48){\vector(2,3){15}}\put(80,115){\vector(-2,-3){15}} 
\put(135,65){\vector(-2,3){20}} 
\put(75,25){\vector(1,0){30}} 
\put(87,55){$\bullet$} \put(65,76){\vector(3,-2){20}} 
\put(28,28){$\bullet$}
\end{picture}
\begin{picture}(180,150)
\put(20,120){Case 3B.}
\put(30,30){\line(1,0){120}}\put(30,30){\line(2,3){60}}\put(150,30){\line(-2,3){60}}
\put(90,125){$P$}\put(15,25){$R$}\put(155,25){$S$} 
\put(59,76){$\bullet$}
\put(35,48){\vector(2,3){15}}\put(80,115){\vector(-2,-3){15}} 
\put(135,65){\vector(-2,3){20}} 
\put(75,25){\vector(1,0){30}} 
\put(86,62){\vector(-3,+2){21}} 
\put(28,28){$\bullet$}
\end{picture}
\end{center}

\medskip
{\bf Interior fixed point.} Our theoretical results tell us that there will be attracting interior fixed points in Cases 2 and 3A. The show that one does not exist in Case 3B, we use \eqref{Geq} to compute
\begin{align}
\rho_1 & = (dc - ad + d^2)/D = d(-a+c+d)/D 
\nonumber\\
\rho_2 & = (cb + cd + c^2)/D = c(b+c+d)/D 
\label{cceq}\\
\rho_3 & = (bd-ac-ab)/D
\nonumber
\end{align}
If $S$ cannot invade the $P,R$ equilibrium then
\beq
\frac{ab}{a+b} > (-c) \frac{a}{a+b} + d \frac{b}{a+b}
\label{3noi12}
\eeq
so the numerator of $\rho_3$ is negative. If all three numerators were negative then there would be an equilibrium 
but the numerator of $\rho_2 > 0$, so there is not.

As a check of our computation of the fixed point $\rho$, we note that when $a=h-e$, $b=f-g+e-h$, $c=g-e$, and $d=h$ we have
$$
\rho_1 = \frac{hg}{D} \qquad \rho_2 = \frac{(g-e)f}{D} \qquad \rho_3 = \frac{ef-hg}{D}
$$
so $D=gf$ and the interior fixed point is 
\beq
\rho_1 = \frac{h}{f} \qquad \rho_2 = 1 - \frac{e}{g} \qquad \rho_3 = \frac{e}{g} - \frac{h}{f}
\label{3scceq}
\eeq
in agreement with page 1496 of \cite{Tom}.

\mn
In {\bf Case 4A}, $S$ can invade the $P,R$ equilibrium, so looking at \eqref{3noi12}, we see that the numerator of $\rho_3$ is negative. The numerator of $\rho_1$ is positive so there can be no equilibrium. 

\mn
In {\bf Case 4B}, $S$ cannot invade so the numerator of $\rho_3$ is positive. Since $a<0$ and $c,d >0$ the numerator of $\rho_1$ is positive. In general we might have $b+c+d<0$ but in our special case
$$
b+c+d = (1+2\theta) g + (1+\theta) (f-g) = \theta g + (1+\theta)f
$$ 
In this case, there is an interior equilibrium, but it is unstable. Since $P$ is locally attracting, we should have convergence to $P$ in the spatial game but cannot prove it.

\begin{center}
\begin{picture}(180,150)
\put(20,120){Case 4A.}
\put(30,30){\line(1,0){120}}\put(30,30){\line(2,3){60}}\put(150,30){\line(-2,3){60}}
\put(90,125){$P$}\put(15,25){$R$}\put(155,25){$S$} 
\put(59,76){$\bullet$}
\put(53,73){\vector(-2,-3){15}}\put(65,90){\vector(2,3){15}} 
\put(135,65){\vector(-2,3){20}} 
\put(75,25){\vector(1,0){30}} 
\put(65,76){\vector(3,-2){20}} 
\put(88,118){$\bullet$}
\end{picture}
\begin{picture}(180,150)
\put(20,120){Case 4B.}
\put(30,30){\line(1,0){120}}\put(30,30){\line(2,3){60}}\put(150,30){\line(-2,3){60}}
\put(90,125){$P$}\put(15,25){$R$}\put(155,25){$S$} 
\put(59,76){$\bullet$}\put(87,55){$\bullet$} 
\put(90,65){\vector(0,1){25}} 
\put(53,73){\vector(-2,-3){15}}\put(65,90){\vector(2,3){15}} 
\put(135,65){\vector(-2,3){20}} 
\put(75,25){\vector(1,0){30}} 
\put(86,62){\vector(-3,+2){21}} 
\put(88,118){$\bullet$}
\end{picture}
\end{center}

\subsection{Glycolytic phenotype}

As indicated in Example \ref{Glyco}, the three strategies are $AG =$ autonomous growth, $INV =$ invasive, 
and $GLY =$ glycolytic phenotype, while the payoff matrix is:

\begin{center}
\begin{tabular}{cccc}
& $AG$ & $INV$ & $GLY$ \\
$AG$  & $\frac{1}{2}$ & 1 & $\frac{1}{2} - n$\\
\nass
$INV$ & $1-c$ & $1-\frac{c}{2}$ & $1-c$ \\
\nass
$GLY$ & $\frac{1}{2} + n - k$ & $1-k$ & $\frac{1}{2} - k$
\end{tabular}
\end{center}

\noindent
As noted earlier, if we subtract a constant from each column to make the diagonal zero,
this will not change the properties of the replicator equation.

\begin{center}
\begin{tabular}{cccc}
& $AG$ & $INV$ & $GLY$ \\
$AG$  & 0 & $\frac{c}{2}$ & $k - n$\\
\nass
$INV$ & $\frac{1}{2}-c$ & 0 & $\frac{1}{2}-c+k$ \\
\nass
$GLY$ & $n - k$ & $\frac{c}{2}-k$ & 0
\end{tabular}
\end{center}

\noindent
Again we simplify by generalizing

\begin{center}
\begin{tabular}{lccc}
& $AG$ & $INV$ & $GLY$ \\
$AG$  & 0 & $a$ & $d$\\
$INV$ & $b$ & $0$ & $e$ \\
$GLY$ & $-d$ & $f$ & $0$
\end{tabular}
\end{center}

\mn
where the constants are given by
$$
\begin{matrix}
a = (1+\theta)(c/2) - \theta(1/2-c) & e = (1+\theta)(1/2-c+k) - \theta(c/2-k) \\
b = (1+\theta)(1/2-c) - \theta(c/2) & f = (1+\theta)(c/2-k) - \theta(1/2-c+k) \\
d = (1+2\theta)(k-n)
\end{matrix}
$$
We will suppose $c < 1/2$ so $a,b,e$ are positive when $\theta=0$. Note that
$$
f<0 \quad\hbox{if}\quad c/2-k < 0 \quad\hbox{or} \quad c/2-k >0 \quad\hbox{and}\quad 1/2-c+k > \frac{1+\theta}{\theta}(c/2-k)
$$
so $f<0$ implies $e>0$.

\mn
{\bf 1. AG vs INV.} If $a,b>0$ there is a mixed strategy equilibrium
$$
p_{12} = \frac{b}{a+b} \qquad q_{12} = \frac{a}{a+b}.
$$
which is attracting on the $AG, INV$ edge.

\mn
{\bf 2. INV vs GLY.} Suppose for the moment that $e>0$. If $f<0$ then INV dominates GLY.
If $f>0$ then there is a fixed point
$$
p_{23} = \frac{e}{e+f} \qquad q_{23} = \frac{f}{e+f}.
$$
which is attracting on the $INV,GLY$ edge.

\mn
{\bf 3. AG vs GLY.} If $d>0$ then $AG \gg GLY$; if $d<0$ then $GLY \gg AG$.

\medskip
In the last two examples we examined all of the possible cases. In this one and the next,
we will only consider games that lead to attracting interior equilibria
In the original game there are two possibilities. 
In both we need $GLY$ to be able to invade the $AG, INV$ equilibrium.

\mn
Case 1.  $f>0$ so there is an $INV,GLY$ equilibrium, which we suppose can be invaded by $AG$. The $AG,GLY$ edge is
blank because it does not matter which one dominates the other. This corresponds to
Example 7.2 so there is coexistence in the spatial game.

\mn
Case 2. $f<0$. In this case we need $d<0$ so $GLY \gg AG$.  This corresponds to
Example 7.2 so there is coexistence in the spatial game.

\begin{center}
\begin{picture}(180,150)
\put(20,120){Case 1.}
\put(30,30){\line(1,0){120}}\put(30,30){\line(2,3){60}}\put(150,30){\line(-2,3){60}}
\put(80,125){$AG$}\put(0,25){$INV$}\put(155,25){$GLY$}
\put(59,76){$\bullet$}\put(87,27){$\bullet$}\put(87,55){$\bullet$} 
\put(65,75){\vector(3,-2){15}}\put(90,35){\vector(0,1){15}} 
\put(35,48){\vector(2,3){15}}\put(80,115){\vector(-2,-3){15}} 
\put(45,25){\vector(1,0){30}}\put(135,25){\vector(-1,0){30}} 
\end{picture}
\qquad
\begin{picture}(180,150)
\put(20,120){Case 2.}
\put(30,30){\line(1,0){120}}\put(30,30){\line(2,3){60}}\put(150,30){\line(-2,3){60}}
\put(80,125){$AG$}\put(0,25){$INV$}\put(155,25){$GLY$}
\put(59,76){$\bullet$}\put(87,55){$\bullet$}
\put(65,75){\vector(3,-2){15}} 
\put(35,48){\vector(2,3){15}}\put(80,115){\vector(-2,-3){15}} 
\put(105,25){\vector(-1,0){30}} 
\put(115,95){\vector(2,-3){20}} 
\end{picture}
\end{center}

There are other possibilities for coexistence in the spatial game. 

\mn
Case 3A. If $a<0$ (and hence $b>0$) the $INV \gg AG$. In this case if $d<0$, $AG$ dominates $GLY$
and we are in the case of Example 7.3.

\mn
Case 3B. If $b<0$ (and hence $a>0$) then $AG \gg INV$. In this case if $d>0$, $GLY$ dominates $AG$
and we are in the case of Example 7.3.

\mn
Case 4. If $f<0$ (and hence $e>0$) then $INV \gg GLY$, if $b<0$ (and hence $a>0$) then $AG \gg INV$.
If $d>0$  then $GLY \gg AG$. So under these choices of sign, we have a rock-paper scissors relationship,
as in Example 7.4.  Let
$$
\Delta = \beta_1\beta_2\beta_3+\alpha_1\alpha_2\alpha_3 = fdb - e(-d)a = d(fb+ea)
$$
If $\Delta>0$, which is always true under these choices of sign,  solutions of the replicator equation will spiral 
into a fixed point, so we will have coexistence in the spatial game.

\subsection{Tumor-stroma interactions} 

As explained in Example \ref{Stroma} there are
three strategies $S =$ stromal cells, $I =$ cells that have become independent of the micro-environment,
and $D = $ cells that remain dependent on the microenviornmnet.
The payoff matrix is:

\begin{center}
\begin{tabular}{lccc}
& $S$ & $D$ & $I$ \\
$S$ & 0 & $\alpha$ & 0 \\
$D$ & $1+\alpha-\beta$ & $1-2\beta$ & $1-\beta + \rho$ \\
\nass
$I$ & $1 - \gamma$  & $1 - \gamma$ & $1 - \gamma$
\end{tabular}
\end{center}

\mn
Recall that we assume $\beta<1$, $\gamma<1$.
Subtracting a constant from the second and third columns to make the diagonals 0 the game becomes

\begin{center}
\begin{tabular}{lccc}
& $S$ & $D$ & $I$ \\
$S$ & 0 & $\alpha+2\beta-1$ & $\gamma-1$ \\
$D$ & $1+\alpha-\beta$ & 0 & $\gamma+ \rho -\beta$ \\
\nass
$I$ & $1 - \gamma$  & $2\beta - \gamma$ & 0
\end{tabular}
\end{center}

\mn
Generalizing we have

\begin{center}
\begin{tabular}{lccc}
& $S$ & $D$ & $I$ \\
$S$ & 0 & $a$ & $-c$ \\
$D$ & $b$ & 0 & $d$ \\
\nass
$I$ & $c$  & $e$ & 0
\end{tabular}
\end{center}

\noindent
where the constants are
$$
\begin{matrix}
a = (1+\theta)(\alpha+2\beta-1) - \theta(1+\alpha-\beta) & d= (1+\theta)(\gamma+\rho-\beta) - \theta(2\beta-\gamma) \\
b= (1+\theta)(1+\alpha-\beta) - \theta)(\alpha+2\beta-1) & e = (1+\theta)(2\beta-\gamma) - \theta(\gamma+\rho+\beta) \\
c = (1+2\theta)(1-\gamma)
\end{matrix}
$$
We assume $\beta<1$  so $a,b>0$ when $\theta=0$, and there will an attracting fixed point (a.f.p.) on the
$S,D$ edge for any $\theta\ge 0$, $c>0$. We assume $\gamma<1$ so $I \gg S$. Almost anything is possible for
the signs of $d$ and $e$. When $\theta=0$ we have
$$
\begin{matrix}
d<0, e>0  & \gamma < \beta-\rho  &  I \gg S\\
d>0, e>0 & \beta-\rho < \gamma < 2\beta & I,S \hbox{ a.f.p.} \\
d>0, e<0 & \gamma > 2\beta & S \gg I
\end{matrix}
$$
There are a large number of possible cases, so we only consider the ones that lead to coexistence.

\mn
{\bf Case 1.} If the $S,D$ and $D,I$ fixed points exist and can be invaded then we are in the situation 
of Example 7.2 and there will be coexistence in the spatial model.

\mn
{\bf Case 2.} If $D \gg I$ and I can invade the $S,D$ fixed point, then we are in the situation of
Example 7.3 and there will be coexistence in the spatial model.

\begin{center}
\begin{picture}(180,150)
\put(20,120){Case 1.}
\put(30,30){\line(1,0){120}}\put(30,30){\line(2,3){60}}\put(150,30){\line(-2,3){60}}
\put(88,125){$S$}\put(15,25){$D$}\put(155,25){$I$}
\put(59,76){$\bullet$}\put(87,27){$\bullet$}\put(87,55){$\bullet$} 
\put(65,75){\vector(3,-2){15}}\put(90,35){\vector(0,1){15}}
\put(115,95){\vector(2,-3){20}} 
\put(80,115){\vector(-2,-3){15}}\put(35,48){\vector(2,3){15}} 
\put(45,25){\vector(1,0){30}}\put(135,25){\vector(-1,0){30}}
\end{picture}
\qquad
\begin{picture}(180,150)
\put(20,120){Case 2.}
\put(30,30){\line(1,0){120}}\put(30,30){\line(2,3){60}}\put(150,30){\line(-2,3){60}}
\put(88,125){$S$}\put(15,25){$D$}\put(155,25){$I$}
\put(59,76){$\bullet$}\put(87,55){$\bullet$}
\put(65,75){\vector(3,-2){15}} 
\put(35,48){\vector(2,3){15}}\put(80,115){\vector(-2,-3){15}} 
\put(105,25){\vector(-1,0){30}} 
\put(115,95){\vector(2,-3){20}} 
\end{picture}
\end{center}

\mn
{\bf Case 3.} If $\theta>0$ we can have $S \gg D$. If the $D,I$ equilibrium exists and can be invaded
then we are in the situation of Example 7.3 and there will be coexistence in the spatial model.

\section{Voter model duality: details} \label{sec:vmd}

In the degenerate case of an evolutionary game in which $G_{i,j} \equiv 1$, the system reduces to the voter model which was introduced in the mid 1970s independently by Clifford and Sudbury \cite{CS} and Holley and Liggett \cite{HL} on the $d$-dimensional integer lattice. It is a very simple model for the spread of an opinion and has been investigated in great detail, see Liggett \cite{L99} for a survey. To define the model we let $p(y)$ be an irreducible probability kernel $p$ on $\ZZ^d$ that is finite range, symmetric $p(y)=p(-y)$, and has covariance matrix $\sigma^2I$. 

In the voter model, each site $x$ has an opinion $\xi_t(x)$ and at the times of a rate 1 Poisson process decides to change its opinion, imitating the voter at $x+y$ with probability $p(y)$. To study the voter model, it is useful to give an explicit construction called the graphical representation, see Harris \cite{H76} and Griffeath \cite{G78}. For each $x \in \ZZ^d$ and $y$ with $p(y)>0$ let $T^{x,y}_n$, $n\ge 1$ be the arrival times of a Poisson process with rate $p(y)$. At the times $T^{x,y}_n$, $n \ge 1$, $x$ decides to change its opinion to match the one at $x+y$.  To indicate this, we draw an arrow from $(x,T^{x,y}_n)$ to $(x+y,T^{x,y}_n)$.

\begin{center}
\begin{picture}(320,220)
\put(30,30){\line(1,0){260}}
\put(30,180){\line(1,0){260}}
\put(40,30){\line(0,1){150}}
\put(80,30){\line(0,1){150}}
\put(120,30){\line(0,1){150}}
\put(160,30){\line(0,1){150}}
\put(200,30){\line(0,1){150}}
\put(240,30){\line(0,1){150}}
\put(280,30){\line(0,1){150}}
\put(37,18){0}
\put(77,18){0}
\put(117,18){0}
\put(157,18){1}
\put(197,18){0}
\put(237,18){1}
\put(277,18){0}
\put(37,185){1}
\put(77,185){1}
\put(117,185){1}
\put(157,185){1}
\put(197,185){1}
\put(237,185){1}
\put(277,185){0}
\put(20,27){0}
\put(20,177){$t$}
\put(120,160){\line(1,0){40}}
\put(118,157){$<$}
\put(200,145){\line(1,0){40}}
\put(198,142){$<$}
\put(80,130){\line(1,0){40}}
\put(78,127){$<$}
\put(160,110){\line(1,0){40}}
\put(158,107){$<$}
\put(40,90){\line(1,0){40}}
\put(73,87){$>$}
\put(80,75){\line(1,0){40}}
\put(113,72){$>$}
\put(240,60){\line(1,0){40}}
\put(273,57){$>$}
\put(120,45){\line(1,0){40}}
\put(153,42){$>$}
\linethickness{1.0mm}
\put(120,180){\line(0,-1){50}}
\put(80,130){\line(0,-1){55}}
\put(120,75){\line(0,-1){30}}
\put(240,180){\line(0,-1){35}}
\put(200,145){\line(0,-1){35}}
\put(160,110){\line(0,-1){80}}
\end{picture}
\end{center}

To define the dual process $\zeta^{x,t}_s$, which we think of as the motion of a particle, we start with $\zeta^{x,t}_0 = x$. To define the dynamics we start at $x$ at time $t$ and work down the graphical representation. The process stays at $x$ until the first time $t-r$ that it encounters the tail of an arrow $x$. At this time, the particle jumps to the site $x+y$ at the head of the arrow, i.e., $\zeta^{x,t}_r = x+y$. The particle stays at $x+y$ until the next time the tail of an arrow is encountered and then jumps to the head of the arrow etc. The definition of the dual guarantees that if $\xi_t(x)$ is the opinion of $x$ at time $t$ in the voter model then we have the result called \eqref{dualeq} in the introduction 
$$
\xi_t(x) = \xi_{t-s}(\zeta^{x,t}_s)
$$
For fixed $x$ and $t$, $\zeta^{x,t}_s$ is a random walk that jumps at rate 1 and by an amount with distribution $p(y)$. It should be clear from the definition that if $\zeta^{x,t}_s = \zeta^{y,t}_s$ for some $s$ then the two random walks will stay together at later times. For these reasons the $\zeta^{x,t}_s$ are called coalescing random walks. 

Having a dual random walk starting from $x$ for each $t$ is useful for the derivation of \eqref{dualeq}. However for some computations it is more convenient to have a single set valued dual process $\zeta^B_t$ that starts with a particle at each point of $B$, particles move as independent random walks until they hit, at which time they coalesce to 1. Extending the reasoning that lead to \eqref{dualeq}
\beq
P( \xi_t(x) = 1 \hbox{ for all $x\in B$}) = P( \xi_0(y) = 1  \hbox{ for all $y\in \zeta^B_t$})
\label{dual2}
\eeq

Holley and Liggett (1975) have shown 

\begin{theorem}
In $d\le 2$ the voter model approaches complete consensus, i.e., $P( \xi_t(x) = \xi_t(y) ) \to 1$. In $d \ge 3$ if we start from product measure with density $u$ (i.e., we assign opinions 1 and 0 independently to sites with probabilities $u$ and $1-u$) then as $t\to\infty$, $\xi^u_t$ converges in distribution to $\xi^u_\infty$, a stationary distribution, $\nu_u$, in which a fraction $u$ of the sites have opinion 1. 
\end{theorem}

\begin{proof}
We will give the proof here since it is short and we think it will be instructive for readers not familiar with duality.
Our assumptions imply that $p(y)$ has mean 0 and finite variance. Well known results for random walk imply that in $d \le 2$ the associated random walk is recurrent. This implies that two independent random walks will eventually meet and hence the probability of disagreement between two sites in the voter model
$$
P( \xi_t(x) \neq \xi_t(y)) \le P( \zeta^{x,t}_t \neq \zeta^{y,t}_t ) \to 0
$$
To prove the second result it suffices to show that $\xi^u_t$ converges in distribution to  a limit $\xi^u_\infty$, for then standard results for Markov processes will imply that the distribution of $\xi^u_\infty$ is stationary. To do this, we use \eqref{dual2} to conclude
$$
P( \xi^u_t(x) = 1 \hbox{ for all $x\in B$}) = E\left(u^{|\zeta^B_t|}\right)
$$ 
$t \to |\zeta^B_t|$ is decreasing, and $u^{|\zeta^B_t|} \le 1$, so $P( \xi_t(x) = 1 \hbox{ for all $x\in B$})$ converges to a limit $\phi(B)$ for all finite sets $B$. Since the probabilities $\phi(B)$ determine the distribution
of $\xi^u_\infty$, the proof is complete.
\end{proof}

\section{Proofs of the coalescence identities} \label{sec:coalid}

Let $v_1, v_2, v_3$ be independent and have distribution $p$. In this section we prove
the coalescence identities stated in Section \ref{sec:vmsi} 

\begin{lemma} \label{L2}
$p(0|v_1+v_2) = p(0|v_1) = p(v_1|v_2)$
\end{lemma}

\begin{proof} Let $S_n$ be the discrete time random walk with jumps distributed according to $p$ and let 
$$
h(x) = P_x (S_n \neq 0 \hbox{ for all $n \ge 0$} ) 
$$
be the probability the walk never returns to 0 starting from $x$. By considering what happens on
the first step $h(v_1) = h(v_1+v_2)$. Since $v_1$ and $-v_1$ have the same distribution
$h(v_1+v_2)=h(v_2-v_1)$. The two results follow since $p(x|y) = h(y-x)$. 
\end{proof}

\begin{lemma} \label{L3}
$p(v_1|v_2+v_3) = ( 1 + 1/\kappa) p(0|v_1)$
\end{lemma}

\begin{proof} Starting at $S_0=-v_1$, $S_1 = v_2-v_1$ and $S_2 = v_2 + v_3 - v_1$. Since $P_{-v_1}(S_1=0) = 1/\kappa$.
\begin{align*}
h(v_2+v_3 - v_1) - h(v_1) & = P_{-v_1}( S_1 = 0, S_n \neq 0 \hbox{ for $n \ge 2$}) \\
& = (1/\kappa) P_0(  S_n \neq 0 \hbox{ for $n \ge 1$}) = (1/\kappa) h(v_3) = (1/\kappa) p(0|v_1)
\end{align*}
which proves the desired result   
\end{proof}

From the two particle identities we easily get some for three particles.
The starting point is to note that considering the possibilities for
$y$ when $x$ and $z$ don't coalesce we have a result we earlier called \eqref{3to2}
$$
p(x|z)  = p(x|y|z) + p(x,y|z) +p(x|y,z) 
$$
Combining this identity with one for another pair that shares a site
in common leads to identities that relate the three ways three particles 
can coalesce to give two.

\begin{lemma}
$p(0,v_1|v_1+v_2) = p(0,v_1+v_2|v_1) =  p(v_1,v_1+v_2|0)$ and hence
$$
p(0|v_1) = 2p(x,y|z) + p(0|v_1|v_1+v_2)
$$
where $x,y,z$ is any ordering of $0, v_1, v_1+v_2$.
\end{lemma}

\begin{proof}
Using (\ref{3to2})
\begin{align*}
p(0|v_1) & = p(0|v_1|v_1+v_2) + p(v_1, v_1+v_2|0) + p(0,v_1+v_2|v_1) \\
p(0|v_1+v_2) & = p(0|v_1|v_1+v_2) + p(v_1,v_1+v_2|0) +p(0,v_1|v_1+v_2)  \\
p(v_1|v_1+v_2) & = p(0|v_1|v_1+v_2)  + p(0,v_1+v_2|v_1) + p(0,v_1|v_1+v_2)
\end{align*}
Since $p(0|v_1+v_2) = p(0|v_1)$ by Lemma \ref{Ted2} the first result follows.
Translation invariance implies $p(0|v_1) = p(0|v_2) = p(v_1|v_1+v_2)$ so comparing the 
first and third lines gives the second result. 
The displayed identity follows from the first two and first and third in the proof.
\end{proof}

\begin{lemma} 
$ p(v_1,v_2|v_2+v_3) = p(v_2,v_2+v_3|v_1) = p(v_1,v_2+v_3|v_2) + (1/\kappa)p(0|v_1)$ and hence
\begin{align*}
p(v_1|v_2) (1+1/\kappa) & = 2 p(v_2,v_2+v_3|v_1) + p(v_1|v_2|v_2+v_3) \\
&= 2 p(v_1,v_2|v_2+v_3) + p(v_1|v_2|v_2+v_3) \\
p(v_1|v_2) (1-1/\kappa) & = 2 p(v_1,v_2+v_3|v_2) + p(v_1|v_2|v_2+v_3)
\end{align*}
\end{lemma}

\begin{proof}
Using (\ref{3to2})
\begin{align*}
p(v_1|v_2)     & =  p(v_1|v_2|v_2+v_3) + p(v_1, v_2+v_3 | v_2) + p(v_2, v_2 + v_3|v_1) \\
p(v_2|v_2+v_3) & =  p(v_1|v_2|v_2+v_3) + p(v_1, v_2+v_3 | v_2) + p(v_1, v_2 |v_2+v_3)  \\
p(v_1|v_2+v_3) & =  p(v_1|v_2|v_2+v_3) + p(v_2, v_2+v_3 | v_1) + p(v_1, v_2 |v_2+v_3) 
\end{align*}
Since $p(v_2|v_2+v_3) = p(v_1|v_2)$ by Lemma \ref{L2}, the first result follows. 
Noting that Lemma \ref{L3} implies $p(v_1|v_2+v_3) = p(0|v_1)(1 + 1/\kappa)$ and subtracting 
the second equation in the proof from the third gives
$$
(1/\kappa) p(0|v_1) = p(v_2,v_2+v_3|v_1) - p(v_1, v_2+v_3|v_2)
$$
and the second result follows. To get the first displayed equation in the lemma,
substitute $p(v_1,v_2+v_3|v_2)  = p(v_2, v_2+v_3|v_1) - (1/\kappa)p(0|v_1)$ in the first equation in the proof.
Since $ p(v_1,v_2|v_2+v_3) = p(v_2,v_2+v_3|v_1)$ the second follows.
For the third one, use $p(v_2, v_2+v_3|v_1) = p(v_1,v_2+v_3|v_2) + (1/\kappa)p(0|v_1)$ in the first equation in the proof.
\end{proof}

\section{Derivation of the limiting PDE} \label{sec:derPDE}

In this section we will compute the functions $\phi_B^i(u)$ and $\phi_D^i(u)$ that
appear in the limiting PDE. To do this it is useful to note that since $\sum_j u_i=1$ we can write the
replicator equation \eqref{rep} as
\beq
\frac{du_i}{dt} = \sum_{j \neq i} \sum_k u_iu_ju_k (G_{i,k} - G_{j,k})
\label{repalt}
\eeq
where the restriction to $j \neq i$ comes from noting that the when $i=j$ the $G$'s cancel.

\subsection{Birth-Death dynamics}

By \eqref{hBD} the perturbation is
$$
h_{i,j}(0,\xi) = \sum_{k} f^{(2)}_{j,k}(0,\xi) G_{j,k}
$$
where $f^{(2)}$ was defined in Section 3.1.
In the multi-strategy case the rate of change of the frequency of strategy $i$ in the voter model
equilibrium is
\begin{align}
\phi^i_B(u) & = \left\langle \sum_{j\neq i}  - 1(\xi(0) = i) h_{i,j}(0,\xi) + 1(\xi(0)=j) h_{j,i}(0,\xi) \right\rangle_{\sqz u} 
\nonumber\\
 & = \sum_{j\neq i}\sum_k  - q(i,j,k) G_{j,k} + q(j,i,k) G_{i,k} 
\label{phiBq}
\end{align}
where $q(a,b,c) = P( \xi(0) = a, \xi(v_1) = b, \xi(v_1+v_2)=c )$ and $v_1$, $v_2$ are independent 
random choices from ${\cal N}$. If $a \neq b$ then
\begin{align}
q(a,b,c) & = p(0|v_1|v_1+v_2) u_a u_b u_c 
\label{evalq}\\
 & + 1_{(c=a)} p(0,v_1+v_2|v_1) u_a u_b + 1_{(c=b)} p(0|v_1,v_1+v_2) u_a u_b \nonumber
\end{align}
Combining the last two equations, and comparing with \eqref{repalt} we see that
\begin{align*}
\phi^i_B(u) &= \phi^i_R(u) p(0|v_1|v_1+v_2) + \sum_{j\neq i} p(0,v_1+v_2|v_1) u_i u_j (-G_{j,i}+ G_{i,j}) \\
& +  \sum_{j \neq i} p(0|v_1,v_1+v_2) u_i u_j (-G_{j,j} + G_{i,i})
\end{align*}
Formula \eqref{BDswitch} implies that $p(0,v_1+v_2|v_1)=p(0|v_1,v_1+v_2)$ so the last two terms can be combined into one.
\beq
\phi^i_B(u) = \phi^i_R(u) p(0|v_1|v_1+v_2) + \sum_{j\neq i} p(0|v_1,v_1+v_2) u_i u_j (G_{i,i}-G_{j,i}+ G_{i,j}-G_{j,j}) 
\label{phiBn}
\eeq
which is the formula given in \eqref{phiBi}

\subsection{Death-Birth dynamics}

In this case \eqref{hDB} implies that the perturbation is 
$$
\bar h_{i,j}^\ep(0,\xi) = h_{i,j}(0,\xi) - f_j \sum_k h_{i,k}(0,\xi) +O(\ep^2)
$$
From this we see that $\phi^i_D = \phi^i_B + \psi_D^i$ where
\begin{align*}
\psi_D^i =  & \left\langle \sum_{j \neq i} 1(\xi(0)=i) f_j(0,\xi) \sum_k h_{i,k} (0,\xi) 
 -  \sum_{j \neq i} 1(\xi(0)=j) f_i(0,\xi) \sum_k h_{j,k}(0,\xi) \right\rangle_{\sqz u}.
\end{align*}
Let $v_1$, $v_2$ and $v_3$ be independent random picks from ${\cal N}$ and let
$$
q(a,b,c,d) = P( \xi(v_1) = a, \xi(0) = b, \xi(v_2) = c, \xi(v_2+v_3) = d )
$$
we can write the new term in $\phi^i_D$ as
\beq
\psi^i_D  = \sum_{j \neq i} \sum_{k, \ell } [ q(j,i,k,\ell) - q(i,j,k,\ell) ] G_{k,\ell}
\label{newDBt}
\eeq
Adding and subtracting $q(i,i,k,\ell)$ we can do the sum over $j$ to get
\beq
\psi^i_D = \sum_{k,\ell} [q(i,k,\ell) - q(i,\cdot, k,\ell) ] G_{k,\ell} 
\label{4to3}
\eeq
where $q(a,b,c)$ is as defined above and
$$
q(a,\cdot,b,c) =  P( \xi(v_1) = a, \xi(v_2) = b, \xi(v_2+v_3) = c )
$$
If we let $q(b,c) = P(\xi(v_2) = b, \xi(v_2+v_3) = c )$ and write the sum in \eqref{4to3} as $\sum_{k\neq i}\sum_\ell$ plus
\begin{align*}
 \sum_\ell [q(i,i,\ell) - q(i,\cdot,i,\ell)] G(i,\ell) & = [q(i,\ell) - q(i,\ell)] G_{i,\ell} \\
& + \sum_{k \neq i}\sum_\ell [-q(k,i,\ell) + q(k,\cdot,i,\ell)] G_{i,\ell}
\end{align*}
To see this move the second sum on the right to the left.

Putting things together we have
$$
\psi^i_D  = \sum_{k\neq i}\sum_\ell [q(i,k,\ell) - q(i,\cdot, k,\ell) ] G_{k,\ell} 
 + \sum_{k\neq i}\sum_\ell [-q(k,i,\ell) + q(k,\cdot, i,\ell) ] G_{k,\ell}
$$
One half of the sum is
$$ 
\sum_{k,\ell} q(i,k,\ell) G_{k,\ell} -q(i,k,\ell) G_{k,\ell} = - \phi^i_B
$$
Since $\phi^i_D = \phi^i_B + \psi^i_D$ we have
\beq
\phi^i_D = \sum_{k,\ell} - q(i,\cdot, k,\ell) G(k,\ell) + q(i,\cdot, k,\ell) G_{k,\ell}
\label{phiDq}
\eeq
Note the similarity to (\ref{phiBq}). If $a \neq b$ then
\begin{align}
q(a,\cdot,b,c) & = p(v_1|v_2|v_2+v_3) u_a u_b u_c 
\label{evalqdot}\\
 & + 1_{(c=a)} p(v_1,v_2+v_3|v_2) u_a u_b + 1_{(c=b)} p(v_1|v_2,v_2+v_3) u_a u_b \nonumber
\end{align}
As in the Birth-Death case it follows that
\begin{align*}
\phi^i_D(u) &= \phi^i_R(u) p(v_1|v_2|v_2+v_3) + \sum_{j\neq i} p(v_1,v_2+v_3|v_2) u_i u_j (-G_{j,i}+ G_{i,j}) \\
& +  \sum_{j \neq i} p(v_1|v_2,v_2+v_3) u_i u_j (-G_{j,j} + G_{i,i})
\end{align*}
By \eqref{DBswitch} $p(v_1,v_2+v_3|v_2) = p(v_1|v_2,v_2+v_3) - (1/\kappa)p(v_1|v_2)$, so we can rewrite this as
\begin{align}
\phi^i_D(u) &= \phi^i_R(u) p(v_1|v_2|v_2+v_3) + \sum_{j\neq i} p(v_1|v_2,v_2+v_3) u_i u_j (G_{i,i} -G_{j,i}+ G_{i,j} -G_{j,j}) 
\nonumber\\
& -(1/\kappa) p(v_1|v_2)  \sum_{j \neq i}  u_i u_j (-G_{j,i} + G_{i,j})
\label{phiDn}
\end{align}

\section{Two strategy games with Death-Birth updating} \label{sec:2sDB}

\subsection{Computation of the phase diagram given in Figure 3}

We give the details behind the conclusions drawn in Figure 2. In the Death-Birth case 
$$
\theta = \frac{\bar p_2}{\bar p_1}(R+S-T+P) - \frac{p(v_1|v_2)}{\kappa \bar p_1}(S-T)
$$
To find the boundaries between the four cases using \eqref{4cases}, we  
let $\mu = \bar p_2/\bar p_1 \in (0,1)$ and $\nu = p(v_1|v_2)/\kappa \bar p_1$. We have $\alpha=\gamma$ when 
\begin{align*}
R-T = - \theta & = - \mu (R+S-T-P) + \nu(S-T) \\
& = -(\mu-\nu) (R+S-T-P) - \nu(R-P)
\end{align*}
Clearly $\nu>0$. The fact that $\mu-\nu>0$ follows from \eqref{cpos}.
Rearranging gives 
$$
(\mu-\nu)(S-P) + \nu(R-P) = (1+\mu-\nu)(T-R)
$$
and hence
\beq
T-R = \frac{\mu-\nu}{1+\mu-\nu}(S-P) +  \frac{\nu}{1+\mu-\nu}(R-P)
\label{line1}
\eeq
Repeating the last calculation shows that $\beta=\delta$ when 
$$
S-P = -\theta  = -(\mu-\nu) (R+S-T-P) - \nu(R-P) < 0
$$
Rearranging gives $(1+\mu-\nu)(S-P) +\nu(R-P) = (\mu-\nu)(T-R)$, and we have
\beq
T-R = \frac{1+\mu-\nu}{\mu-\nu}(S-P) +  \frac{\nu}{\mu-\nu}(R-P)
\label{line2}
\eeq

To find the intersections of the lines in \eqref{line1} and \eqref{line2}, we set 
$$
\frac{\mu-\nu}{1+\mu-\nu}(S-P) +  \frac{\nu}{1+\mu-\nu}(R-P)
= \frac{1+\mu-\nu}{\mu-\nu}(S-P) +  \frac{\nu}{\mu-\nu}(R-P)
$$
which holds if
$$
(\mu-\nu)^2 (S-P) + \nu(\mu-\nu)(R-P) = (1+\mu-\nu)^2(S-P) + \nu(1+\mu-\nu)(R-P)
$$
Solving gives
$$
S - P = \frac{-\nu (R-P)}{1+2(\mu-\nu)} < 0
$$
Using \eqref{line2} we see that at this value of $S$
$$
T-R = - \frac{\nu(R-P)}{\mu-\nu} \left[ \frac{1+\mu-\nu}{1+2(\mu-\nu)} - 1 \right] 
= \frac{\nu(R-P)}{1+2(\mu-\nu)} > 0
$$

To simplify the formulas for \eqref{line1} and \eqref{line2} let
\beq
P^* =  P - \frac{\nu (R-P)}{1+2(\mu-\nu)} \qquad  R^* = R + \frac{\nu(R-P)}{1+2(\mu-\nu)} 
\label{stardef}
\eeq
From this and \eqref{line1} we see that
\begin{align*}
T-R^* & = T - R - \frac{\nu(R-P)}{1+2(\mu-\nu)} \\
& = \frac{\mu-\nu}{1+\mu-\nu}(S-P) +  \frac{\nu}{1+\mu-\nu}(R-P) - \frac{\nu(R-P)}{1+2(\mu-\nu)} \\
& = \frac{\mu-\nu}{1+\mu-\nu}(S-P) - \frac{\mu - \nu}{1+\mu-\nu} \cdot \frac{\nu(R-P)}{1+2(\mu-\nu)}
= \frac{\mu-\nu}{1+\mu-\nu}(S-P^*)
\end{align*}
A similar calculation using \eqref{line2} shows that
\begin{align*}
T-R^* & = T - R - \frac{\nu(R-P)}{1+2(\mu-\nu)} \\
& = \frac{1+\mu-\nu}{\mu-\nu}(S-P) +  \frac{\nu}{\mu-\nu}(R-P) - \frac{\nu(R-P)}{1+2(\mu-\nu)} \\
& = \frac{1+\mu-\nu}{\mu-\nu}(S-P) + \frac{1+\mu-\nu}{\mu-\nu}\cdot \frac{\nu(R-P)}{1+2(\mu-\nu)}
= \frac{1+\mu-\nu}{\mu-\nu}(S-P^*)
\end{align*}
so the two lines can be written as 
$$
T-R^* = \frac{\mu-\nu}{1+\mu-\nu}(S-P^*) \qquad T-R^* = \frac{1+\mu-\nu}{\mu-\nu}(S-P^*)
$$
This leads to the four regions drawn in Figure 2.
In the lower left region where there is bistability, 1's win if $\bar u > 1/2$
or what is the same if strategy 1 is better than strategy 2 when $u=1/2$ or
$$
R + S - T - P > -2\theta
$$
Plugging in the value of $\theta$ this becomes
\beq
(1+2(\mu-\nu))(R+S-T-P) >  -2\nu(R-P)
\label{1swinDB}
\eeq
Dividing each side by $1+2(\mu-\nu)$ and recalling \eqref{stardef} this can be written as
$R^*-T > P^*-S$. For the proof of Theorem \ref{TF2} it is useful to write this as
\beq
\frac{1+2\mu}{1+2(\mu-\nu)} R + S > T + \frac{1+2\mu}{1+2(\mu-\nu)}
\label{keyeq}
\eeq
In the coexistence region, the equilibrium is
\beq
\bar u = \frac{S-P+\theta}{S-P+T-R}
\label{neweqDB}
\eeq
Plugging in the value of $\theta$, we see that $\bar u$ is constant on lines through $(S,T)=(R^*,P^*)$.
Again for the proof of Theorem \ref{TF2} a less compact form is more desirable.
Recalling the definition of $\theta$ we see that $\bar u  > 1/2$ when
$$
2(S-P) + 2\mu (R+S-T-P) - 2\nu(S-T) > S-P+T-R
$$
which rearranges to become
$$
(1+2\mu)R + [1+2(\mu-\nu)] S > [1+2(\mu-\nu)] T  + (1+2\mu)P
$$
Dividing by $1+2(\mu-\nu)$ this becomes \eqref{keyeq}. To prove Theorem \ref{TF2} now,
we note that
\begin{align*}
1 + 2\mu & = \frac{\bar p_1 + 2 \bar p_2}{\bar p_1} = (1+1/\kappa)\frac{p(v_1|v_2)}{\bar p_1} \\
1 + 2(\mu-\nu) & = (1+1/\kappa)\frac{p(v_1|v_2)}{\bar p_1} - 2 \frac{p(v_1|v_2)}{\kappa \bar p_1}
\end{align*}
so we have
$$
\frac{1+2\mu}{1+2(\mu-\nu)} = \frac{\kappa+1}{\kappa-1}
$$

\subsection{Proof of Tarnita's formula}

\begin{lemma} \label{phihalf}
Suppose the $\phi$ in the limiting PDE is cubic. The limiting frequency of 1's in the PDE is $> 1/2$ 
(and hence also in the particle system with small $\ep$) if and only if $\phi(1/2) >0$.  
\end{lemma}

\begin{proof}
As explained in the proof of Theorem \ref{CDPG} given in \cite{CDP} the behavior of the PDE is related to the shape of $\phi$
in the following way.

\mn
$\phi'(0)>0$, $\phi'(1)>0$. There is an interior fixed point $\bar u$, which is the limiting frequency.
If $\phi(1/2)>0$ then $\bar u > 1/2$.

\mn
$\phi'(0)>0$, $\phi'(1)<0$. $\phi>0$ on $(0,1)$ so the system converges to 1.

\mn
$\phi'(0)<0$, $\phi'(1)>0$. $\phi<0$ on $(0,1)$ so the system converges to 0.

\mn
$\phi'(0)<0$, $\phi'(1)<0$. There is an unstable fixed point $\bar u$. If $\bar u < 1/2$ 
which in this case is equivalent to $\phi(1/2)>0$ the system 
converges to 1, otherwise it converges to 0.

\mn
Combining our observations proves the desired result. \end{proof}

\mn
{\bf Death-Birth updating.} In this case, we have $\phi(1/2)>0$ if 
$$
a + b + \theta > c + d - \theta
$$
where $\theta = (p_2/p_1)(a+b-c-d) - p(v_1|v_2)/\kappa p_1$.  Rearranging we see that the last inequality holds if
$$
\kappa(\bar p_1 +2\bar p_2) (a+b-c-d) - 2p(v_1|v_2) (b-c)
$$
Using \eqref{DB1} $\bar p_1 + 2\bar p_2 = p(v_1|v_2)( 1 + 1/\kappa )$ this becomes
$$
p(v_1|v_2)[ (\kappa + 1)(a-d) + (\kappa-1)(b-c)
$$
which is the condition in Tarnita's formula with $\sigma=(\kappa+1)/(\kappa-1)$.

\section{Equilibria for three strategy games} \label{sec:alg3}

Recall that we can write the game matrix as
$$
\begin{pmatrix} 0 & \alpha_3 & \beta_2 \\ \beta_3 & 0 & \alpha_1 \\ \alpha_2 & \beta_1 & 0 \end{pmatrix}
$$
In equilibrium all three strategies must have the same payoff so an interior equilibrium $(x,y,1-x-y)$ must satisfy
$$
\alpha_3 y + \beta_2(1-x-y) =  \beta_3x + \alpha_1(1-x-y) =  \alpha_2x + \beta_1y 
$$
Equating the first to the third, the second to the third and rearranging we obtain
\begin{align*}
(\alpha_2+\beta_2) x + (\beta_1 + \beta_2 - \alpha_3) y & = \beta_2  \\
(\alpha_1+ \alpha_2 - \beta_3) x + (\alpha_1+\beta_1) y & = \alpha_1 
\end{align*}
From this we see that [to find $x$ look at $(\alpha_1+\beta_1)E_1-(\beta_1 - \beta_2-\alpha_3)E_2$]
\begin{align*}
x & = \frac{ \beta_2(\alpha_1+\beta_1) - \alpha_1(\beta_1+\beta_2-\alpha_3)   } 
{(\alpha_1+\beta_1)(\alpha_2+\beta_2) - (\alpha_1 + \alpha_2-\beta_3)(\beta_1 +\beta_2 - \alpha_3)  } \\
y & = \frac{ \alpha_1(\alpha_2+\beta_2) - \beta_2(\alpha_1 + \alpha_2 -\beta_3) -  }
{(\alpha_1+\beta_1)(\alpha_2+\beta_2) - (\alpha_1 + \alpha_2-\beta_3)(\beta_1 +\beta_2 - \alpha_3)  }
\end{align*} 
The 13 terms in the denominator reduce to nine:
$$
D  =  \alpha_1\alpha_2 + \alpha_1 \alpha_3 + \alpha_2 \alpha_3 + \beta_1\beta_2 + \beta_1\beta_3 + \beta_2 \beta_3
 - \alpha_1\beta_1 - \alpha_2\beta_2 - \alpha_3\beta_3
$$
The numerators for $x$ and $y$ reduce to 3 terms each, so the equilibrium must be:
\begin{align*}
\bar u_1 & = \frac{  \alpha_1\alpha_3 + \beta_1\beta_2 - \alpha_1\beta_1}{D} 
\\
\bar u_2 & = \frac{  \alpha_2\alpha_1 + \beta_2\beta_3 - \alpha_2\beta_2}{D} 
\\
\bar u_3 & = \frac{  \alpha_3\alpha_2 + \beta_3\beta_1  - \alpha_3\beta_3 }{D}.
\end{align*}
which is the formula given in \eqref{Geq}.





\ACKNO{We thank Ed Perkins and a second anonymous referee for suggestions
that improved the correctness and readability of the paper.}


\end{document}